\documentclass[a4paper]{article}

\usepackage{graphicx,amssymb}
\usepackage{amsmath}
\usepackage{cases}
\usepackage{float}
\usepackage{listings}
\usepackage{times}
\usepackage{tikz}
\usepackage{epic}
\usepackage{titlesec}
\usepackage{geometry}
\usepackage{amsthm}
\usepackage{bm}
\usepackage{stmaryrd}
\usepackage{subfigure}
\usepackage{indentfirst}
\usepackage{xcolor}
\usepackage{epstopdf}
\usepackage{enumerate}
\usepackage{enumitem}
\usepackage{hyperref}

\definecolor{red}{rgb}{1.00,0.00,0.00}
{\numberwithin{equation}{section}
\setlength{\parindent}{1em}

\newtheorem{theorem}{Theorem}[section]
\newtheorem{lemma}{Lemma}[section]
\newtheorem{remark}{Remark}[section]

\newtheorem{example}{Example}[section]
\newtheorem{assumption}{Assumption}[section]
\newtheorem{corollary}{Corollary}[section]

\SetSymbolFont{stmry}{bold}{U}{stmry}{m}{n}

\newcommand{\normmm}[1]{{\left\vert\kern-0.25ex\left\vert
\kern-0.25ex\left\vert #1
    \right\vert\kern-0.25ex\right\vert\kern-0.25ex\right\vert}}

\geometry{left=3cm,right=3cm,top=4cm,bottom=2.5cm}

\newcommand{\Red}[1]{\textcolor{black}{#1}}

\begin{document}           % End of preamble and beginning of text.

\title{Staggered DG method for coupling of the Stokes and Darcy-Forchheimer problems}
\author{ Lina Zhao\footnotemark[1]\qquad
\;Eric Chung\footnotemark[2]\qquad
\;Eun-Jae Park\footnotemark[3]\qquad
\;Guanyu Zhou\footnotemark[4]\qquad}
\renewcommand{\thefootnote}{\fnsymbol{footnote}}
\footnotetext[1]{Department of Mathematics, The Chinese University of Hong Kong, Hong Kong Special Administrative Region. ({lzhao@math.cuhk.edu.hk})}
\footnotetext[2]{Department of Mathematics, The Chinese University of Hong Kong, Hong Kong Special Administrative Region. ({tschung@math.cuhk.edu.hk})}
\footnotetext[3]{Department of Computational Science and Engineering, Yonsei University, Seoul 03722, Republic of Korea. ({ejpark@yonsei.ac.kr})}
\footnotetext[4]{Department of Applied Mathematics, The Tokyo University of Science, 1-3 Kagurazaka, Shinjuku-Ku, Tokyo, 162-8601 Japan. ({zhoug@rs.tus.ac.jp})}
\maketitle

\textbf{Abstract:}
In this paper we develop a staggered discontinuous Galerkin method for the Stokes and Darcy-Forchheimer problems coupled with the \Red{Beavers-Joseph-Saffman} conditions. The method is defined by imposing staggered continuity for all the variables involved and the interface conditions are enforced by switching the roles of the variables met on the interface, which eliminate the hassle of introducing additional variables. This method can be flexibly applied to rough grids such as the highly distorted grids and the polygonal grids. In addition, the method allows nonmatching grids on the interface thanks to the special inclusion of the interface conditions, which is highly appreciated from a practical point of view. A new discrete trace inequality and a generalized Poincar\'{e}-Friedrichs inequality are proved, which enables us to prove the optimal convergence estimates under reasonable regularity assumptions. Finally, several numerical experiments are given to illustrate the performances of the proposed method, and the numerical results indicate that the proposed method is accurate and efficient, in addition, it is a good candidate for practical applications.

\vskip 0.2cm \textbf{Key words:} Staggered DG method, The Stokes equations, Darcy-Forchheimer equations, Highly distorted grids, \Red{Beavers-Joseph-Saffman} interface conditions, Nonmatching grids

 \pagestyle{myheadings} \thispagestyle{plain} \markboth{Zhao}
    {SDG method for coupled Stokes and Darcy-Forchheimer problems}

\section{Introduction}

%The grids do not need to be matching across the interface.
%
%This is the result that considers coupling  on general quadrilateral and polygonal meshes.
%
%Very few results are available for coupling, in addition, the optimal convergence is even less.
%
%Recently, developing robust numerical schemes on general meshes are in growing demand.
%
%Developing numerical schemes on general meshes is still in its infancy, and our work will definitively inspire more works in this direction.

The fluid flow between porous media and free-flow zones have extensive applications in hydrology, environment science and biofluid dynamics, which motivates a lot of researchers to derive suitable mathematical and numerical models for the fluid movement. The system can be viewed as a coupled problem with two physical systems interacting across an interface. The simplest mathematical formulation for the coupled problem is the coupled Stokes and Darcy flows problems, which imposes the Stokes equations in the Stokes region and the Darcy's law in the porous media region, coupled with proper interface conditions. The most suitable and popular interface conditions are called Beavers-Joseph-Saffman conditions, which are proposed by Saffman \cite{Saffman71}. The coupled model has received great attention and a large number of works have been devoted to the coupled Stokes and Darcy flows problem, see, for example \cite{Discacciati02,Layton03,Riviere052,Riviere05,Burman07,Gatica11,Girault14,Lipnikov14}. However, Darcy's law only provides linear relationship between the gradient of pressure and velocity, which usually fails for complex physical problems. Forchheimer \cite{Forchheimer1901} conducted flow experiments in sandpacks and recognized that for moderate Reynolds numbers (Re $>0.1$ approximately), Darcy's law is not adequate. He found that the pressure gradient and Darcy velocity should satisfy nonlinear relationship and the Darcy-Forchheimer law accounts for this nonlinear behavior. Mixed methods for the steady Darcy-Forchheimer flow was studied in more general setting in \cite{Park95} and unsteady case was analyzed in \cite{KimPark99,Park05}. A nonconforming primal mixed method for the Darcy-Forchheimer was introduced
in \cite{Girault08} and a block-centered finite difference method was considered in \cite{RuiPan12,RuiLiu15}.

%The coupled problem has many practical applications.  A large amount of works have been devoted to the coupled Stokes and Darcy flows problem, see, for example \cite{}. In practice, in porus media, the Darcy law can not . Instead, the Darcy-Forchheimer model has to be employed.

Developing robust numerical schemes on general meshes has drawn great attention in recent years, and numerous numerical schemes have been developed on general meshes. Among all the methods, we mention in particular the virtual element (VEM) method, the hybrid high order (HHO) method and the weak Galerkin (WG) method \cite{Beir13,PietroErn14,WangYe14,BurmanErn18}. However, to the best of our knowledge, none of the existing numerical schemes applicable on general meshes have been exploited to simulate the coupling of the Stokes and Darcy-Forchheimer problems. In addition, very few results are available for coupling of the Stokes and Darcy-Forchheimer problems even on triangular meshes, and it is not an easy task to design numerical schemes for this coupled model. In \cite{ZhangRui16} a stabilized Crouzeix-Raviart element method is developed and optimal order error estimates are achieved. In \cite{Almonacid19} a fully mixed method is proposed, whereas the authors can only obtain a suboptimal convergence rate. Staggered discontinuous Gaerkin (DG) method is initially developed by Chung and Engquist \cite{ChungWave06,ChungWave09}, since then it has been successfully applied to a wide range of partial differential equations arising from practical applications, see, e.g., \cite{ChungEngquist,ChungCiarletYu13,KimChungLee,KimChungLam16,ChungParkLina,
LinaParkStoke-tri,LinaParkconvection}. Recently, staggered DG method has been successfully designed for Darcy law and the Stokes equations, respectively, on general quadrilateral and polygonal meshes \cite{LinaPark,LinaParkShin}, and the numerical results proposed therein indicate that it is accurate and robust for rough grids. Staggered DG method has many distinctive features such as local and global conservations, optimal convergence estimates and robustness to mesh distortion. Furthermore, the method favors adaptive mesh refinement by allowing hanging nodes. Thus, staggered DG method is a good candidate for problems under real applications, and it will be beneficial if one can apply this method to physical problems. Therefore, the goal of this paper is to initiate a study on staggered DG method for the coupled Stokes and \Red{Darcy-Forchheimer} problems that can be flexibly applied to general quadrilateral and polygonal meshes, in addition, arbitrary polynomial orders are allowed.

The key idea of staggered DG method is to divide the \Red{initial} partition (general quadrilateral or polygonal meshes) into the union of triangles by connecting the interior point to all the vertices of the initial partition. Then the corresponding basis functions for the associated variables with staggered continuity are defined on the resulting triangulations. The interface conditions are imposed by switching the roles of the variables met on the interface. By doing so, no additional variables are introduced on the interface, which is different from the method proposed in \cite{Almonacid19}. The resulting method allows nonmatching grids on the interface due to the special inclusion of the interface conditions.
%Furthermore, the proposed method is robust to mesh distortion and allows hanging nodes, which are desirable features in practical applications. We note that the treatment of hanging nodes for our method is particularly simple, and one just needs to incorporate the hanging nodes in the construction of the method, which favors adaptive mesh refinement.
In addition to the aforementioned features, another distinctive property is that the staggered continuity property naturally gives an interelement flux term in contrast to other DG methods, where the numerical fluxes or penalty parameters have to be carefully chosen. Note that developing stable numerical schemes for the Stokes and Darcy-Forchheimer problems is still in its infancy, and the present method is well suited to general quadrilateral and polygonal meshes, which will definitely inspire more works in this direction.

The primary difficulty for the convergence estimates results from the terms appearing on the interface. To overcome this issue, a novel discrete trace inequality and a generalized Poincar\'{e}-Friedrichs inequality are proposed. Our analysis can be carried out by
exploiting these two inequalities and the monotone property of the nonlinear operator under certain regularity assumptions without any restrictions on the source terms. The main novelties of this paper are as follows: First, this is the first result for coupling of the Stokes and Darcy-Forchheimer problems  which can be flexibly applied to general quadrilateral and polygonal meshes and well adapted to incorporate hanging nodes in the construction of the method. Second, we can prove the optimal convergence estimates by proving some new discrete trace inequality and generalized Poincar\'{e}-Friedrichs inequality.

The rest of the paper is organized as follows. In the next section we derive staggered DG method for the coupled Stokes and Darcy-Forchheimer problems and present some preliminary lemmas. Then in Section~\ref{sec:error} we present the convergence estimates, where the discrete trace inequality and the generalized Poincar\'{e}-Friedrichs inequality are the crucial ingredients.  Several numerical experiments are carried out in Section~\ref{sec:numerical} to validate the theoretical results. Finally, the conclusion is given in Section~\ref{sec:conclusion}.

\section{Description of staggered DG method}
\subsection{Preliminaries}

Let $\Omega_S$ and $\Omega_D$ be two bounded and simply connected polygonal domains in $\mathbb{R}^2$ such that $\partial \Omega_S\cap \partial \Omega_D=\Gamma\neq \emptyset$ and $\Omega_S\cap \Omega_D=\emptyset$. Then, let $\Gamma_S:=\partial \Omega_S\backslash \overline{\Gamma}$, $\Gamma_D:=\partial \Omega_D\backslash \overline{\Gamma}$ and denote by $\bm{n}_S$ the unit normal vector pointing from $\Omega_S$ to $\Omega_D$ and by $\bm{n}_D$ the unit normal vector pointing from $\Omega_D$ to $\Omega_S$, on the interface $\Gamma$ we have $\bm{n}_D=-\bm{n}_S$. In addition, $\bm{t}$ represents the unit tangential vector along the interface $\Gamma$. Figure~\ref{coupled} gives a schematic representation of the geometry.

When kinematic effects surpass viscous effects in a porous medium, the Darcy velocity $\bm{u}_D$ and the pressure gradient $\nabla p_D$ does not satisfy a linear relation. Instead, a nonlinear approximation, known as the Darcy-Forchheimer model, is considered. When it is imposed on the porous medium $\Omega_D$ with homogeneous Dircihlet boundary condition on $\Gamma_D$ the equations read:
\begin{align}
\nabla \cdot \bm{u}_D&=f_D\quad \mbox{in}\;\Omega_D,\label{eq:Darcy1}\\
\frac{\mu}{\rho}K^{-1}\bm{u}_D+\frac{\beta}{\rho}|\bm{u}_D|\bm{u}_D+\nabla p_D&=\bm{g}_D\quad \mbox{in}\;\Omega_D,\label{eq:Darcy2}\\
p_D&=0 \qquad \mbox{on}\; \Gamma_D,
\end{align}
where $K$ is the permeability tensor, assumed to be uniformly positive definite and bounded, $\rho$ is the density of the fluid, $\mu$ is its viscosity and $\beta$ is a dynamic viscosity, all assumed to be positive constants. In addition, $\bm{g}_D$ and $f_D$ are source terms. We remark that in this context we exploit homogeneous Dirichlet boundary condition, in fact, we can also consider homogeneous Neumann boundary condition, i.e., $\bm{u}_D\cdot\bm{n}_D=0\;\mbox{on}\;\Gamma_D$ and the arguments used in this paper are still true.

The fluid motion in $\Omega_S$ is described by the Stokes equations:
\begin{align}
\nabla \cdot \bm{\sigma}_S+\nabla p_S&=\bm{f}_S\hspace{1.15cm} \mbox{in}\; \Omega_S\label{eq:Stokes1},\\
\bm{\sigma}_S&=-\nu \nabla \bm{u}_S\quad \mbox{in}\; \Omega_S\label{eq:Stokes2},\\
\nabla \cdot \bm{u}_S&=0\hspace{1.45cm}\mbox{in}\; \Omega_S,\label{eq:Stokes3}\\
\bm{u}_S&=0\hspace{1.45cm}\mbox{on}\; \Gamma_S,
\end{align}
where $\nu>0$ denotes the viscosity of the fluid.

On the interface, we prescribe the following interface conditions
\begin{subequations}
\begin{align}
\bm{u}_S\cdot\bm{n}_{S}&=\bm{u}_D\cdot\bm{n}_{S} \quad\; \mbox{on}\; \Gamma,\label{interface1}\\
p_S-\nu\bm{n}_{S}\frac{\partial\bm{u}_S}{\partial \bm{n}_{S}}&=p_D\hspace{1.3cm} \mbox{on}\; \Gamma,\label{interface-normal}\\
-\nu \bm{t}\frac{\partial \bm{u}_{S}}{\partial \bm{n}_{S}}&=G\bm{u}_{S}\cdot\bm{t}\qquad \mbox{on}\; \Gamma.\label{interface-BJS}
\end{align}
\label{interface-condition}
\end{subequations}
Condition (\ref{interface1}) represents continuity of the fluid velocity's normal components, (\ref{interface-normal}) represents the balance of forces acting across the interface, and (\ref{interface-BJS}) is the Beaver-Joseph-Saffman condition \cite{Beaver}. The constant $G>0$ is given and is usually obtained from experimental data.

Before closing this section, we introduce some notations that will be used later. Let $D\subset \mathbb{R}^d, d=1,2$. By $(p, q)_D:=\int_D pq\;dx$, we
denote the inner product in $L^2(D),D\subset \mathbb{R}^2$. We use the same symbol $(\cdot,\cdot)_D$ for the scalar product in $L^2(D)^2$ and in $L^2(D)^{2\times 2}$. We denote by $\langle \cdot,\cdot\rangle_e$ the inner product in $L^2(D),D\subset \mathbb{R}$, and its vector and tensor versions.
% When $D$ coincides with $\Omega$, the subscript $\Omega$ will be
%dropped.
Given an integer  $m\geq 0$ and $p\geq 1$, $W^{m,p}(D)$ and
$W_0^{m,p}(D)$ denote the usual Sobolev space provided the norm
and semi-norm $\|v\|_{W^{m,p}(D)}=\{\sum_{|\alpha|\leq m}\|D^\alpha
v\|^p_{L^p(D)}\}^{1/p}$, $|v|_{W^{m,p}(D)}=\{\sum_{|\alpha|=
m}\|D^\alpha v\|^p_{L^p(D)}\}^{1/p}$. If $p=2$ we usually write
$H^m(D)=W^{m,2}(D)$ and $H_0^m(D)=W_0^{m,2}(D)$,
$\|v\|_{H^m(D)}=\|v\|_{W^{m,2}(D)}$ and $|v|_{H^m(D)}=|v|_{W^{m,2}(D)}$.
In the sequel, we use $C$ to denote a generic positive constant which may have different values at different occurrences.
%Further, we use the notation $H^m(D), \|v\|_{m,D}$ and $|v|_{m,D}$.

%where $\kappa$ is the fraction coefficient.

\begin{figure}
\centering
\includegraphics[width=4cm]{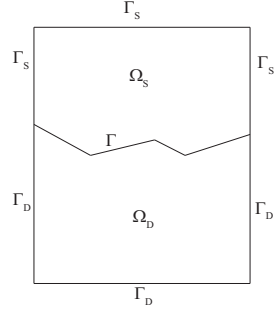}
\caption{Coupled domain with interface $\Gamma$.}
\label{coupled}
\end{figure}

\subsection{Weak formulation}

This section presents the derivation of the weak formulation for the model problems \eqref{eq:Darcy1}-\eqref{interface-BJS}. To begin, set
\begin{align*}
U_S^0=\{\bm{v}_S\in H^1(\Omega_S)^2: \bm{v}_S\mid_{\Gamma_S}=\bm{0}\},\quad \Sigma_0=\{q_D\in  W^{1,3/2}(\Omega_D):q_D\mid_{\Gamma_D}=0\}.
\end{align*}
Then multiplying \eqref{eq:Stokes1} by $\bm{v}_S\in U_S^0$, \eqref{eq:Darcy1} by $q_D\in \Sigma_0$, integration by parts and adding these two equations imply
\begin{align*}
&-(\bm{\sigma}_S, \nabla \bm{v}_S)_{\Omega_S}+\langle \bm{\sigma}_S\bm{n}_{S},\bm{v}_S\rangle_{\Gamma}-(p_S,\nabla \cdot\bm{v}_S)_{\Omega_S}+\langle p_S, \bm{v}\cdot\bm{n}_{S}\rangle_\Gamma-(\bm{u}_D, \nabla q_D)_{\Omega_D}\\
&\;+\langle \bm{u}_D\cdot\bm{n}_{D},q_D\rangle_\Gamma=(\bm{f}_S,\bm{v}_S)_{\Omega_S}+(f_D,q_D)_{\Omega_D}.
\end{align*}
Simple algebraic calculation yields
\begin{align*}
\langle \bm{\sigma}_S\bm{n}_{S},\bm{v}_S\rangle_\Gamma
=\langle\bm{n}_{S}\bm{\sigma}_S\bm{n}_{S},\bm{v}_S\cdot\bm{n}_{S}\rangle_\Gamma+\langle \bm{n}_{S}\bm{\sigma}_S\bm{t},\bm{v}_S\cdot\bm{t}\rangle_\Gamma,
\end{align*}
which gives by employing \eqref{interface-normal} and \eqref{interface-BJS}
\begin{align}
\langle (p_SI+\bm{\sigma}_S)\bm{n}_{S},\bm{v}_S\rangle_\Gamma=\langle p_D, \bm{v}_S\cdot\bm{n}_{S}\rangle_\Gamma+G\langle \bm{u}_S\cdot\bm{t},\bm{v}_S\cdot\bm{t}\rangle_\Gamma.\label{eq:relation}
\end{align}
Here, $I$ is the $2\times 2$ identity matrix.

Therefore, the weak formulation for the coupled Stokes and Darcy-Forchheimer problems reads: find $(\bm{\sigma}_S, \bm{u}_S,p_S)$ $ \in L^2(\Omega_S)^{2\times 2}\times U_S^0\times L^2(\Omega_S)$ and $(\bm{u}_D,p_D)\in L^3(\Omega_D)^2\times \Sigma_0$ such that
\begin{equation}
\begin{split}
&-(\bm{\sigma}_S,\nabla \bm{v}_S)_{\Omega_S}+\langle p_D, \bm{v}_S\cdot\bm{n}_{S}\rangle_\Gamma+G\langle \bm{u}_S\cdot\bm{t}, \bm{v}_S\cdot\bm{t}\rangle_\Gamma-(p_S,\nabla \cdot \bm{v}_S)_{\Omega_S}-(\bm{u}_D,\nabla q_D)_{\Omega_D}\\
&-\langle \bm{u}_S\cdot\bm{n}_{S},q_D\rangle_\Gamma=
(f_D,q_D)_{\Omega_D}+(\bm{f}_S,\bm{v}_S)_{\Omega_S},\\
&\frac{\mu}{\rho}( K^{-1}\bm{u}_D,\bm{v}_D)_{\Omega_D}+
\frac{\beta}{\rho}(|\bm{u}_D|\bm{u}_D,\bm{v}_D)_{\Omega_D}+(\nabla p_D,\bm{v}_D)_{\Omega_D}+\nu^{-1}(\bm{\sigma}_S,\bm{\tau}_S)_{\Omega_S}+(\nabla \bm{u}_S, \bm{\tau}_S)_{\Omega_S}\\
&=(\bm{g},\bm{v}_D)_{\Omega_D},\\
& (\nabla\cdot \bm{u}_S, q_S)_{\Omega_S}=0,\\
&\forall (\bm{\tau}_S,\bm{v}_S,q_S)\in L^2(\Omega_S)^{2\times 2}\times U_S^0\times L^2(\Omega_S),\quad (\bm{v}_D,q_D)\in L^3(\Omega_D)^2\times \Sigma_0.
\end{split}
\label{eq:weak}
\end{equation}

\subsection{Description of staggered DG method}

\begin{figure}
\centering
\includegraphics[width=12cm]{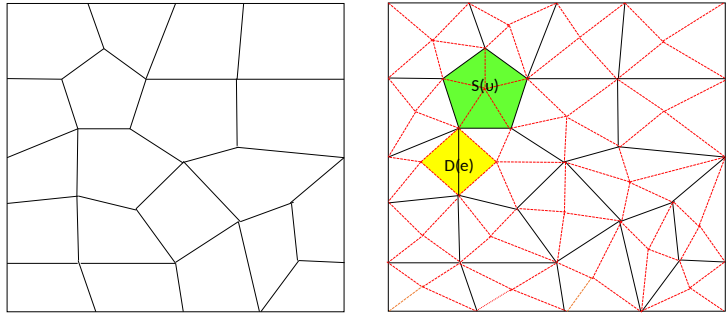}
\setlength{\abovecaptionskip}{-0.5cm}
\caption{Schematic of the primal mesh $S(\nu)$, the dual mesh $D(e)$ and the primal simplexes.}
\label{grid}
\end{figure}

This section gives a detailed derivation of staggered DG method for \eqref{eq:Darcy1}-\eqref{interface-BJS}. Then the inf-sup condition for the bilinear forms related to the Darcy-Forchheimer region is proved motivated by the methodology employed in \cite{ChungWave09}. First, we introduce the staggered meshes and the spaces exploited in the definition of staggered DG method. Following \cite{LinaPark,LinaParkShin}, we first let $\mathcal{T}_{u_i}$ ($i=S,D$) be the
initial partition of the domain $\Omega_i$ into non-overlapping quadrilateral or polygonal meshes. We require that $\mathcal{T}_{u_i}$ be aligned with $\Gamma$.
We also let $\mathcal{F}_{u,i}$ be the set of all edges in the initial partion $\mathcal{T}_{u_i}$ and $\mathcal{F}_{u,i}^{0}\subset \mathcal{F}_{u,i}$ be the
subset of all interior edges of $\Omega_i$. The union of all the nodal points in the initial partition $\mathcal{T}_{u_i}$ is denoted as $\mathcal{P}_i$ and $\mathcal{P}=\mathcal{P}_S\cup \mathcal{P}_D$.
For each quadrilateral or polygonal mesh $E$ in the initial partition $\mathcal{T}_{u_i}$, we select an interior point $\nu$ and
create new edges by
connecting $\nu$ to the vertices of the primal element.
This process will divide $E$ into the union of triangles, where the triangle is denoted as $\tau$, and we rename
the union of these triangles by $S(\nu)$. The union of all $\nu \in \mathcal{T}_{u_i}$ is denoted as $\mathcal{N}_i$ and $\mathcal{N}=\mathcal{N}_S\cup \mathcal{N}_D$. We remark that $S(\nu)$ is the quadrilateral or polygonal mesh in the initial partition.
Moreover, we will use $\mathcal{F}_{p,i}$ to denote the set of all the new edges generated by this subdivision process and
use $\mathcal{T}_{h_i}$ to denote the resulting quasi-uniform triangulation,
on which our basis functions are defined.
% Furthermore,
%$\mathcal{T}_{h_i}$ is assumed to satisfy local quasi-uniform assumption in the sense that for any pair of elements $\tau$ and
%$\tau'$ in $\mathcal{T}_{h_i}$ which share an edge, there exists a
%constant $\kappa$ independent of $h_\tau$ and $h_{\tau'}$ such
%that $\kappa^{-1} \le h_\tau/h_{\tau'} \le \kappa$.
In addition, we define $\mathcal{F}_i:=\mathcal{F}_{u,i}\cup \mathcal{F}_{p,i}$ and $\mathcal{F}_i^{0}:=\mathcal{F}_{u,i}^{0}\cup \mathcal{F}_{p,i}$. For each triangle
$\tau\in \mathcal{T}_{h_i}$, we let $h_\tau$ be the diameter of
$\tau$ and $h_i=\max\{h_\tau, \tau\in \mathcal{T}_{h_i}\}$, and we define $h=\max\{h_S,h_D\}$. Also, we let $h_e$ to denote the length of edge $e\in \mathcal{F}_i$.
This construction is illustrated in Figure~\ref{grid},
where black solid lines are edges in $\mathcal{F}_{u,i}$
and red dotted lines are edges in $\mathcal{F}_{p,i}$.

Moreover, %we define $\mathcal{D}_i^{int}$ as the union of all $D(e)$ for all interior edges $e\in \mathcal{F}_{u,i}^{0}$ and $\mathcal{D}_i^{ext}$
%as the union of all $D(e)$ for all
%boundary edges $e\in \mathcal{F}_{u,i}\backslash \mathcal{F}_{u,i}^{0}$.
the union of all $D(e)$ for $e\in \mathcal{F}_i$ is denoted as $\mathcal{D}_i$. For each interior edge $e\in \mathcal{F}_{u,i}^0$, we use $D(e)$ to denote the union of the two triangles in $\mathcal{T}_{h_i}$ sharing the edge $e$,
and for each boundary edge $e\in\mathcal{F}_{u,i}\backslash\mathcal{F}_{u,i}^0$, we use $D(e)$ to denote the triangle in $\mathcal{T}_{h_i}$ having the edge $e$,
see Figure~\ref{grid}.

For each edge $e$, we define
a unit normal vector $\bm{n}_{e}$ as follows: If $e\in \mathcal{F}_i\setminus \mathcal{F}_i^{0}$, then
$\bm{n}_{e}$ is the unit normal vector of $e$ pointing towards the outside of $\Omega_i$. If $e\in \mathcal{F}_i^{0}$, an
interior edge, we then fix $\bm{n}_{e}$ as one of the two possible unit normal vectors on $e$.
When there is no ambiguity,
we use $\bm{n}$ instead of $\bm{n}_{e}$ to simplify the notation.

We also make the following mesh regularity assumptions (cf. \cite{Beir13,Cangiani16,LinaPark}).
\begin{assumption}\label{assum:regularity}
We assume the existence of a constant $\rho>0$ such that
\begin{enumerate}[itemindent=1em,,label=(\arabic*)]
  \item For every element $E\in \mathcal{T}_{u_i}$ $(i=S,D)$ and every edge $e\in \partial E$, it satisfies $h_e\geq \rho h_E$, where $h_E$ denotes the diameter of $E$.
  \item Each element $E$ in $\mathcal{T}_{u_i}$ is star-shaped with respect to a ball of radius $\geq \rho h_E$.
\end{enumerate}
\end{assumption}
We remark that the above assumptions ensure that the triangulation $\mathcal{T}_{h_i}$ is shape regular.

For a scalar or vector function $v$ belonging to broken Sobolev space, its jump on $e\in
\mathcal{F}_{p,i}$ is defined as
\begin{equation*}
[v]=v_{1}-v_{2},
\end{equation*}
where $v_j=v_{\tau_j},j=1,2$ and $\tau_{1}$, $\tau_{2}$ are the
two triangles in $\mathcal{T}_{h_i}$ having the edge $e$. For the boundary edges, we simply define $[v]=v_{1}$.

Let $k\geq 0$ be the order of approximation. For every $\tau
\in \mathcal{T}_{h_i}$ and $e\in\mathcal{F}_i$,
we define $P_{k}(\tau)$ and $P_{k}(e)$ as the spaces of polynomials of degree less than or equal to $k$ on $\tau$ and $e$, respectively.
Next, we will introduce the staggered DG spaces.
First, we define the following locally $H^{1}(\Omega)-$conforming
staggered DG space $U_{h_i}$ for $i=S,D$
% firstly we define the finiteelement space for the scalar field.
%Locally $H^{1}(\Omega)-$conforming finite element space for the scalar field:
\begin{equation*}
U_{h_i}:=\{v : v\mid_{\tau}\in P_{k}(\tau), \forall\tau \in \mathcal{T}_{h_i}; [v]\mid_e=0\;\forall e\in \mathcal{F}_{u,i}^0, v\mid_{\Gamma_i}=0\}.
\end{equation*}
Notice that, if $w\in U_{h_i}$, then $w\mid_{D(e)}\in H^1(D(e))$ for each edge $e\in \mathcal{F}_{u,i}$. The degrees of freedom for this space can be described as:

(UD1) For each edge $e\in \mathcal{F}_{u,i}$, we have
\begin{align*}
\phi_e(w):=\langle w,p_k\rangle_e \quad \forall p_k\in P_k(e).
\end{align*}

(UD2) For each $\tau\in \mathcal{T}_{h_i}$, we have
\begin{align*}
\phi_\tau(w):=(w, p_{k-1})_\tau \quad \forall p_{k-1}\in P_{k-1}(\tau).
\end{align*}

The discrete $H^1$ norm for the space $U_{h_S}$ is given by
\begin{equation*}
\begin{split}
%\|v\|_{X_i}^2&=\|v\|_{0}^2+\sum_{e\in \mathcal{F}_{u,i}^0}h_e\|v\|_{0,e}^2+\sum_{e\in \Gamma}h_e\|v\|_{0,e}^2,\\
\|v\|_{Z_S}^2&=\sum_{\tau\in \mathcal{T}_{h_S}}\|\nabla v\|_{L^2(\tau)}^2+\sum_{e\in \mathcal{F}_{p,S}}h_e^{-1}\|[v]\|_{0,e}^2.
\end{split}
\end{equation*}
We define the following norm for $\bm{v}=(v_1,v_2)\in [U_{h_S}]^2$
\begin{align*}
\|\bm{v}\|_h^2=\|v_1\|_{Z_S}^2+\|v_2\|_{Z_S}^2.
\end{align*}
Due to the nonlinear term in the Darcy-Forchheimer region, we adjust the commonly employed norms for staggered DG space, instead, we introduce the following mesh-dependent norms for $U_{h_D}$, which flavors our analysis
\begin{align*}
\|v\|_{X_D}^{3/2}&=\int_{\Omega_D} |v|^{3/2}+\sum_{e\in \mathcal{F}_{u,D}^0\cup \Gamma}h_e \int_e |v|^{3/2}\;ds,\\
\|v\|_{Z_D}^{3/2}&=\sum_{\tau\in \mathcal{T}_{h_D}}\int_{\tau} |\nabla v |^{3/2}\;dx+\sum_{e\in \mathcal{F}_{p,D}}h_e^{-1/2}\int_e|[v]|^{3/2}\;ds.
\end{align*}

We next define the following staggered DG space $\bm{V}_{h_i}$
\begin{equation*}
\bm{V}_{h_i}:=\{\bm{\tau}: \bm{\tau}\mid_{\tau} \in [P_{k}(\tau)]^2,\forall \tau \in \mathcal{T}_{h_i};[\bm{\tau}\cdot\bm{n}]\mid_e=0, \forall e\in \mathcal{F}_{p,i}\}.
\end{equation*}
Note that if $\bm{v}\in \bm{V}_{h_i}$, then $\bm{v}\mid_{S(\nu)}\in H(\textnormal{div};S(\nu))$ for each $S(\nu)\in \mathcal{T}_{u_i}$. The following degrees of freedom can be defined correspondingly.

(VD1) For each edge $e\in \mathcal{F}_{p,i}$, we have
\begin{align*}
\psi_e(\bm{v}):=\langle \bm{v}\cdot\bm{n}, p_k\rangle_e \quad \forall p_k\in P_k(e).
\end{align*}

(VD2) For each $\tau\in \mathcal{T}_{h_i}$, we have
\begin{align*}
\psi_\tau(\bm{v}):=(\bm{v}, \bm{p}_{k-1})_\tau \quad \forall \bm{p}_{k-1}\in P_{k-1}(\tau)^2.
\end{align*}
We define the following discrete $L^2$ norm and discrete $H(\textnormal{div};\Omega_S)$ seminorm for $\bm{V}_{h_S}$
\begin{equation*}
\begin{split}
\|\bm{\tau}\|_{X_S'}^2&=\|\bm{\tau}\|_{L^2(\Omega_S)}^2+\sum_{e\in \mathcal{F}_{p,S}}h_e\|\bm{\tau}\cdot \bm{n}\|_{0,e}^2,\\
\|\bm{\tau}\|_{Z_S'}^2&=\sum_{\tau\in \mathcal{T}_{h_S}}\|\nabla \cdot \bm{\tau}\|_{0,\tau}^2+\sum_{e\in \mathcal{F}_{u,S}^0}h_e^{-1}\| [\bm{\tau}\cdot \bm{n}]\|_{0,e}^2+\sum_{e\in \Gamma}h_e^{-1}\|\bm{\tau}\cdot \bm{n}\|_{0,e}^2.
\end{split}
\end{equation*}
In addition, $\bm{V}_{h_D}$ is equipped by
\begin{align*}
\|\bm{v}_D\|_{X_D'}^3&=\int_{\Omega_D} |\bm{v}_D|^3\;dx+\sum_{e\in \mathcal{F}_{p,D}}h_e\int_e |\bm{v}_D\cdot\bm{n}|^3\;ds,\\
\|\bm{v}_D\|_{Z_D'}^3&=\sum_{\tau\in \mathcal{T}_{h_D}}\int_{\tau} |\nabla \cdot \bm{v}_D|^{3}\;dx+\sum_{e\in \mathcal{F}_{u,D}^0}h_e^{-2}\int_e |[\bm{v}_D\cdot\bm{n}]|^{3}\;ds+\sum_{e\in \Gamma}h_e^{-2}\int_e |\bm{v}_D\cdot\bm{n}|^{3}\;ds.
\end{align*}
Scaling arguments imply there exists a positive constant $C$ such that
\begin{align}
C\|\bm{\tau}\|_{X_S'}\leq \|\bm{\tau}\|_{L^2(\Omega_S)}\leq \|\bm{\tau}\|_{X_S'} \quad\forall \bm{\tau}\in \bm{V}_{h_S}.\label{scaling}
\end{align}

%\begin{align*}
%\|q_2\|_{M}^{3/2}&=\int_\Omega q_2^{3/2}+\sum_{e\in \mathcal{F}_u^0}h_e \int_e | q_2|^{3/2}\;ds,\\
%\|q_D\|_{E}^{3/2}&=\int_\Omega |\nabla q_D|^{3/2}\;dx+\sum_{e\in \mathcal{F}_p}h_e^{-1/2}\int_e|[q_D]|^{3/2}\;ds,\\
%\|\bm{v}_D\|_{M'}^3&=\int_\Omega |\bm{v}_D|^3\;dx+\sum_{e\in \mathcal{F}_p}h_e\int_e |\bm{v}_D\cdot\bm{n}|^3\;ds,\\
%\|\bm{v}_D\|_{E'}^3&=\int_\Omega |\nabla \cdot \bm{v}_D|^{3}\;dx+\sum_{e\in \mathcal{F}_p}h_e^{-2}\int_e |\bm{v}_D\cdot\bm{n}|^{3}\;ds
%\end{align*}

Finally, locally $H^{1}(\Omega)-$conforming finite element space for pressure is defined as:
\begin{equation*}
P_{h}:=\{q: q\mid_{S(\nu)}\in P_{k}(S(\nu))\;\forall \nu\in \mathcal{N}_S\}
\end{equation*}
with norm
\begin{align*}
\|q\|_P^2=\|q\|_{L^2(\Omega_S)}^2+\sum_{e\in \mathcal{F}_{p,S}}h_e\|q\|_{0,e}^2.
\end{align*}
Norm equivalence yields
\begin{align*}
C\|q\|_P\leq \|q\|_{L^2(\Omega_S)}\leq \|q\|_P \quad \forall q\in P_h.
\end{align*}

Now we are ready to derive the discrete formulation for \eqref{eq:Darcy1}-\eqref{interface-condition}. Multiplying \eqref{eq:Stokes1} by $\bm{v}_S\in [U_{h_S}]^2$, \eqref{eq:Stokes2} by $\bm{\tau}_S\in [\bm{V}_{h_S}]^2$, \eqref{eq:Stokes3} by $q_S\in P_h$ and integration by parts yield
\begin{equation}
\begin{split}
&-(\bm{\sigma}_S,\nabla \bm{v}_S)_{\Omega_S}+\sum_{e\in \mathcal{F}_{p,S}}\langle \bm{\sigma}_S\bm{n},[\bm{v}_S]\rangle_e+\langle \bm{\sigma}_S\bm{n}_{S},\bm{v}_S\rangle_\Gamma+\sum_{e\in \mathcal{F}_{p,S}}\langle p_S,[\bm{v}_S\cdot\bm{n}]\rangle_e\\
&\;+\langle p_S,\bm{v}_S\cdot\bm{n}_{S}\rangle_\Gamma-(p_S,\nabla \cdot \bm{v}_S)_{\Omega_S} =(\bm{f}_S,\bm{v}_S)_{\Omega_S},\\
&\nu^{-1}(\bm{\sigma}_S,\bm{\tau}_S)_{\Omega_S}=-\sum_{e\in \mathcal{F}_{u,S}^0}\langle \bm{u}_S,[\bm{\tau}_S\bm{n}]\rangle_e-\langle \bm{u}_S, \bm{\tau}_S\bm{n}_{S}\rangle_\Gamma+(\bm{u}_S,\nabla \cdot\bm{\tau}_S)_{\Omega_S},\\
&\sum_{e\in \mathcal{F}_{u,S}^0}\langle \bm{u}_S\cdot\bm{n},[q_S]\rangle_e+\langle \bm{u}_S\cdot\bm{n}_{S},q_S\rangle_\Gamma-(\bm{u}_S, \nabla q_S)_{\Omega_S}=0.
\end{split}
\label{eq:weakS}
\end{equation}
Multiplying \eqref{eq:Darcy1} by $q_D\in U_{h_D}$ and \eqref{eq:Darcy2} by $\bm{v}_D\in \bm{V}_{h_D}$, it then follows from integration by parts
\begin{equation}
\begin{split}
&\sum_{e\in \mathcal{F}_{p,D}}\langle \bm{u}_D\cdot\bm{n},[q_D]\rangle_e+\langle \bm{u}_D\cdot\bm{n}_{D},q_D\rangle_\Gamma-(\bm{u}_D,\nabla q_D)_{\Omega_D}=(f_D,q_D)_{\Omega_D},\\
&\frac{\mu}{\rho}(K^{-1}\bm{u}_D,\bm{v}_D)_{\Omega_D}
+\frac{\beta}{\rho}(|\bm{u}_D|\bm{u}_D,\bm{v}_D)_{\Omega_D}+\sum_{e\in \mathcal{F}_{u,D}^0}\langle p_D,[\bm{v}_D\cdot\bm{n}]\rangle_e+\langle p_D,\bm{v}_D\cdot\bm{n}_{D}\rangle_\Gamma\\
&\;-(p_D,\nabla \cdot\bm{v}_D)_{\Omega_D}=(\bm{g}_D,\bm{v}_D)_{\Omega_D}.
\end{split}
\label{eq:weakD}
\end{equation}
%Simple algebraic calculation yields
%\begin{align*}
%\langle \bm{\sigma}_S\bm{n}_{S},\bm{v}_S\rangle_\Gamma
%=\langle\bm{n}_{S}\bm{\sigma}_S\bm{n}_{S},\bm{v}_S\cdot\bm{n}_{S}\rangle_\Gamma+\langle \bm{n}_{S}\bm{\sigma}_S\bm{t},\bm{v}_S\cdot\bm{t}\rangle_\Gamma,
%\end{align*}
%which gives by employing \eqref{interface-normal} and \eqref{interface-BJS}
%\begin{align}
%\langle (p_SI+\bm{\sigma}_S)\bm{n}_{S},\bm{v}_S\rangle_\Gamma=\langle p_D, \bm{v}_S\cdot\bm{n}_{S}\rangle_\Gamma+G\langle \bm{u}_S\cdot\bm{t},\bm{v}_S\cdot\bm{t}\rangle_\Gamma.
%\label{eq:relation}
%\end{align}
Adding the first equation of \eqref{eq:weakS} and \eqref{eq:weakD}, and employing \eqref{eq:relation}, we have
\begin{align*}
&-(\bm{\sigma}_S,\nabla \bm{v}_S)_{\Omega_S}+\sum_{e\in \mathcal{F}_{p,S}}\langle \bm{\sigma}_S\bm{n},[\bm{v}_S]\rangle_e+\langle p_D,\bm{v}_S\cdot\bm{n}_{S}\rangle_\Gamma+G\langle \bm{u}_S\cdot\bm{t},\bm{v}_S\cdot\bm{t}\rangle_\Gamma\\
&\;+\sum_{e\in \mathcal{F}_{p,S}}\langle p_S,[\bm{v}_S\cdot\bm{n}]\rangle_e-(p_S,\nabla\cdot\bm{v}_S)_{\Omega_S}+\sum_{e\in \mathcal{F}_{p,D}}\langle \bm{u}_D\cdot\bm{n},[q_D]\rangle_e+\langle \bm{u}_S\cdot\bm{n}_{D},q_D\rangle_\Gamma\\
&\;-(\bm{u}_D,\nabla q_D)_{\Omega_D} =(\bm{f}_S,\bm{v}_S)_{\Omega_S}+(f_D,q_D)_{\Omega_D}.
\end{align*}
Adding the second equations of \eqref{eq:weakS} and \eqref{eq:weakD}, we can get
\begin{align*}
&\nu^{-1}(\bm{\sigma}_S,\bm{\tau}_S)_{\Omega_S}+\sum_{e\in \mathcal{F}_{u,S}^0}\langle \bm{u}_S,[\bm{\tau}_S\bm{n}]\rangle_e+\langle \bm{u}_S, \bm{\tau}_S\bm{n}_{S}\rangle_\Gamma-(\bm{u}_S,\nabla \cdot \bm{\tau}_S)_{\Omega_S}\\
&\;+\frac{\mu}{\rho}(K^{-1}\bm{u}_D,\bm{v}_D)_{\Omega_D}+
\frac{\beta}{\rho}(|\bm{u}_D|\bm{u}_D,\bm{v}_D)_{\Omega_D}+\sum_{e\in \mathcal{F}_{u,D}^0}\langle p_D,[\bm{v}_D\cdot\bm{n}]\rangle_e\\
&\;+\langle p_D,\bm{v}_D\cdot\bm{n}_{D}\rangle_\Gamma-(p_D,\nabla \cdot\bm{v}_D)_{\Omega_D}=(\bm{g}_D,\bm{v}_D)_{\Omega_D}.
\end{align*}
For the convenience of the presentation, we introduce the nonlinear operator $A:L^3(\Omega_D)^2\rightarrow L^{3/2}(\Omega_D)^2$ defined by
\begin{align}
%D^{-1}(\bm{u})&=\frac{\mu}{\varrho K}+\frac{\beta}{\varrho}|\bm{u}|,\\
%D^{-1}(\bm{u})&=\frac{\mu}{\varrho}K^{-1}
%+\frac{\beta}{\varrho}|\bm{u}| I, \quad \mbox{if} K\; \mbox{is tensor},\\
A(\bm{v})&=\frac{\mu\Red{K^{-1}}}{\rho}\bm{v}+\frac{\beta}{\rho}|\bm{v}|\bm{v} \quad \forall \bm{v}\in L^3(\Omega_D)^2.\label{A}
\end{align}
In addition, we define the following bilinear forms corresponding to the aforementioned derivations
\begin{align*}
a_S(\bm{\tau}_S,\bm{v}_S)&=-\sum_{\tau\in \mathcal{T}_{h_S}}(\bm{\tau}_S,\nabla \bm{v}_S)_{\tau}+\sum_{e\in \mathcal{F}_{p,S}}\langle \bm{\tau}_S\bm{n},[\bm{v}_S]\rangle_e,\\
a_S^*(\bm{v}_S,\bm{\tau}_S)&=-\sum_{e\in \mathcal{F}_{u,S}^0}\langle \bm{v}_S,[\bm{\tau}_S\bm{n}]\rangle_e-\langle \bm{v}_S,\bm{\tau}_S\bm{n}_{S}\rangle_\Gamma+\sum_{\tau\in \mathcal{T}_{h_S}}(\nabla \cdot\bm{\tau}_S,\bm{v}_S)_{\tau},\\
b_S(\bm{v}_S,q_S)&=-\sum_{e\in \mathcal{F}_{u,S}^0}\langle \bm{v}_S\cdot\bm{n},[q_S]\rangle_e-\langle \bm{v}_S\cdot\bm{n}_{S},q_S\rangle_\Gamma+\sum_{\tau\in \mathcal{T}_{h_S}}(\bm{v}_S,\nabla q_S)_{\tau},\\
b_S^*(q_S,\bm{v}_S)&=\sum_{e\in \mathcal{F}_{p,S}}\langle q_S,[\bm{v}_S\cdot\bm{n}]\rangle_e-\sum_{\tau\in \mathcal{T}_{h_S}}(q_S,\nabla \cdot \bm{v}_S)_{\tau},\\
a_D(\bm{v}_D,q_D)&=\sum_{e\in \mathcal{F}_{p,D}}\langle \bm{v}_D\cdot\bm{n},[q_D]\rangle_e-\sum_{\tau\in \mathcal{T}_{h_D}}(\bm{v}_D,\nabla q_D)_{\tau},\\
a_D^*(q_D,\bm{v}_D)&=-\sum_{e\in \mathcal{F}_{u,D}^0}\langle q_D,[\bm{v}_D\cdot\bm{n}]\rangle_e-\langle q_D, \bm{v}_D\cdot\bm{n}_{D}\rangle_{\Gamma}+\sum_{\tau\in \mathcal{T}_{h_D}}(q_D,\nabla \cdot \bm{v}_D)_{\tau}.
\end{align*}

Then, the staggered DG formulation for the coupled Stokes and Darcy-Forchheimer problems reads: find $(\bm{\sigma}_{S,h},\bm{u}_{S,h},p_{S,h})$ $\in [\bm{V}_{h_S}]^2\times [U_{h_S}]^2\times P_h$ and $(\bm{u}_{D,h}, p_{D,h})\in \bm{V}_{h_D}\times U_{h_D}$ such that
\begin{equation}
\begin{split}
&a_S(\bm{\sigma}_{S,h},\bm{v}_S)+b_S^*(p_{S,h},\bm{v}_S)+a_D(\bm{u}_{D,h},q_D)+\langle p_{D,h},\bm{v}_S\cdot\bm{n}_{S} \rangle_\Gamma+G\langle \bm{u}_{S,h}\cdot\bm{t},\bm{v}_S\cdot\bm{t}\rangle_\Gamma\\
&\;-\langle\bm{u}_{S,h}\cdot\bm{n}_{S},q_D\rangle_\Gamma
=(f_D,q_D)_{\Omega_D}+(\bm{f}_S,\bm{v}_S)_{\Omega_S},\\
&\nu^{-1}(\bm{\sigma}_{S,h},\bm{\tau}_S)_{\Omega_S}-a_S^*(\bm{u}_{S,h},\bm{\tau}_S)+
(A(\bm{u}_{D,h}),\bm{v}_D)_{\Omega_D}-a_D^*(p_{D,h},\bm{v}_D)=(\bm{g}_D,\bm{v}_D)_{\Omega_D},\\
&b_S(\bm{u}_{S,h},q_S)=0,\\
& \forall (\bm{\tau}_S,\bm{v}_S,q_S)\in [\bm{V}_{h_S}]^2\times [U_{h_S}]^2\times P_h, \; (\bm{v}_D,q_D)\in  \bm{V}_{h_D}\times U_{h_D}.
\end{split}
\label{eq:discrete1}
\end{equation}

Next, we define the following projection operators, which are defined by employing the special property of staggered DG method.
Let $I_h: H^1(\Omega_i)\rightarrow U_{h_i},i=S,D$ be defined by
\begin{align*}
\langle I_h v-v,\phi\rangle_e &=0 \quad \forall \phi\in P_k(e),\ e\in \mathcal{F}_{u,i},\\
(I_hv-v,\phi)_\tau&=0\quad \forall \phi\in P_{k-1}(\tau),\ \tau\in \mathcal{T}_{h_i}.
\end{align*}
In addition, $J_h: H^1(\Omega_i)\rightarrow \bm{V}_{h_i}$ is defined by
\begin{align*}
\langle(J_h\bm{q}-\bm{q})\cdot\bm{n},v\rangle_e&=0\quad \forall v\in P_{k}(e),\ \forall e\in \mathcal{F}_{p,i},\\
(J_h\bm{q}-\bm{q}, \bm{\phi})_\tau&=0\quad \forall \bm{\phi}\in P_{k-1}(\tau)^2,\ \forall \tau\in \mathcal{T}_{h_i}.
\end{align*}
Straightforwardly, we have
\begin{equation}
\begin{split}
a_D(\bm{u}_D-J_h\bm{u}_D, q_D)&=0  \quad \forall q_D\in U_{h_D},\\
a_S(\bm{\sigma}_S-J_h\bm{\sigma}_S,\bm{v}_S)&=0\quad \forall \bm{v}_S\in [U_{h_S}]^2.
\end{split}
\label{eq:interpolationI}
\end{equation}
In addition, we can obtain the following error estimates for $i=S,D$ by following \cite{Ciarlet78}
\begin{equation}
\begin{split}
\|\psi-I_h \psi\|_{L^2(\Omega_i)}&\leq C h^{k+1}\|\psi\|_{H^\alpha(\Omega_i)}\hspace{0.5cm} \forall \psi\in H^{k+1}(\Omega_i),\\
\|I_h\psi\|_{L^{3/2}(\Omega_D)}&\leq C \|\psi\|_{L^{3/2}(\Omega_D)}\hspace{1cm} \forall \psi\in L^{3/2}(\Omega_D),\\
\|\bm{\phi}-I_h\bm{\phi}\|_{h}&\leq C h^{k}\|\bm{\phi}\|_{H^{k+1}(\Omega_S)}\hspace{0.5cm} \forall \bm{\phi}\in H^{k+1}(\Omega_S)^2,\\
\|\bm{z}-J_h\bm{z}\|_{L^2(\Omega_i)}&\leq Ch^{k+1} \|\bm{z}\|_{H^{k+1}(\Omega_i)}\hspace{0.4cm} \forall \bm{z}\in H^{k+1}(\Omega_i)^2,\\
\|J_h\bm{z}\|_{L^3(\Omega_D)}&\leq C \|\bm{z}\|_{L^3(\Omega_D)}\hspace{1.4cm} \forall \bm{z}\in L^3(\Omega_D)^2.
\end{split}
\label{eq:interpolation}
\end{equation}
%where $1\leq \alpha \leq k+1$.

Here, we briefly derive the error estimates for $\|\cdot\|_{Z_D}$. By virtue of the trace inequality (see, e.g., \cite{BrennerScott08}) and Young's inequality, we can obtain
\begin{align*}
\|v\|_{L^{3/2}(\partial \tau)}^{3/2}\leq C \Big(\|v\|_{L^{3/2}(\tau)}^{1/3}\|v\|_{W^{1,3/2}(\tau)}^{2/3}\Big)^{3/2}\leq C \Big(h_\tau^{1/2}\|w\|_{W^{1,3/2}(\tau)}^{3/2}+h_\tau^{-1}\|w\|_{L^{3/2}(\tau)}^{3/2}\Big).
\end{align*}
It then follows from the fact that $I_h$ is the polynomial preserving operator (cf. \cite{Ciarlet78})
\begin{align}
\|v-I_hv\|_{Z_D}\leq  C h^{\alpha-1}\|v\|_{W^{\alpha,3/2}(\Omega_D)}\quad \forall v\in W^{\alpha,3/2}(\Omega_D)\label{eq:IhD}.
\end{align}

Next, we define $\pi_h$ to be the standard conforming finite element interpolation operator for the triangulation $\mathcal{T}_{h_S}$. Then we can obtain (cf. \cite{Ciarlet78})
\begin{align}
\|q-\pi_h q\|_{L^2(\Omega_S)}\leq C h^{k+1} \|q\|_{H^{k+1}(\Omega_S)}\; \mbox{and}\; \|q-\pi_hq\|_{L^2(e)}\leq C h^{k+1/2}\|q\|_{H^{k+1}(\tau)}, \label{eq:pih}
\end{align}
where $e\subset \partial \tau$ and $\tau\in \mathcal{T}_{h_S}$.

Given any function $\zeta\in L^2(e)\; \forall e\in \mathcal{F}_D$, the restriction of $P_\partial \zeta$ to each edge $e$ is defined as the element of $P_k(e)$ that satisfies
\begin{align*}
\langle P_\partial \zeta-\zeta,\phi\rangle_e=0\quad \forall \phi\in P_k(e), e\in \mathcal{F}_D,
\end{align*}
which by definition gives for any $p> 1$
\begin{align}
\|P_\partial \zeta\|_{L^p(e)}\leq C \|\zeta\|_{L^p(e)}.\label{eq:pi}
\end{align}

The next two lemmas prove the inf-sup conditions for the bilinear forms $a_D(\cdot,\cdot)$ and $a_D^*(\cdot,\cdot)$. We note that the proof is motivated by the methodology exploited in \cite{ChungWave09}, while some critical modifications are needed to fit into the current framework.
\begin{lemma}\label{lemma:infsupa}
(inf-sup condition).
There exists a uniform constant $C>0$ such that
\begin{align*}
\inf_{\bm{v}\in \bm{V}_{h_D}}\sup_{q\in U_{h_D}}\frac{a_D^*(q,\bm{v})}{\|q\|_{X_D}\|\bm{v}\|_{Z_D'}}\geq C.
\end{align*}

\end{lemma}

\begin{proof}
Let $\bm{v}\in \bm{V}_{h_D}$. It suffices to find $q\in U_{h_D}$ such that
\begin{align*}
a_D^*(q,\bm{v})\geq C \|\bm{v}\|_{Z_D'}^{3}\quad  \mbox{and}\quad  \|q\|_{X_D}^{3/2}\leq C\|\bm{v}\|_{Z_D'}^3.
\end{align*}

First, we will find $q_1$ such that $q_1\mid_e=0$ for each $e\in \mathcal{F}_{u,D}^0\cup \Gamma$ and the following hold
\begin{align}
a_D^*(q_1,\bm{v})\geq C \sum_{\tau\in \mathcal{T}_{h_D}}\int_{\tau} |\nabla \cdot \bm{v}|^3\;dx \quad \mbox{and}\quad \int_{\Omega_D} |q_1|^{3/2}\;dx\leq C\sum_{\tau\in \mathcal{T}_{h_D}}\int_{\tau} |\nabla \cdot \bm{v}|^3\;dx.\label{eq:infaD}
\end{align}
The second inequality in \eqref{eq:infaD} implies that $\|q_1\|_{X_D}^{3/2}\leq C \|\bm{v}\|_{Z_D'}^3$.

Note that for each $\tau\in \mathcal{T}_h$, there is exactly one edge $e\subset \partial \tau$ that belongs to $\mathcal{F}_{u,D}^0\cup \Gamma$, thus, we define $q_1^*$ by
\begin{align*}
q_1^*\mid_\tau=\lambda_{\tau,1}\nabla \cdot \bm{v}|\nabla \cdot \bm{v}|,
\end{align*}
where $\lambda_{\tau,1}$ is the unique linear function defined on $\tau$ such that $\lambda_{\tau,1}=0$ at the two vertices of $e$ and $\lambda_{\tau,1}=1$ at the other vertex of $\tau$. Then let $q_1=I_h q_1^*$, we have $q_1\in U_{h_D}$. Since $\lambda_{\tau,1}\leq 1$, it then follows
\begin{align*}
\int_{\Omega_D} |q_1|^{3/2}\;dx\leq C\int_{\Omega_D} |q_1^*|^{3/2}\;dx \leq C \sum_{\tau\in \mathcal{T}_{h_D}}\int_{\tau} |\nabla \cdot\bm{v}|^3\;dx,
\end{align*}
which proves the second inequality in \eqref{eq:infaD}. Since $\lambda_{\tau,1}=0$ on $e\in \mathcal{F}_{u,D}^0\cup \Gamma$, we have
\begin{align*}
\sum_{e\in \mathcal{F}_{u,D}^0}\int_e q_1[\bm{v}\cdot\bm{n}]\;ds+\langle q_1, \bm{v}\cdot\bm{n}\rangle_\Gamma=0.
\end{align*}

By using the norm equivalence, we can get
\begin{align*}
\sum_{\tau\in \mathcal{T}_{h_D}}\int_{\tau} q_1\nabla \cdot \bm{v}\;dx&=\sum_{\tau\in \mathcal{T}_{h_D}}\int_{\tau} q_1^*\nabla \cdot \bm{v}\;dx=\sum_{\tau\in \mathcal{T}_{h_D}}\int_\tau\lambda_{\tau,1}(\nabla \cdot \bm{v})^2|\nabla \cdot \bm{v}|\;dx\\
&\geq C \sum_{\tau\in \mathcal{T}_{h_D}}\int_{\tau} |\nabla \cdot \bm{v}|^3\;dx.
\end{align*}
Therefore, the first inequality in \eqref{eq:infaD} is proved.

Secondly, we will find $q_2$ as follows. Let $U_h^0$ be the subset of $U_{h_D}$ defined by
\begin{align*}
U_h^0=\{q\in U_{h_D}| \int_\tau qv_{k-1}\;dx=0,\forall v_{k-1}\in P_{k-1}(\tau), \forall \tau\in \mathcal{T}_{h_D}\}.
\end{align*}
Then we define $q_2\in U_h^0$ such that all degrees of freedom associated with (UD1) are given by
\begin{align}
\phi_e(q_2)=\int_e q_2p_e\;ds=-h_e^{-2}\int_e \frac{|[\bm{v}\cdot\bm{n}]|^3}{[\bm{v}\cdot\bm{n}]}p_e\;ds\quad \forall p_e\in P_k(e)\label{eq:defphi}
\end{align}
for all $e\in \mathcal{F}_u^0\cup \Gamma$. For each $e\in \mathcal{F}_u^0\cup \Gamma$, taking $p_e=[\bm{v}\cdot\bm{n}]$ in \eqref{eq:defphi}, we can get
\begin{align*}
\int_e q_2[\bm{v}\cdot\bm{n}]\;ds=-h_e^{-2}\int_e|[\bm{v}\cdot\bm{n}]|^3\;ds.
\end{align*}
It then follows
\begin{equation}
\begin{split}
a_D^*(q_2,\bm{v})&=-\sum_{e\in \mathcal{F}_{u,D}^0}(q_2,[\bm{v}\cdot\bm{n}])_e-\langle q_2,\bm{v}\cdot\bm{n}\rangle_\Gamma\\
&=\sum_{e\in \mathcal{F}_{u,D}^0}h_e^{-2}\int_e |[\bm{v}\cdot\bm{n}]|^3\;ds+\int_\Gamma h_e^{-2}|\bm{v}\cdot\bm{n}|^3\;ds.
\end{split}
\label{eq:aDs1}
\end{equation}

Taking $p_e=P_\partial(\frac{|q_2|^{3/2}}{q_2})$ in \eqref{eq:defphi}, we have from \eqref{eq:pi}
\begin{align*}
\int_e q_2p_e\;ds&=\int_e q_2(\frac{|q_2|^{3/2}}{q_2})\;ds=\int_e |q_2|^{3/2}\;ds=-h_e^{-2}\int_e \frac{|[\bm{v}\cdot\bm{n}]|^3}{[\bm{v}\cdot\bm{n}]}P_{\partial}(\frac{|q_2|^{3/2}}{q_2})\;ds\\
&\leq C h_e^{-2}(\int_e |[\bm{v}\cdot\bm{n}]|^3\;ds)^{2/3}(\int_e |q_2|^{3/2}\;ds)^{1/3}.
\end{align*}
Therefore
\begin{align*}
\int_e |q_2|^{3/2}\;ds\leq C h_e^{-3}\int_e |[\bm{v}\cdot\bm{n}]|^3\;ds.
\end{align*}
%\begin{align*}
%\|q_2\|_{0,3/2,e}\leq C h_e^{-2}\|[\bm{v}\cdot\bm{n}]\|_{0,3,e}^2.
%\end{align*}
%namely
%\begin{align*}
%\|q_2\|_{0,3/2,e}^{3/2}\leq C h_e^{-3}\|[\bm{v}\cdot\bm{n}]\|_{0,3,e}^3.
%\end{align*}
Besides, norm equivalence yields
\begin{align*}
\int_\tau |q_2|^{3/2}\;dx\leq C h_e\int_e |q_2|^{3/2}\;ds\quad \forall e\in \partial \tau\cap (\mathcal{F}_{u,D}^0\cup \Gamma).
\end{align*}
Therefore, we can get
\begin{align*}
\|q_2\|_{X_D}^{3/2}&=\int_{\Omega_D} |q_2|^{3/2}\;dx+\sum_{e\in \mathcal{F}_{u,D}^0\cup \Gamma}h_e \int_e | q_2|^{3/2}\;ds\\
&\leq C\Big(\sum_{e\in \mathcal{F}_{u,D}^0}h_e^{-2}\int_e |[\bm{v}\cdot\bm{n}]|^{3}\;ds+\sum_{e\in \Gamma}h_e^{-2}\int_e |\bm{v}\cdot\bm{n}|^{3}\;ds\Big).
\end{align*}
Let $q=q_1+q_2$. Then
\begin{align*}
\|q\|_{X_D}^{3/2}\leq  \|q_1\|_{X_D}^{3/2}+\|q_2\|_{X_D}^{3/2}\leq C \|\bm{v}\|_{Z_D'}^3.
\end{align*}
We have from \eqref{eq:infaD} and \eqref{eq:aDs1}
\begin{align*}
a_D^*(q,\bm{v})=a_D^*(q_1,\bm{v})+a_D^*(q_2,\bm{v})\geq C \|\bm{v}\|_{Z_D'}^3.
\end{align*}

\end{proof}

\begin{lemma}\label{lemma:insup3}
(inf-sup condition).
There exists a uniform constant $C>0$ such that
\begin{align*}
\inf_{q\in U_{h_D}}\sup_{\bm{v}\in \bm{V}_{h_D}}\frac{a_D(\bm{v},q)}
{\|q\|_{Z_D}\|\bm{v}\|_{X_D'}}
\geq C.
\end{align*}

\end{lemma}

\begin{proof}

Let $q\in U_{h_D}$. It suffices to find $\bm{v}\in \bm{V}_{h_D}$ such that
\begin{align*}
a_D(\bm{v},q)\geq C \|q\|_{Z_D}^{3/2} \quad \mbox{and}\quad \|\bm{v}\|_{X_D'}^3\leq C \|q\|_{Z_D}^{3/2}.
\end{align*}
We first find $\bm{v}_1\in \bm{V}_{h_D}$ such that $\bm{v}_1\cdot\bm{n}=0$ for all $e\in \mathcal{F}_{p,D}$ and
\begin{align}
a_D(\bm{v}_1,q)\geq C  \sum_{\tau\in \mathcal{T}_{h_D}}\int_{\tau}|\nabla q|^{3/2}\;dx\quad \mbox{and} \quad \int_{\Omega_D}|\bm{v}_1|^3\;dx\leq C\sum_{\tau\in \mathcal{T}_{h_D}} \int_{\tau}|\nabla q|^{3/2}\;dx.\label{eq:aD}
\end{align}
In particular, the second inequality in \eqref{eq:aD} implies that $\|\bm{v}_1\|_{X_D'}^3\leq C \|q\|_{Z_D}^{3/2}$. Let $\tau\in \mathcal{T}_{h_D}$. Then, on the boundary $\partial \tau$, there are exactly two edges $e_i\in \mathcal{F}_{p,D},i=1,2$ with the corresponding unit normal vectors $\bm{n}^{(i)}$, we take
\begin{align*}
\bm{v}_1^*\cdot\bm{n}^{(1)}=-\lambda_{\tau,1}\frac{q_{k-1}^{(1)}}{|q_{k-1}^{(1)}|^{1/2}},\quad
\bm{v}_1^*\cdot\bm{n}^{(2)}=-\lambda_{\tau,2}\frac{q_{k-1}^{(2)}}{|q_{k-1}^{(2)}|^{1/2}},
\end{align*}
where $q_{k-1}^{(1)}$ and $q_{k-1}^{(2)}$ will be determined later. In the above definition $\lambda_{\tau,i},i=1,2$ is the unique linear function defined on $\tau$ such that $\lambda_{\tau,i}=0$ at the two vertices of $e_i$ and $\lambda_{\tau,i}=1$ at the remaining vertex of $\tau$. By this definition, we have $\bm{v}_1^*\cdot\bm{n}=0$ for all $e\in \mathcal{F}_{p,D}$. Let $\mathcal{A}$ be the matrix such that the rows are given by the two vectors $\bm{n}^{(1)}$ and $\bm{n}^{(2)}$. Then clearly $\mathcal{A}$ is invertible. Moreover
\begin{align*}\mathcal{A}\bm{v}_1^*=
\left(
  \begin{array}{c}
    -\lambda_{\tau,1} \frac{q_{k-1}^{(1)}}{|q_{k-1}^{(1)}|^{1/2}}\\
    -\lambda_{\tau,2} \frac{q_{k-1}^{(2)}}{|q_{k-1}^{(2)}|^{1/2}} \\
  \end{array}
\right),
\quad \bm{v}_1^*=-\mathcal{A}^{-1}
\left(
  \begin{array}{c}
   \lambda_{\tau,1}\frac{q_{k-1}^{(1)}}{|q_{k-1}^{(1)}|^{1/2}}\\
    \lambda_{\tau,2} \frac{q_{k-1}^{(2)}}{|q_{k-1}^{(2)}|^{1/2}}\\
  \end{array}
\right).
\end{align*}
Now we define $q_{k-1}^{(1)}$ and $q_{k-1}^{(2)}$ as the two functions such that
\begin{align*}
\mathcal{A}\nabla q =\left(
                      \begin{array}{c}
                        q_{k-1}^{(1)} \\
                        q_{k-1}^{(2)} \\
                      \end{array}
                    \right).
\end{align*}
Let $\bm{v}_1=J_h\bm{v}_1^*$, it then follows
\begin{align*}
-\int_\tau \bm{v}_1\cdot\nabla q\;dx=-\int_\tau \bm{v}_1^*\cdot\nabla q\;dx \geq C \int_\tau |\nabla q|^{3/2}\;dx.
\end{align*}
Summing over all $\tau$ and using the fact that $\langle \bm{v}_1\cdot\bm{n},[q]\rangle_e=0$ for all $e\in \mathcal{F}_{p,D}$ we obtain the first inequality in \eqref{eq:aD}. For the second inequality of \eqref{eq:aD}, we have
\begin{align*}
\int_{\Omega_D} |\bm{v}_1|^3\;dx \leq C \sum_{\tau\in \mathcal{T}_{h_D}}\int_\tau \sum_{i=1}^2 |\bm{v}_1^*\cdot\bm{n}^{(i)}|^3\;dx\leq C \sum_{\tau\in \mathcal{T}_{h_D}}\int_\tau \sum_{i=1}^2|q_{k-1}^{(i)}|^{3/2}\;dx\leq C\sum_{\tau\in \mathcal{T}_{h_D}} \int_{\tau} |\nabla q|^{3/2}\;dx.
\end{align*}
Secondly, we will find $\bm{v}_2$ as follows. Let $\bm{V}_h^0$ be the subset of $\bm{V}_{h_D}$ defined by
\begin{align*}
\bm{V}_h^0=\{\bm{v}\in \bm{V}_{h_D}| \int_\tau \bm{v}\cdot \bm{p}_{k-1}\;dx=0\quad \forall \bm{p}_{k-1}\in P_{k-1}(\tau)^2,\;\forall \tau\in \mathcal{T}_{h_D}\}.
\end{align*}
We define $\bm{v}_2\in \bm{V}_h^0$ such that all degrees of freedom associated with (VD1) are given by
\begin{align}
\psi_e(\bm{v}_2)=\int_e \bm{v}_2\cdot\bm{n} p_e\;ds=h_e^{-1/2}\int_e\frac{|[q]|^{3/2}}{[q]}p_e\;ds\quad \forall p_e\in P_k(e)\label{eq:defpsi}
\end{align}
for all $e\in \mathcal{F}_{p,D}$. We take $p_e = [q]$ in \eqref{eq:defpsi} to get
\begin{align*}
\int_e \bm{v}_2\cdot\bm{n}[q]\;ds=h_e^{-1/2}\int_e|[q]|^{3/2}\;ds.
\end{align*}
The definition of $a_D$ yields
\begin{align}
a_D(\bm{v}_2,q)=\sum_{e\in \mathcal{F}_{p,D}} h_e^{-1/2}\int_e|[q]|^{3/2}\;ds.\label{eq:aD2}
\end{align}
Taking $p_e=P_\partial(\bm{v}_2\cdot\bm{n}|\bm{v}_2\cdot\bm{n}|)$ in \eqref{eq:defpsi} and using \eqref{eq:pi}
\begin{align*}
\int_e\bm{v}_2\cdot\bm{n}p_e\;ds&=
\int_e\bm{v}_2\cdot\bm{n}(\bm{v}_2\cdot\bm{n}|\bm{v}_2\cdot\bm{n}|)\;ds=\int_e |\bm{v}_2\cdot\bm{n}|^3\;ds=
h_e^{-1/2}\int_e\frac{|[q]|^{3/2}}{[q]}P_{\partial}(\bm{v}_2\cdot\bm{n}|\bm{v}_2\cdot\bm{n}|)\;ds\\
&\leq h_e^{-1/2}(\int_e|\bm{v}_2\cdot\bm{n}|^{3}\;ds)^{2/3}(\int_e|[q]|^{3/2}\;ds)^{1/3}.
\end{align*}
Therefore, we can get
\begin{align*}
\int_e |\bm{v}_2\cdot\bm{n}|^3\;ds\leq C h_e^{-3/2}\int_e |[q]|^{3/2}\;ds.
\end{align*}
%\begin{align*}
%\|\bm{v}_2\cdot\bm{n}\|_{0,3,e}\leq C h_e^{-1/2}\|q\|_{0,3/2,e}^{1/2},
%\end{align*}
%which is
%\begin{align*}
%\|\bm{v}_2\cdot\bm{n}\|_{0,3,e}^3\leq C h_e^{-3/2}\|q\|_{0,3/2,e}^{3/2}.
%\end{align*}

On the other hand, norm equivalence yields
\begin{align*}
\int_\tau |\bm{v}_2|^3\;dx\leq C \sum_{i=1}^2 h_{e_i} \int_{e_i} |\bm{v}_2\cdot\bm{n}|^3ds,
\end{align*}
where $e_i,i=1,2$ are the two edges on $\partial \tau$ that belong to $\mathcal{F}_{p,D}$.

Therefore
\begin{align*}
\|\bm{v}_2\|_{X_D'}^3=\int_{\Omega_D}|\bm{v}_2|^3\;dx+\sum_{e\in \mathcal{F}_{p,D}}h_e\int_e |\bm{v}_2\cdot\bm{n}|^3\;ds\leq C \sum_{e\in \mathcal{F}_{p,D}}h_e^{-1/2}\int_e|[q]|^{3/2}\;ds.
\end{align*}
Finally, we define $\bm{v}=\bm{v}_1+\bm{v}_2$, then we can get
\begin{align*}
\|\bm{v}\|_{X_D'}^3\leq \|\bm{v}_1\|_{X_D'}^3+\|\bm{v}_2\|_{X_D'}^3\leq C \|q\|_{Z_D}^{3/2}.
\end{align*}
By \eqref{eq:aD} and \eqref{eq:aD2}, we have
\begin{align*}
a_D(\bm{v},q)=a_D(\bm{v}_1,q)+a_D(\bm{v}_2,q)\geq C \|q\|_{Z_D}^{3/2}.
\end{align*}

\end{proof}

The following discrete inf-sup conditions are already proved in \cite{ChungWave09,KimChungLee}, and we save them for later use.
\begin{align}
\inf_{q\in P_h}\sup_{\bm{v}\in [U_{h_S}]^2}\frac{b_S(\bm{v},q)}
{\|\bm{v}\|_h\|q\|_{L^2(\Omega_S)}}&\geq C\label{inf-sup-bh},\\
\inf_{\bm{\tau}_S\in[\bm{V}_{h_S}]^2}\sup_{\bm{v}_S\in [U_{h_S}]^2} \frac{a_S^*(\bm{v}_S,\bm{\tau}_S)}{\|\bm{\tau}_S\|_{Z_S'}\|\bm{v}_S\|_{L^2(\Omega_S)}}&\geq C,\label{eq:inf-supaS}\\
\inf_{\bm{v}_S\in [U_{h_S}]^2}\sup_{\bm{\tau}_S\in[\bm{V}_{h_S}]^2} \frac{a_S(\bm{\tau}_S,\bm{v}_S)}{\|\bm{\tau}_S\|_{X_S'}\|\bm{v}_S\|_{h}}&\geq C.\label{eq:inf-supaS2}
\end{align}
In addition, integration by parts reveals the following adjoint properties
\begin{equation}
\begin{split}
a_D^*(q_D,\bm{v}_D)&=a_D(\bm{v}_D,q_D)\quad \forall (q_D,\bm{v}_D)\in U_{h_D}\times \bm{V}_{h_D},\\
a_S^*(\bm{v}_S,\bm{\tau}_S)&=a_S(\bm{\tau}_S,\bm{v}_S)\quad \forall (\bm{v}_S,\bm{\tau}_S)\in [U_{h_S}]^2\times [\bm{V}_{h_S}]^2,\\
b_S^*(q,\bm{v})&=b_S(\bm{v},q)\qquad \forall (\bm{v},q)\in [U_{h_S}]^2\times P_h.
\end{split}
\label{eq:discreteaD}
\end{equation}

\section{Error analysis}\label{sec:error}
In this section, we first prove the existence and uniqueness of the numerical solution by using the monotone property of the nonlinear operator and the inf-sup conditions for the bilinear forms involved, in addition, the discrete trace inequality is the crux for the stability of pressure in the Stokes region. Then, we present the convergence estimates for all the variables measured in $L^2$-norm.

For the stability and convergence analysis, we also need some special properties of the operator $A$. First we have (cf. \cite{Girault08})
\begin{align}
\forall \bm{v},\bm{w}, \quad |A(\bm{v})-A(\bm{w})|\leq \frac{\mu}{\rho}\|K^{-1}\||\bm{v}-\bm{w}|+
\frac{\beta}{\rho}|\bm{v}-\bm{w}|(|\bm{v}|+|\bm{w}|)
.\label{continuity}
\end{align}
Next, we state the following three lemmas, which will be employed frequently in the stability and convergence analysis, and interested readers can refer to \cite{Girault08} for more details.
\begin{lemma}\label{lemma:monotone}
For fixed $\bm{u}_\ell\in L^3(\Omega_D)^2$, the mapping $\bm{u}\rightarrow A(\bm{u}+\bm{u}_\ell)$ defined by (\ref{A}) is monotone from $L^3(\Omega_D)^2$ into $L^{3/2}(\Omega_D)^2$:
\begin{align*}
\forall \bm{u},\bm{v}\in L^3(\Omega_D)^2, \int_{\Omega_D} (A(\bm{u}+\bm{u}_\ell)-A(\bm{v}+\bm{u}_\ell))\cdot (\bm{u}-\bm{v})\;dx\geq \frac{\mu}{\rho}\lambda_{\min} \|\bm{u}-\bm{v}\|_{L^2(\Omega_D)}^2,
\end{align*}
where $\lambda_{\min}>0$ denotes the smallest eigenvalue of $K$ over $\Omega_D$.

\end{lemma}

\begin{lemma}\label{lemma:coercivity}
For fixed $\bm{u}_l\in L^3(\Omega_D)^2$, the mapping $\bm{u}\rightarrow A(\bm{u}+\bm{u}_l)$ defined by (\ref{A}) is coercive in $L^3(\Omega_D)^2$
\begin{align*}
\lim_{\|\bm{u}\|_{L^3(\Omega_D)}\rightarrow \infty}\Big(\frac{1}{\|\bm{u}_h\|_{L^3(\Omega_D)}}\int_{\Omega_D} A(\bm{u}+\bm{u}_l)\cdot\bm{u}\;dx\Big)=\infty.
\end{align*}

\end{lemma}

\begin{lemma}\label{lemma:hemicontinuity}
The mapping $A$ is hemi-continuous in $L^3(\Omega_D)^2$; for fixed $\bm{u}_l,\bm{u}$ and $\bm{v}$ in $L^3(\Omega_D)^2$, the mapping
\begin{align*}
t\rightarrow \int_{\Omega_D} A(\bm{u}_l+\bm{u}+t\bm{v})\cdot\bm{v}\;dx
\end{align*}
is continuous from $\mathbb{R}$ into $\mathbb{R}$.

\end{lemma}

Next, we aim to showing the unique solvability of the numerical approximations to \eqref{eq:discrete1} and the convergence estimates. For this purpose, we need to prove some novel discrete trace inequality and the generalized Poincar\'{e}-Friedrichs inequality as stated in Lemmas~\ref{lemma:trace} and \ref{lemma:poincarev}. For the ease of later analysis, we define the following nonconforming $P_1$ finite element space
\begin{align*}
M_h=\{q\in L^2(\Omega_D);q\mid_\tau\in P_1(\tau)\;\forall \tau\in \mathcal{T}_{h_D},\int_e[q]\;ds=0\; \forall e\in \mathcal{F}_{u,D}^0; q\mid_{\Gamma_D}=0\}
\end{align*}
and
\begin{align*}
H^1_0(\mathcal{T}_{h_D})=\{z\in L^2(\Omega_D); z\mid_\tau\in H^1(\tau)\; \forall \tau\in \mathcal{T}_{h_D}; z\mid_{\Gamma_D}=0\}.
\end{align*}

The next two lemmas (cf. \cite{Girault08}) will be exploited in the proof of our discrete trace inequality and generalized Poincar\'{e}-Friedrichs inequality.
\begin{lemma}\label{lemma:L2g}
For any pair of numbers $p\geq 1$ and $q>1$ such that
\begin{align*}
p\leq p_0\;\mbox{with}\; \frac{1}{p_0}=\frac{1}{q}-\frac{1}{2}\geq 0,
\end{align*}
there exists a constant $C$, independent of the meshsize such that
\begin{align*}
\|z_h\|_{L^p(\Omega_D)}\leq C (\sum_{\tau\in \mathcal{T}_{h_D}}\|\nabla z_h\|_{L^q(\tau)}^q)^{1/q}\quad \forall z_h\in M_h.
\end{align*}

\end{lemma}
\begin{lemma}\label{lemma:traced}
For any pair of real numbers $p\geq 1$ and $q>1$ such that
\begin{align*}
p\leq p_0\;\mbox{with}\; \frac{1}{p_0}=\frac{2}{q}-1,
\end{align*}
there exists a constant $C$, independent of the meshsize such that
\begin{align*}
\|z_h\|_{L^p(\partial \Omega_D)}\leq C (\sum_{\tau\in \mathcal{T}_{h_D}}\|\nabla z_h\|_{L^q(\tau)}^q)^{1/q}\quad \forall z_h\in M_h.
\end{align*}

\end{lemma}

Now we can prove the trace inequality, which can be used in the proof of the stability and convergence estimates.

\begin{lemma}(discrete trace inequality).\label{lemma:trace}
The following estimate holds
\begin{align*}
\|z_h\|_{L^2(\Gamma)}\leq C\|z_h\|_{Z_D} \quad \forall z_h\in U_{h_D}.
\end{align*}
%and
%\begin{align*}
%\|\bm{v}_h\|_{0,\Gamma}\leq C\|\bm{v}_h\|_{h}\quad \forall \bm{v}_h\in [U_{h_S}]^2.
%\end{align*}

\end{lemma}

\begin{proof}

Following \cite{Brenner}, we define $\pi:H^1_0(\mathcal{T}_{h_D})\rightarrow M_h$ by
\begin{align*}
(\pi\xi)(m)=\frac{1}{h_e}\int_e \{\xi\}\;ds\quad \forall e\in \mathcal{F}_{D},
\end{align*}
where $m$ is the middle point of the edge $e$ and $\{\xi\}$ is the average of the traces from the two sides of $e$, for $e\subset \Gamma$, $\{\xi\}\mid_e=\xi\mid_e$.

Let $\Pi_\tau: H^1(\tau)\rightarrow P_1(\tau), \forall \tau\in \mathcal{T}_{h_D}$ be the local interpolation operator defined by
\begin{align*}
(\Pi_\tau\xi)(m)=\frac{1}{h_e}\int_e \xi\;ds\quad \forall e\subset \partial \tau.
\end{align*}
We have
\begin{align*}(\pi\xi-\Pi_\tau\xi)(m)=
  \begin{cases}
    \frac{1}{2h_e}\int_e [\xi]\;ds \quad \forall e\subset \partial \tau\backslash(\Gamma_D\cup \Gamma),\\
    0\qquad \forall e\subset \partial \tau\cap (\Gamma_D\cup \Gamma). \\
\end{cases}
\end{align*}

It is well known that the following estimates hold for the local interpolation operator (cf. \cite{Ciarlet78})
\begin{align*}
\|\xi-\Pi_\tau \xi\|_{L^2(\tau)}&\leq C h_\tau^{2/3}|\xi|_{W^{1,3/2}(\tau)}\quad\forall \xi\in W^{1,3/2}(\tau), \forall \tau\in \mathcal{T}_{h_D},\\
\|\nabla (\xi-\Pi_\tau \xi)\|_{L^2(\tau)}&\leq C h_\tau^{-1/3}|\xi|_{W^{1,3/2}(\tau)}\quad\forall  \xi\in W^{1,3/2}(\tau), \forall \tau\in \mathcal{T}_{h_D},
\end{align*}
which can be combined with the continuous trace inequality yields
\begin{align*}
\|\xi-\Pi_\tau \xi\|_{L^2(e)}&\leq C\Big( h_\tau^{-1/2}\|\xi-\Pi_\tau \xi\|_{L^2(\tau)}+h_\tau^{1/2}\|\nabla (\xi-\Pi_\tau \xi)\|_{L^2(\tau)}\Big)\\
&\leq C h_\tau^{1/6}|\xi|_{W^{1,3/2}(\tau)}\quad \forall  \xi\in W^{1,3/2}(\tau), \forall e\subset \partial \tau.
\end{align*}
It follows from the triangle inequality that
\begin{align*}
\|z_h\|_{L^2(\Gamma)}
&\leq  \|z_h-\Pi_\tau z_h\|_{L^2(\Gamma)}+\|\Pi_\tau z_h-\pi z_h\|_{L^2(\Gamma)}+\|\pi z_h\|_{L^2(\Gamma)}\\
&\leq C \Big(\Big(\sum_{\tau\in \mathcal{T}_{h_D}}\int_\tau |\nabla z_h|^{3/2}\;dx\Big)^{2/3}+\|\pi z_h\|_{L^2(\Gamma)}\Big).
\end{align*}
Taking $p=2$, $q=3/2$ in Lemma~\ref{lemma:traced} yields
\begin{align*}
\|\pi z_h\|_{L^2(\Gamma)}\leq C \Big(\sum_{\tau\in \mathcal{T}_{h_D}}\|\nabla (\pi z_h)\|_{L^{3/2}(\tau)}^{3/2}\Big)^{2/3}.
\end{align*}
Thus
\begin{align*}
\|z_h\|_{L^2(\Gamma)}\leq C\|z_h\|_{Z_D}.
\end{align*}
\end{proof}

Proceeding analogously, we can prove
\begin{corollary}\label{coro:trace}
The following estimate holds
\begin{align*}
\|\bm{v}_h\|_{L^2(\Gamma)}\leq C \|\bm{v}_h\|_h\quad \forall \bm{v}_h\in [U_{h_S}]^2.
\end{align*}

\end{corollary}

Next, we prove the following generalized Poincar\'{e}-Friedrichs inequality, which will be useful for the subsequent analysis.
\begin{lemma}(generalized Poincar\'{e}-Friedrichs inequality). \label{lemma:poincarev}
The following estimate holds
\begin{align*}
\|q_h\|_{L^2(\Omega_D)}\leq C \|q_h\|_{Z_D}\quad \forall q_h\in U_{h_D}.
\end{align*}

\end{lemma}

\begin{proof}
First, the standard estimates for the local interpolation operator (cf. \cite{Ciarlet78}) imply
\begin{align*}
\|q_h-\Pi_\tau q_h\|_{L^2(\tau)}\leq C h_\tau^{2/3}|q_h|_{W^{1,3/2}(\tau)}.
\end{align*}
Proceeding analogously to (2.5) in \cite{Brenner}, we can obtain
\begin{align*}
\|\Pi_\tau q_h-\pi q_h\|_{L^2(\tau)}&\leq C |\sum_{e\subset \partial \tau\backslash (\Gamma_D\cup \Gamma)}\int_e [q_h]\;ds|\\
&\leq C \Big(\sum_{e\subset \partial \tau\backslash (\Gamma_D\cup \Gamma)}\int_e |[q_h]|^{3/2}\;ds\Big)^{2/3}\Big(\sum_{e\subset \partial \tau\backslash (\Gamma_D\cup \Gamma)}\int_e 1^3\;ds\Big)^{1/3}\\
&\leq C h_e^{1/3}\Big(\sum_{e\subset \partial \tau\backslash (\Gamma_D\cup \Gamma)}\int_e |[q_h]|^{3/2}\;ds\Big)^{2/3}.
\end{align*}
Finally, taking $p=2, q=3/2$ in Lemma~\ref{lemma:L2g} leads to
\begin{align*}
\|q_h\|_{L^2(\Omega_D)}&\leq \|q_h-\pi q_h\|_{L^2(\Omega_D)}+\|\pi q_h\|_{L^2(\Omega_D)}\\
&\leq C\Big(\Big(\sum_{\tau\in \mathcal{T}_{h_D}}\|q_h-\Pi_\tau q_h\|_{L^2(\tau)}^2\Big)^{1/2}+\Big(\sum_{\tau\in \mathcal{T}_{h_D}}\|\Pi_\tau q_h-\pi q_h\|_{L^2(\tau)}^2\Big)^{1/2}\\
&\;+\Big(\sum_{\tau\in \mathcal{T}_{h_D}}\|\nabla (\pi q_h)\|_{L^{3/2}(\tau)}^{3/2}\Big)^{2/3}\Big)\\
&\leq C \|q_h\|_{Z_D}.
\end{align*}

\end{proof}

\begin{theorem}\label{thm:unique-con}
(existence and uniqueness of the continuous problem)

The continuous problem \eqref{eq:weak} has a unique solution $(\bm{\sigma}_S, \bm{u}_S,p_S)\in L^2(\Omega_S)^{2\times 2}\times U_S^0\times L^2(\Omega_S)$ and $(\bm{u}_D,p_D)\in L^3(\Omega_D)^2\times \Sigma_0$.
\end{theorem}

\begin{proof}
We define
\begin{align*}
X=\{\bm{v}_S\in U_S^0,\; (\nabla \cdot \bm{v}_S,q_S)_{\Omega_S}=0\quad \forall q_S\in L^2(\Omega_S)\}.
\end{align*}
Since $(\nabla \cdot \bm{v}_S,q_S)_{\Omega_S}$ satisfies the following inf-sup condition (cf. \cite{Braess07})
\begin{align*}
\inf_{q_S\in L^2(\Omega_S)}\sup_{\bm{v}_S\in H^1(\Omega_S)^2}\frac{(\nabla \cdot \bm{v}_S,q_S)_{\Omega_S}}{\|q_S\|_{L^2(\Omega_S)}\|\bm{v}_S\|_{H^1(\Omega_S)}}\geq C.
\end{align*}
Problem \eqref{eq:weak} is equivalent to: find $(\bm{\sigma}_{S},\bm{u}_{S},p_{S})\in  L^2(\Omega_S)^{2\times 2}\times X\times L^2(\Omega_S)$ and $(\bm{u}_{D},p_{D})\in L^3(\Omega_D)^2\times \Sigma_0$ such that
\begin{equation}
\begin{split}
&-(\bm{\sigma}_S,\nabla \bm{v}_S)_{\Omega_S}+\langle p_D, \bm{v}_S\cdot\bm{n}_{S}\rangle_\Gamma+G\langle \bm{u}_S\cdot\bm{t}, \bm{v}_S\cdot\bm{t}\rangle_\Gamma-(\bm{u}_D,\nabla q_D)_{\Omega_D}\\
&-\langle \bm{u}_S\cdot\bm{n}_{S},q_D\rangle_\Gamma=
(f_D,q_D)_{\Omega_D}+(\bm{f}_S,\bm{v}_S)_{\Omega_S},\\
&(A(\bm{u}_D),\bm{v}_D)_{\Omega_D}+(\nabla p_D,\bm{v}_D)_{\Omega_D}+\nu^{-1}(\bm{\sigma}_S,\bm{\tau}_S)_{\Omega_S}+(\nabla \bm{u}_S, \bm{\tau}_S)_{\Omega_S}=(\bm{g},\bm{v}_D)_{\Omega_D},\\
&\forall (\bm{\tau}_S,\bm{v}_S)\in L^2(\Omega_S)^{2\times 2}\times X,\quad (\bm{v}_D,q_D)\in L^3(\Omega_D)^2\times \Sigma_0.
\end{split}
\label{eq:new-continuous}
\end{equation}

In order to show the unique solvability of \eqref{eq:new-continuous}, we first need to show the well-posedness of the following problem: given $(\bm{u}_1,p_1)\in X\times \Sigma_0$, find $(\bm{\sigma}_0,\bm{u}_0)\in L^2(\Omega_S)^{2\times 2}\times L^3(\Omega_D)^2$ such that
\begin{equation}
\begin{split}
&\nu^{-1}(\bm{\sigma}_0,\bm{\tau}_S)+(A(\bm{u}_0),\bm{v}_D)=(\bm{g}_D,\bm{v}_D)
-(\nabla p_1,\bm{v}_D)_{\Omega_D}-(\nabla \bm{u}_1,\bm{\tau}_S)_{\Omega_S}\\
&\quad \forall (\bm{\tau}_S,\bm{v}_D)\in L^2(\Omega_S)^{2\times 2}\times L^3(\Omega_D)^2.
\end{split}
\label{eq:11}
\end{equation}

Since $A$ satisfies  Lemma~\ref{lemma:monotone}-Lemma~\ref{lemma:hemicontinuity}, for given $(\bm{u}_1,p_1)\in X\times \Sigma_0$, there exists a unique $(\bm{\sigma}_0,\bm{u}_0)\in L^2(\Omega_S)^{2\times 2}\times L^3(\Omega_D)^2$ of \eqref{eq:11}. From this fact, for each $S=(\bm{u}_1,p_1)\in X\times \Sigma_0$, it is now possible to define $(\bm{\sigma}(S),\bm{u}(S))$ as the unique element in $ L^2(\Omega_S)^{2\times 2}\times L^3(\Omega_D)^2$ such that
\begin{align*}
&\nu^{-1}(\bm{\sigma}(S),\bm{\tau}_S)+(A(\bm{u}(S)),\bm{v}_D)=(\bm{g}_D,\bm{v}_D)
-(\nabla p_1,\bm{v}_D)_{\Omega_D}-(\nabla \bm{u}_1,\bm{\tau}_S)_{\Omega_S},\\
&\forall (\bm{\tau}_S,\bm{v}_D)\in L^2(\Omega_S)^{2\times 2}\times L^3(\Omega_D)^2.
\end{align*}
%which satisfies
%\begin{align*}
%&\nu^{-1}(\bm{\sigma}(S_1)-\bm{\sigma}(S_2),\bm{\tau}_S)+
%(A(\bm{u}(S_1))-A(\bm{u}(S_2)),\bm{v}_D)=a_S^*(\bm{u}_1-\bm{u}_2,\bm{\tau}_S)
%+a_D^*(p_1-p_2,\bm{v}_D)\\
%&\hspace{2cm} \forall (\bm{\tau}_S,\bm{v}_D)\in [\bm{V}_{h_S}]^2\times \bm{V}_{h_D}.
%\end{align*}
Then problem \eqref{eq:new-continuous} is equivalent to: find $S=(\bm{u}_{S},p_{D})\in X\times \Sigma_0$ such that
\begin{align*}
&[\mathbb{T}(S),\beta]:=-( \bm{\sigma}_S(S),\nabla \bm{v}_S)_{\Omega_S}-(\bm{u}_D(S),\nabla q_D)_{\Omega_D}+\langle p_{D},\bm{v}_S\cdot\bm{n}_S\rangle_\Gamma+G\langle \bm{u}_{S}\cdot\bm{t},\bm{v}_S\cdot\bm{t}\rangle_\Gamma\\
&\;-\langle \bm{u}_{S}\cdot\bm{n}_S,q_D\rangle_\Gamma=(f_D,q_D)_{\Omega_D}
+(\bm{f}_S,\bm{v}_S)_{\Omega_S},\\
&\forall \beta=(\bm{v}_S,q_D)\in  X\times \Sigma_0.
\end{align*}
In addition, we define a operator $\mathbb{C}$ such that $[\mathbb{C}(S),\beta]:=\langle p_{D},\bm{v}_S\cdot\bm{n}\rangle_\Gamma+G\langle \bm{u}_{S}\cdot\bm{t},\bm{v}_S\cdot\bm{t}\rangle_\Gamma-\langle \bm{u}_{S}\cdot\bm{n}_S,q_D\rangle_\Gamma\;\forall\beta=(\bm{v}_S,q_D)\in  X\times \Sigma_0$. It is easy to see that $\mathbb{C}$ is positive semidefinite, furthermore, the following inf-sup conditions hold (cf. \cite{Girault08})
\begin{align*}
\inf_{q\in W^{1,3/2}(\Omega_D)}\sup_{\bm{v}\in L^3(\Omega_D)^2}\frac{-(\nabla q,\bm{v})_{\Omega_D}}{\|\nabla q\|_{L^{3/2}(\Omega_D)}\|\bm{v}\|_{L^3(\Omega_D)}}&\geq C,\\
\inf_{\bm{v}\in H^1(\Omega_S)^2}\sup_{\bm{\tau}\in L^2(\Omega_S)^{2\times 2}}\frac{-(\nabla \bm{v},\bm{\tau})_{\Omega_S}}{\|\nabla \bm{v}\|_{L^2(\Omega_S)}\|\bm{\tau}\|_{L^2(\Omega_S)}}&\geq C.
\end{align*}
Then proceeding as in \cite{Almonacid19}, we can show that $\mathbb{T}$ is injective, continuous, monotone, bounded and coercive, which means that $\mathbb{T}$ is bijective.

Therefore, there exists a unique solution to \eqref{eq:weak}.
\end{proof}

\begin{theorem}(existence and uniqueness of the discrete problem).
The discrete problem \eqref{eq:discrete1} admits a unique solution $(\bm{\sigma}_{S,h},\bm{u}_{S,h},p_{S,h})\in [\bm{V}_{h_S}]^2\times [U_{h_S}]^2\times P_h$ and $(\bm{u}_{D,h},p_{D,h})\in \bm{V}_{h_D}\times U_{h_D}$. Moreover, the following estimates hold
\begin{align*}
&\|\bm{\sigma}_{S,h}\|_{L^2(\Omega_S)}^2+\|\bm{u}_{S,h}\|_{L^2(\Omega_S)}^2
+\|\bm{u}_{D,h}\|_{L^2(\Omega_D)}^2+\|\bm{u}_{D,h}\|_{L^3(\Omega_D)}^3\\
&\leq C\Big(\|f_D\|_{L^2(\Omega_D)}^2+\|f_D\|_{L^2(\Omega_D)}^3
+\|\bm{g}_D\|_{L^2(\Omega_D)}^2+\|\bm{f}_S\|_{L^2(\Omega_S)}^2\Big),\\
&\|p_{D,h}\|_{Z_D}+\|p_{D,h}\|_{L^2(\Omega_D)}+\|p_{S,h}\|_{L^2(\Omega_S)}\leq C \Big(\|f_D\|_{L^2(\Omega_D)}+\|f_D\|_{L^2(\Omega_D)}^2+\|\bm{g}_D\|_{L^2(\Omega_D)}\\
&\hspace{7cm}+\|\bm{g}_D\|_{L^2(\Omega_D)}^2+\|\bm{f}_S\|_{L^2(\Omega_S)}
+\|\bm{f}_S\|_{L^2(\Omega_S)}^2\Big).
\end{align*}
\end{theorem}

\begin{proof}

By the discrete inf-sup condition (cf. Lemma~\ref{lemma:infsupa}, \eqref{inf-sup-bh} and \eqref{eq:inf-supaS}) and Lemma~\ref{lemma:monotone}-Lemma~\ref{lemma:hemicontinuity}, the existence and uniqueness of \eqref{eq:discrete1} can be proved similarly to Theorem~\ref{thm:unique-con}. We omit it for simplicity.

On the other hand, taking $\bm{\tau}_S=\bm{\sigma}_{S,h}$, $\bm{v}_S=\bm{u}_{S,h}$ and $\bm{v}_D=\bm{u}_{D,h}$, $q_D=p_{D,h}$ in \eqref{eq:discrete1} and adding the resulting equations yield
\begin{align*}
\nu^{-1}\|\bm{\sigma}_{S,h}\|_{L^2(\Omega_S)}^2+G\langle \bm{u}_{S,h}\cdot\bm{t},\bm{u}_{S,h}\cdot\bm{t}\rangle_\Gamma +(A(\bm{u}_{D,h}),\bm{u}_{D,h})_{\Omega_D}
=(f_D,p_{D,h})_{\Omega_D}+(\bm{f}_S,\bm{u}_{S,h})_{\Omega_S}.
\end{align*}
The definition of $A$ (cf. (\ref{A})) gives
\begin{align*}
(A(\bm{u}_{D,h}), \bm{u}_{D,h})_{\Omega_D}\geq C\Big(\|\bm{u}_{D,h}\|_{L^2(\Omega_D)}^2+\|\bm{u}_{D,h}\|_{L^3(\Omega_D)}^3\Big).
\end{align*}
The discrete inf-sup condition \eqref{eq:inf-supaS2} implies
\begin{align*}
\|\bm{u}_{S,h}\|_{h}\leq C \sup_{\bm{\tau}_S\in [\bm{V}_{h_S}]^2 }\frac{a_S^*(\bm{u}_{S,h},\bm{\tau}_S)}{\|\bm{\tau}_S\|_{L^2(\Omega_S)}}
=C\sup_{\bm{\tau}_S\in [\bm{V}_{h_S}]^2 }\frac{(\bm{\sigma}_{S,h},\bm{\tau}_S)_{\Omega_S}}{\|\bm{\tau}_S\|_{L^2(\Omega_S)}}\leq C \|\bm{\sigma}_{S,h}\|_{L^2(\Omega_S)}.
\end{align*}
Lemma~\ref{lemma:insup3} and the adjoint property \eqref{eq:discreteaD} imply the existence of $p_{D,h}$ satisfying
\begin{equation}
\begin{split}
\|p_{D,h}\|_{Z_D}&\leq C \sup_{\bm{v}_D\in \bm{V}_{h_D}}\frac{a_D^*(p_{D,h}, \bm{v}_D)}{\|\bm{v}_D\|_{L^3(\Omega_D)}}
=C\sup_{\bm{v}_D\in \bm{V}_{h_D}}\frac{(A(\bm{u}_{D,h}),\bm{v}_D)_{\Omega_D}-(\bm{g}_D,\bm{v}_D)_{\Omega_D}}
{\|\bm{v}_D\|_{L^3(\Omega_D)}}\\
&\leq C\Big(\|\bm{u}_{D,h}\|_{L^2(\Omega_D)}+\|\bm{u}_{D,h}\|_{L^3(\Omega_D)}^2
+\|\bm{g}_D\|_{L^2(\Omega_D)}\Big).
\end{split}
\label{phZ}
\end{equation}

Thus, the preceding arguments, Lemma~\ref{lemma:poincarev} and the discrete Poincar\'{e}-Friedrichs inequality (cf. \cite{Brenner}) imply that
\begin{align*}
&\|\bm{\sigma}_{S,h}\|_{L^2(\Omega_S)}^2+\|\bm{u}_{D,h}\|_{L^2(\Omega_D)}^2
+\|\bm{u}_{D,h}\|_{L^3(\Omega_D)}^3\\
&\leq C(\|f_D\|_{L^2(\Omega_D)}\|p_{D,h}\|_{L^2(\Omega_D)}+
\|\bm{f}_S\|_{L^2(\Omega_S)}\|\bm{u}_{S,h}\|_{L^2(\Omega_S)})\\
&\leq  C(\|f_D\|_{L^2(\Omega_D)}\|p_{D,h}\|_{Z_D}
+\|\bm{f}_S\|_{L^2(\Omega_S)}\|\bm{u}_{S,h}\|_{L^2(\Omega_S)})\\
&\leq C \Big(\|f_D\|_{L^2(\Omega_D)}(\|\bm{u}_{D,h}\|_{L^2(\Omega_D)}
+\|\bm{u}_{D,h}\|_{L^3(\Omega_D)}^2
+\|\bm{g}_D\|_{L^2(\Omega_D)})
+\|\bm{f}_S\|_{L^2(\Omega_S)}\|\bm{\sigma}_{S,h}\|_{L^2(\Omega_S)}\Big).
\end{align*}
Therefore
\begin{align*}
&\|\bm{\sigma}_{S,h}\|_{L^2(\Omega_S)}^2+\|\bm{u}_{D,h}\|_{L^2(\Omega_D)}^2
+\|\bm{u}_{D,h}\|_{L^3(\Omega_D)}^3\\
&\leq C \Big(\|f_D\|_{L^2(\Omega_D)}^2+\|f_D\|_{L^2(\Omega_D)}^3
+\|\bm{g}_D\|_{L^2(\Omega_D)}^2+\|\bm{f}_S\|_{L^2(\Omega_S)}^2\Big).
\end{align*}

%Then, we have from discrete Poincar\'{e} inequality (cf. \cite{Brenner})
%\begin{align*}
%\|\bm{u}_h\|_{0,2}^2+\|\bm{u}_h\|_{0,3}^3&\leq
%C (\|\bm{g}\|_{0,2}\|\bm{u}_h\|_{0,2}+\|f\|_{0,2}\|p_h\|_{0,2})\\
%&\leq C (\|\bm{g}\|_{0,2}\|\bm{u}_h\|_{0,2}
%+\|f\|_{0,2}\|p_h\|_{Z}).
%\end{align*}
%Therefore
%\begin{align}
%\|\bm{u}_h\|_{0,2}+\|\bm{u}_h\|_{0,3}\leq C (\|\bm{g}\|_{0,2}+\|f\|_{0,2}).
%\end{align}

%Then, we have from H\"{o}lder's inequality and (\ref{phZ})
%\begin{align*}
%\|\bm{u}_h\|_{0,2}^2+\|\bm{u}_h\|_{0,3}^3\leq C(\|\bm{g}\|_{0,2}^2
%+\|f\|_{0,2}^{2}+\|f\|_{0,2}^{3}).
%\end{align*}

%\Red{
%By the generalized discrete Poincar\'{e} inequality, we have
%\begin{align*}
%\|p_{D,h}\|_{L^2(\Omega_D)}\leq C \|p_{D,h}\|_{Z_D}.
%\end{align*}
%}
Finally, we have from Lemma~\ref{lemma:poincarev} and \eqref{phZ}
\begin{align*}
&\|p_{D,h}\|_{L^2(\Omega_D)}\leq C\|p_{D,h}\|_{Z_D}\leq  C\Big(\|\bm{u}_{D,h}\|_{L^2(\Omega_D)}+\|\bm{u}_{D,h}\|_{L^3(\Omega_D)}^2
+\|\bm{g}_D\|_{L^2(\Omega_D)}\Big)\\
&\leq C \Big( \|f_D\|_{L^2(\Omega_D)}+\|f_D\|_{L^2(\Omega_D)}^2+\|\bm{g}_D\|_{L^2(\Omega_D)}
+\|\bm{g}_D\|_{L^2(\Omega_D)}^2+\|\bm{f}_S\|_{L^2(\Omega_S)}
+\|\bm{f}_S\|_{L^2(\Omega_S)}^2\Big).
\end{align*}

%The discrete Poincar\'{e} inequality yields
%\begin{align*}
%\|\bm{u}_{S,h}\|_0\leq C \|\bm{u}_{S,h}\|_Z.
%\end{align*}
The norm equivalence \eqref{scaling}, the discrete inf-sup condition \eqref{inf-sup-bh}, the first equation of \eqref{eq:discrete1}, Lemma~\ref{lemma:trace}, and Corollary~\ref{coro:trace} imply
\begin{align*}
&\|p_{S,h}\|_{L^2(\Omega_S)}\leq C \sup_{\bm{v}_S\in [U_{h_S}]^2} \frac{b_S^*(p_{S,h},\bm{v}_S)}{\|\bm{v}_S\|_h}\\
&=C\sup_{\bm{v}_S\in [U_{h_S}]^2 }\frac{(\bm{f}_S,\bm{v}_S)_{\Omega_S}-a_S(\bm{\sigma}_{S,h},\bm{v}_S)-\langle p_{D,h},\bm{v}_S\cdot\bm{n}_{S}\rangle_\Gamma-G \langle \bm{u}_{S,h}\cdot\bm{t},\bm{v}_S\cdot\bm{t}\rangle_\Gamma}{\|\bm{v}_S\|_h}\\
&\leq C \Big(\|\bm{f}_S\|_{L^2(\Omega_S)}+\|\bm{\sigma}_{S,h}\|_{L^2(\Omega_S)}
+\|p_{D,h}\|_{Z_D}+\|\bm{u}_{S,h}\|_{h}\Big)\\
&\leq C \Big(\|f_D\|_{L^2(\Omega_D)}+\|f_D\|_{L^2(\Omega_D)}^{2}+\|\bm{g}_D\|_{L^2(\Omega_D)}
+\|\bm{g}_D\|_{L^2(\Omega_D)}^{2}+\|\bm{f}_S\|_{L^2(\Omega_S)}
+\|\bm{f}_S\|_{L^2(\Omega_S)}^{2}\Big),
\end{align*}
which completes the proof.

\end{proof}

The rest of this section presents the convergence estimates for all the variables. To this end, we first write down the following error equations by performing integration by parts on \eqref{eq:discrete1}
\begin{equation}
\begin{split}
&a_S(\bm{\sigma}_S-\bm{\sigma}_{S,h},\bm{v}_S)+b_S^*(p_S-p_{S,h},\bm{v}_S)
+a_D(\bm{u}_D-\bm{u}_{D,h},q_D)+\langle p_D-p_{D,h},\bm{v}_S\cdot\bm{n}_{S} \rangle_\Gamma\\
&\;+G\langle (\bm{u}_S-\bm{u}_{S,h})\cdot\bm{t},\bm{v}_S\cdot\bm{t}\rangle_\Gamma
-\langle(\bm{u}_S-\bm{u}_{S,h})\cdot\bm{n}_{S},q_D\rangle_\Gamma
=0,\\
&\nu^{-1}(\bm{\sigma}_S-\bm{\sigma}_{S,h},\bm{\tau}_S)_{\Omega_S}
-a_S^*(\bm{u}_S-\bm{u}_{S,h},\bm{\tau}_S)+
(A(\bm{u}_D)-A(\bm{u}_{D,h}),\bm{v}_D)_{\Omega_D}\\
&\;-a_D^*(p_D-p_{D,h},\bm{v}_D)=0, \\
&b_S(\bm{u}_S-\bm{u}_{S,h},q_S)=0,\\
& \forall (\bm{\tau}_S,\bm{v}_S,q_S)\in [\bm{V}_{h_S}]^2\times [U_{h_S}]^2\times P_h, \; (\bm{v}_D,q_D)\in \bm{V}_{h_D}\times U_{h_D}.
\end{split}
\label{eq:error}
\end{equation}

%\begin{lemma}
%
%
%\begin{align*}
%&\|J_h\bm{\sigma}_S-\bm{\sigma}_{S,h}\|_{L^2(\Omega_S)}
%+\|J_h\bm{u}_D-\bm{u}_{D,h}\|_{L^2(\Omega_D)}\leq C \Big(\|\bm{\sigma}_S-J_h\bm{\sigma}_S\|_{L^2(\Omega_S)}+\|p_S-\tilde{p}\|_P\\
%&\;+\|\bm{u}_D-J_h\bm{u}_D\|_{L^2(\Omega_D)}
%+\|\bm{u}_D-J_h\bm{u}_D\|_{L^2(\Omega_D)}(\|\bm{u}_D\|_{L^4(\Omega_D)}
%+\|J_h\bm{u}_D\|_{L^4(\Omega_D)})\Big),\\
%&\|I_hp_D-p_{D,h}\|_{Z_D}\leq C\Big(\|\bm{u}_D-\bm{u}_{D,h}\|_{L^2(\Omega_D)}+\|\bm{u}_D-\bm{u}_{D,h}\|_{L^2(\Omega_D)}
%(\|\bm{u}_{D,h}\|_{0,6}+\|\bm{u}_D\|_{L^6(\Omega_D)})\Big),\\
%&\|p_{S,h}-\pi_h p\|_{L^2(\Omega_S)}\leq C \Big(\|p_S-\tilde{p}\|_P+
%\|J_h\bm{\sigma}_S-\bm{\sigma}_{S,h}\|_{L^2(\Omega_S)}
%+\|I_hp_D-p_{D,h}\|_{Z_D}+\|I_h\bm{u}_S-\bm{u}_{S,h}\|_{Z_S}\Big),\\
%\|I_h\bm{u}_S-\bm{u}_{S,h}\|_h&\leq \|\bm{\sigma}_S-\bm{\sigma}_{S,h}\|_{L^2(\Omega_S)}.
%\end{align*}
%
%
%\end{lemma}

The main result of this section can be stated as follows.
\begin{theorem}\label{thm:con1}
Assume that $\bm{u}_D\in W^{\alpha, 4}(\Omega_D)^2$,$p_D\in W^{\alpha+1,3/2}(\Omega_D)$, $\bm{u}_S\in H^{\alpha+1}(\Omega_S)^2$, $p_S\in H^\alpha(\Omega_S)$, then we can obtain the following error estimates
\begin{align*}
&\|\bm{\sigma}_S-\bm{\sigma}_{S,h}\|_{L^2(\Omega_D)}
+\|\bm{u}_D-\bm{u}_{D,h}\|_{L^2(\Omega_D)}\\
&\quad \leq C h^{\min\{k+1,\alpha\}}\Big(\|\bm{u}_S\|_{H^{\min\{k+2,\alpha+1\}}(\Omega_S)}
+\|p_S\|_{H^{\min\{k+1,\alpha\}}(\Omega_S)}
+\|\bm{u}_D\|_{W^{\min\{k+1,\alpha\},4}(\Omega_D)}\Big),\\
%\|p_D-p_{D,h}\|_{Z_D}&\leq C h^{\alpha}(\|\bm{u}_D\|_{H^\alpha(\Omega_D)}),\\
&\|p_S-p_{S,h}\|_{L^2(\Omega_S)}\\
&\quad\leq C h^{\min\{k+1,\alpha\}}\Big(\|p_S\|_{H^{\min\{k+1,\alpha\}}(\Omega_S)}
+\|\bm{u}_S\|_{H^{\min\{k+2,\alpha+1\}}(\Omega_S)}
+\|\bm{u}_D\|_{W^{\min\{k+1,\alpha\},4}(\Omega_D)}\Big),\\
&\|\bm{u}_S-\bm{u}_{S,h}\|_h+\|p_D-p_{D,h}\|_{Z_D}\\
&\quad\leq C h^{\min\{k,\alpha\}} \Big(\|\bm{u}_S\|_{H^{\min\{k,\alpha\}+1}(\Omega_S)}+\|p_S\|_{H^{\min\{k,\alpha\}}(\Omega_S)}
+\|\bm{u}_D\|_{W^{\min\{k,\alpha\},4}(\Omega_D)}+\|p_D\|_{W^{\min\{k,\alpha\}+1,3/2}(\Omega_D)}\Big).
\end{align*}

\end{theorem}

\begin{proof}
The adjoint property \eqref{eq:discreteaD} and \eqref{eq:interpolationI} imply that
\begin{align*}
a_D^*(I_hp_D-p_{D,h},J_h\bm{u}_D-\bm{u}_{D,h})&=a_D(J_h\bm{u}_D-\bm{u}_{D,h},I_hp_D-p_{D,h})\\
&=a_D(\bm{u}_D-\bm{u}_{D,h},I_hp_D-p_{D,h}),\\
a_S^*(I_h\bm{u}_S-\bm{u}_{S,h},J_h\bm{\sigma}_S-\bm{\sigma}_{S,h})
&=a_S(J_h\bm{\sigma}_S-\bm{\sigma}_{S,h},I_h\bm{u}_S-\bm{u}_{S,h})\\
&=a_S(\bm{\sigma}_S-\bm{\sigma}_{S,h},I_h\bm{u}_S-\bm{u}_{S,h}).
\end{align*}
Recall that we have the following error equations (cf. \eqref{eq:error})
\begin{equation*}
\begin{split}
&a_S(\bm{\sigma}_S-\bm{\sigma}_{S,h},I_h\bm{u}_S-\bm{u}_{S,h})
+b_S^*(p_S-p_{S,h},I_h\bm{u}_S-\bm{u}_{S,h})\\
&\;+a_D(\bm{u}_D-\bm{u}_{D,h},I_hp_D-p_{D,h})+\langle p_D-p_{D,h},(I_h\bm{u}_S-\bm{u}_{S,h})\cdot\bm{n}_{S}\rangle_\Gamma\\
&\;+G\langle (\bm{u}_S-\bm{u}_{S,h})\cdot\bm{t}, (I_h\bm{u}_S-\bm{u}_{S,h})\cdot\bm{t}\rangle_\Gamma-\langle (\bm{u}_S-\bm{u}_{S,h})\cdot\bm{n}_{S},I_hp_D-p_{D,h}\rangle_\Gamma=0
\end{split}
\end{equation*}
and
\begin{equation}
\begin{split}
&\nu^{-1}(\bm{\sigma}_S-\bm{\sigma}_{S,h}, J_h\bm{\sigma}_S-\bm{\sigma}_{S,h})+(A(\bm{u}_D)-A(\bm{u}_{D,h}), J_h\bm{u}_D-\bm{u}_{D,h})_{\Omega_D}\\
&\;-a_S^*(I_h\bm{u}_S-\bm{u}_{S,h}, J_h\bm{\sigma}_S-\bm{\sigma}_{S,h})-a_D^*(I_hp_D-p_{D,h},J_h\bm{u}_D-\bm{u}_{D,h})=0.
\end{split}
\label{eq:error2}
\end{equation}
Since $b_S(\bm{u}_S-\bm{u}_{S,h},q_h)=0\;\forall q_h\in P_h$, we have from the adjoint property \eqref{eq:discreteaD}
\begin{align*}
&a_S(\bm{\sigma}_S-\bm{\sigma}_{S,h},I_h\bm{u}_S-\bm{u}_{S,h})
+b_S^*(p_S-\pi_h p_S,I_h\bm{u}_S-\bm{u}_{S,h})+a_D(\bm{u}_D-\bm{u}_{D,h},I_hp_D-p_{D,h})\\
&\;+\langle p_D-p_{D,h},(I_h\bm{u}_S-\bm{u}_{S,h})\cdot\bm{n}_{S}\rangle_\Gamma+G\langle (\bm{u}_S-\bm{u}_{S,h})\cdot\bm{t}, (I_h\bm{u}_S-\bm{u}_{S,h})\cdot\bm{t}\rangle_\Gamma\\
&\;-\langle (\bm{u}_S-\bm{u}_{S,h})\cdot\bm{n}_{S},I_hp_D-p_{D,h}\rangle_\Gamma=0.
\end{align*}
Thus,
\begin{equation}
\begin{split}
&a_S(\bm{\sigma}_S-\bm{\sigma}_{S,h},I_h\bm{u}_S-\bm{u}_{S,h})
+b_S^*(p_S-\pi_h p_S,I_h\bm{u}_S-\bm{u}_{S,h})+a_D(\bm{u}_D-\bm{u}_{D,h},I_hp_D-p_{D,h})\\
&\;+G\langle (\bm{u}_S-\bm{u}_{S,h})\cdot\bm{t}, (I_h\bm{u}_S-\bm{u}_{S,h})\cdot\bm{t}\rangle_\Gamma=0.
\end{split}
\label{eq:errort}
\end{equation}
We have from \eqref{eq:error2}
\begin{align*}
&\nu^{-1}\|J_h\bm{\sigma}_S-\bm{\sigma}_{S,h}\|_{L^2(\Omega_S)}^2+(A(J_h\bm{u}_D)-A(\bm{u}_{D,h}), J_h\bm{u}_D-\bm{u}_{D,h})_{\Omega_D}\\
&=
\nu^{-1}(J_h\bm{\sigma}_S-\bm{\sigma}_S,J_h\bm{\sigma}_S-\bm{\sigma}_{S,h})_{\Omega_S}
+a_S^*(I_h\bm{u}_S-\bm{u}_{S,h}, J_h\bm{\sigma}_S-\bm{\sigma}_{S,h})\\
&\;+a_D^*(I_hp_D-p_{D,h},J_h\bm{u}_D-\bm{u}_{D,h})+(A(J_h\bm{u}_D)-A(\bm{u}_{D}), J_h\bm{u}_D-\bm{u}_{D,h})_{\Omega_D}.
\end{align*}
The discrete adjoint property \eqref{eq:discreteaD} and \eqref{eq:errort} reveal that
\begin{align*}
&a_S^*(I_h\bm{u}_S-\bm{u}_{S,h}, J_h\bm{\sigma}_S-\bm{\sigma}_{S,h})
+a_D^*(I_hp_D-p_{D,h},J_h\bm{u}_D-\bm{u}_{D,h})\\
&=a_S(\bm{\sigma}_S-\bm{\sigma}_{S,h},I_h\bm{u}_S-\bm{u}_{S,h})
+a_D(\bm{u}_D-\bm{u}_{D,h},I_hp_D-p_{D,h})\\
&=-b_S^*(p_S-\pi_h p_S,I_h\bm{u}_S-\bm{u}_{S,h})-G\langle (I_h\bm{u}_S-\bm{u}_{S,h})\cdot\bm{t}, (I_h\bm{u}_S-\bm{u}_{S,h})\cdot\bm{t}\rangle_\Gamma\\
&\leq C \|p_S-\pi_h p_S\|_P\|I_h\bm{u}_S-\bm{u}_{S,h}\|_{h}.
\end{align*}
The Cauchy-Schwarz inequality and \eqref{continuity} yield
\begin{align*}
&(A(J_h\bm{u}_D)-A(\bm{u}_{D}), J_h\bm{u}_D-\bm{u}_{D,h})_{\Omega_D}\\
&\leq C\Big( \frac{\mu}{\rho}\|K^{-1}\|_{L^\infty(\Omega_D)}
\|J_h\bm{u}_D-\bm{u}_{D,h}\|_{L^2(\Omega_D)}
\|\bm{u}_D-J_h\bm{u}_D\|_{L^2(\Omega_D)}\\
&\;+\frac{\beta}{\rho}\|J_h\bm{u}_D-\bm{u}_{D,h}\|_{L^2(\Omega_D)}
\|\bm{u}_D-J_h\bm{u}_D\|_{L^4(\Omega_D)}(\|\bm{u}_D\|_{L^4(\Omega_D)}
+\|J_h\bm{u}_D\|_{L^4(\Omega_D)})\Big).
\end{align*}
Therefore, the preceding arguments and Lemma~\ref{lemma:monotone} yield
\begin{equation}
\begin{split}
&\|J_h\bm{\sigma}_S-\bm{\sigma}_{S,h}\|_{L^2(\Omega_S)}^2
+\|J_h\bm{u}_D-\bm{u}_{D,h}\|_{L^2(\Omega_D)}^2\\
&\leq C \Big(\|J_h\bm{\sigma}_S-\bm{\sigma}_S\|_{L^2(\Omega_S)}
\|J_h\bm{\sigma}_S-\bm{\sigma}_{S,h}\|_{L^2(\Omega_S)}
+\|p_S-\pi_h p_S\|_P\|I_h\bm{u}_S-\bm{u}_{S,h}\|_{h}\\
&\;+\frac{\mu}{\rho}\|K^{-1}\|_{L^\infty(\Omega_D)}
\|J_h\bm{u}_D-\bm{u}_{D,h}\|_{L^2(\Omega_D)}
\|\bm{u}_D-J_h\bm{u}_D\|_{L^2(\Omega_D)}\\
&\;+\frac{\beta}{\rho}\|J_h\bm{u}_D-\bm{u}_{D,h}\|_{L^2(\Omega_D)}
\|\bm{u}_D-J_h\bm{u}_D\|_{L^4(\Omega_D)}(\|\bm{u}_D\|_{L^4(\Omega_D)}
+\|J_h\bm{u}_D\|_{L^4(\Omega_D)})\Big).
\end{split}
\label{eq:sigmas}
\end{equation}
An application of the discrete inf-sup condition \eqref{eq:inf-supaS2} and the adjoint property \eqref{eq:discreteaD} yields
\begin{equation}
\begin{split}
\|I_h\bm{u}_S-\bm{u}_{S,h}\|_{h}&\leq C \sup_{\bm{\tau}_S\in [\bm{V}_{h_S}]^2 }\frac{a_S(\bm{\tau}_S,I_h\bm{u}_S-\bm{u}_{S,h})}{\|\bm{\tau}_S\|_{L^2(\Omega_S)}}
=C\sup_{\bm{\tau}_S\in [\bm{V}_{h_S}]^2}
\frac{a_S^*(I_h\bm{u}_S-\bm{u}_{S,h},\bm{\tau}_S)}{\|\bm{\tau}_S\|_{L^2(\Omega_S)}}\\
&=C\sup_{\bm{\tau}_S\in [\bm{V}_{h_S}]^2 }\frac{\nu^{-1}(\bm{\sigma}_S-\bm{\sigma}_{S,h},\bm{\tau}_S)_{\Omega_S}}
{\|\bm{\tau}_S\|_{L^2(\Omega_S)}}\leq C \|\bm{\sigma}_S-\bm{\sigma}_{S,h}\|_{L^2(\Omega_S)},
\end{split}
\label{eq:IhuuS}
\end{equation}
which can be combined with \eqref{eq:sigmas} to get
\begin{align*}
&\|J_h\bm{\sigma}_S-\bm{\sigma}_{S,h}\|_{L^2(\Omega_S)}^2
+\|J_h\bm{u}_D-\bm{u}_{D,h}\|_{L^2(\Omega_D)}^2\\
&\leq C \Big(\|J_h\bm{\sigma}_S-\bm{\sigma}_S\|_{L^2(\Omega_S)}
\|J_h\bm{\sigma}_S-\bm{\sigma}_{S,h}\|_{L^2(\Omega_S)}
+\|p_S-\pi_h p_S\|_P\|\bm{\sigma}_S-\bm{\sigma}_{S,h}\|_{L^2(\Omega_S)}\\
&\;+\frac{\mu}{\rho}\|K^{-1}\|_{L^\infty(\Omega_D)}
\|J_h\bm{u}_D-\bm{u}_{D,h}\|_{L^2(\Omega_D)}
\|\bm{u}_D-J_h\bm{u}_D\|_{L^2(\Omega_D)}\\
&\;+\frac{\beta}{\rho}\|J_h\bm{u}_D-\bm{u}_{D,h}\|_{L^2(\Omega_D)}
\|\bm{u}_D-J_h\bm{u}_D\|_{L^4(\Omega_D)}(\|\bm{u}_D\|_{L^4(\Omega_D)}
+\|J_h\bm{u}_D\|_{L^4(\Omega_D)})\Big).
\end{align*}
Young's inequality gives
\begin{equation}
\begin{split}
&\|J_h\bm{\sigma}_S-\bm{\sigma}_{S,h}\|_{L^2(\Omega_S)}^2
+\|J_h\bm{u}_D-\bm{u}_{D,h}\|_{L^2(\Omega_S)}^2\\
&\leq C \Big(\|J_h\bm{\sigma}_S-\bm{\sigma}_S\|_{L^2(\Omega_S)}^2
+\|p_S-\pi_h p_S\|_P^2+\|\bm{\sigma}_S-J_h\bm{\sigma}_{S}\|_{L^2(\Omega_S)}^2\\
&\;+\frac{\mu}{\rho}\|K^{-1}\|_{L^\infty(\Omega_D)}
\|\bm{u}_D-J_h\bm{u}_D\|_{L^2(\Omega_D)}^2\\
&\;+\frac{\beta}{\rho}
\|\bm{u}_D-J_h\bm{u}_D\|_{L^4(\Omega_D)}^2(\|\bm{u}_D\|_{L^4(\Omega_D)}^2
+\|J_h\bm{u}_D\|_{L^4(\Omega_D)}^2)\Big).
\end{split}
\label{eq:flux}
\end{equation}
Lemma~\ref{lemma:insup3} and the adjoint property \eqref{eq:discreteaD} yield
\begin{align*}
\|I_hp_D-p_{D,h}\|_{Z_D}&\leq C \sup_{\bm{v}_D\in \bm{V}_{h_D}}\frac{a_D(\bm{v}_D, I_hp_D-p_{D,h} )}{\|\bm{v}_D\|_{L^3(\Omega_D)}}=C\sup_{\bm{v}_D\in \bm{V}_{h_D}}\frac{a_D^*(I_hp_D-p_{D,h},\bm{v}_D)}{\|\bm{v}_D\|_{L^3(\Omega_D)}}\\
&=C\sup_{\bm{v}_D\in \bm{V}_{h_D}}\frac{(A(\bm{u}_D)-A(\bm{u}_{D,h}),\bm{v}_D)_{\Omega_D}}{\|\bm{v}_D\|_{L^3(\Omega_D)}}.
\end{align*}
An application of the triangle inequality, the inverse estimates and Sobolev imbedding theorem yields for which $\bm{u}_D\in L^6(\Omega_D)^2$
\begin{align*}
\|\bm{u}_{D,h}\|_{L^6(\Omega_D)}&\leq \|\bm{u}_{D,h}-J_h\bm{u}_D\|_{L^6(\Omega_D)}
+\|J_h\bm{u}_D\|_{L^6(\Omega_D)}\\
&\leq C ( h^{-2/3}\|\bm{u}_{D,h}-J_h\bm{u}_D\|_{L^2(\Omega_D)}+\|\bm{u}_D\|_{L^6(\Omega_D)})\\
&\leq C ( h^{-2/3}\|\bm{u}_{D,h}-J_h\bm{u}_D\|_{L^2(\Omega_D)}+\|\bm{u}_D\|_{H^1(\Omega_D)}).
\end{align*}
The interpolation error estimates \eqref{eq:interpolation}, \eqref{eq:pih} and \eqref{eq:flux} imply that
\begin{align*}
\|\bm{u}_{D,h}\|_{L^6(\Omega_D)}\leq C \|\bm{u}_D\|_{W^{1,4}(\Omega_D)}.
\end{align*}
By (\ref{continuity}), we can obtain
\begin{align*}
(A(\bm{u}_D)-A(\bm{u}_{D,h}), \bm{v}_D)_{\Omega_D}&\leq \frac{\mu}{\rho}\|K^{-1}\|_{L^\infty(\Omega_D)}
\|\bm{u}_D-\bm{u}_{D,h}\|_{L^2(\Omega_D)}\|\bm{v}_D\|_{L^2(\Omega_D)}\\
&+\frac{\beta}{\rho}\|\bm{u}_D-\bm{u}_{D,h}\|_{L^2(\Omega_D)}
(\|\bm{u}_{D,h}\|_{L^6(\Omega_D)}+\|\bm{u}_D\|_{L^6(\Omega_D)})\|\bm{v}_D\|_{L^3(\Omega_D)}.
\end{align*}
Thus
\begin{equation}
\begin{split}
&\|I_hp_D-p_{D,h}\|_{Z_D}\\
&\leq C \Big(\|\bm{u}_D-\bm{u}_{D,h}\|_{L^2(\Omega_D)}+\|\bm{u}_D-\bm{u}_{D,h}\|_{L^2(\Omega_D)}
(\|\bm{u}_{D,h}\|_{L^6(\Omega_D)}+\|\bm{u}_D\|_{L^6(\Omega_D)})\Big).
\end{split}
\label{eq:IhppD}
\end{equation}
It remains to estimate $\|p_{S,h}-\pi_h p_S\|_{L^2(\Omega_S)}$. The discrete inf-sup condition \eqref{inf-sup-bh} yields
\begin{align*}
\|p_{S,h}-\pi_h p_S\|_{L^2(\Omega_S)}\leq C \sup_{\bm{v}\in [U_{h_S}]^2}\frac{b_S(\bm{v},p_{S,h}-\pi_h p_S)}{\|\bm{v}\|_{h}}.
\end{align*}
It then follows from \eqref{eq:error}, the definitions of $I_h$ and $J_h$, the Cauchy-Schwarz inequality, Lemma~\ref{lemma:trace} and Corollary~\ref{coro:trace} that
\begin{align*}
b_S^*(p_{S,h}-\pi_h p_S,\bm{v})&=b_S^*(p_{S}-\pi_h p_S,\bm{v})
+a_S(\bm{\sigma}_S-\bm{\sigma}_{S,h},\bm{v})+\langle p_D-p_{D,h},\bm{v}\cdot\bm{n}_{S}\rangle_\Gamma\\
&\;+G\langle (\bm{u}_S-\bm{u}_{S,h})\cdot\bm{t},\bm{v}\cdot\bm{t}\rangle_\Gamma\\
&=b_S^*(p_{S}-\pi_h p_S,\bm{v})
+a_S(J_h\bm{\sigma}_S-\bm{\sigma}_{S,h},\bm{v})+\langle I_hp_D-p_{D,h},\bm{v}\cdot\bm{n}_{S}\rangle_\Gamma\\
&\;+G\langle (I_h\bm{u}_S-\bm{u}_{S,h})\cdot\bm{t},\bm{v}\cdot\bm{t}\rangle_\Gamma\\
&\leq C \Big(\|p_S-\pi_h p_S\|_P\|\bm{v}\|_{h}+
\|J_h\bm{\sigma}_S-\bm{\sigma}_{S,h}\|_{L^2(\Omega_S)}\|\bm{v}\|_{h}\\
&\;+\|I_hp_D-p_{D,h}\|_{Z_D}\|\bm{v}\|_{h}
+\|I_h\bm{u}_S-\bm{u}_{S,h}\|_{h}\|\bm{v}\|_{h}\Big).
\end{align*}
Therefore
\begin{align*}
&\|p_{S,h}-\pi_h p_S\|_{L^2(\Omega_S)}\\
&\leq C \Big(\|p_S-\pi_h p_S\|_P+
\|J_h\bm{\sigma}_S-\bm{\sigma}_{S,h}\|_{L^2(\Omega_S)}
+\|I_hp_D-p_{D,h}\|_{Z_D}+\|I_h\bm{u}_S-\bm{u}_{S,h}\|_{h}\Big).
\end{align*}
The proof is complete by combining the preceding arguments and the interpolation error estimates \eqref{eq:interpolation}, \eqref{eq:IhD} and \eqref{eq:pih}.

\end{proof}

We can achieve the following superconvergence by following \eqref{eq:IhuuS} and \eqref{eq:IhppD}
\begin{remark}(superconvergence)
Assume that $\bm{u}_S\in H^{\alpha+1}(\Omega_S)^{2},p_S\in H^\alpha(\Omega_S)$ and $\bm{u}_D\in W^{\alpha,4}(\Omega_D)^2$, then
\begin{align*}
&\|I_h\bm{u}_S-\bm{u}_{S,h}\|_h+\|I_hp_D-p_{D,h}\|_{Z_D}\\
&\quad\leq C h^{\min\{k+1,\alpha\}}\Big(\|p_S\|_{H^{\min\{k+1,\alpha\}}(\Omega_S)}
+\|\bm{u}_S\|_{H^{\min\{k+2,\alpha+1\}}(\Omega_S)}
+\|\bm{u}_D\|_{W^{\min\{k+1,\alpha\},4}(\Omega_D)}\Big).
\end{align*}

\end{remark}

%\begin{theorem}\label{thm:con1}
%Assume that $\bm{u}_D\in W^{\alpha, 4}(\Omega_D)$,$p_D\in W^{\alpha+1,3/2}(\Omega_D)$, $\bm{\sigma}_S\in H^{\alpha}(\Omega_S)$, $p_S\in H^\alpha(\Omega_S)$, $1\leq \alpha \leq k+1$, then
%\begin{align*}
%\|\bm{\sigma}_S-\bm{\sigma}_{S,h}\|_{L^2(\Omega_D)}
%+\|\bm{u}_D-\bm{u}_{D,h}\|_{L^2(\Omega_D)}&\leq C h^{\alpha}\Big(\|\bm{\sigma}_S\|_{H^{\alpha}(\Omega_S)}+\|p_S\|_{H^{\alpha}(\Omega_S)}
%+\|\bm{u}_D\|_{W^{\alpha,4}(\Omega_D)}\Big),\\
%%\|p_D-p_{D,h}\|_{Z_D}&\leq C h^{\alpha}(\|\bm{u}_D\|_{H^\alpha(\Omega_D)}),\\
%\|p_S-p_{S,h}\|_{L^2(\Omega_S)}&\leq C h^\alpha(\|p_S\|_{H^\alpha(\Omega_S)}+\|\bm{\sigma}_S\|_{H^{\alpha}(\Omega_S)}
%+\|\bm{u}_D\|_{W^{\alpha,4}(\Omega_D)}),\\
%\|\bm{u}_S-\bm{u}_{S,h}\|_h+\|p_D-p_{D,h}\|_{Z_D}&\leq C h^\alpha \Big(\|\bm{\sigma}_S\|_{H^{\alpha}(\Omega_S)}+\|p_S\|_{H^{\alpha}(\Omega_S)}
%+\|\bm{u}_D\|_{W^{\alpha,4}(\Omega_D)}\Big).
%\end{align*}
%
%\end{theorem}

%\begin{lemma}(Sobolev's inequalities)
%
%for any pair of real numbers $p\geq 1$ and $q>1$ such that $p\leq p_0$ with $\frac{1}{p_0}=\frac{1}{q}-\frac{1}{2}$, there exists a positive constant independent of $h$ such that
%\begin{align*}
%\|z_h\|_{L^2(\Omega)}\leq C \|z_h\|_Z
%\end{align*}
%
%\end{lemma}

\begin{theorem}

Assume that $\bm{u}_D\in W^{\alpha, 4}(\Omega_D)^2$,$p_D\in H^{\alpha+1}(\Omega_D)$, $\bm{u}_S\in H^{\alpha+1}(\Omega_S)^2$, $p_S\in H^\alpha(\Omega_S)$, then we have the following $L^2$ error estimates
\begin{align*}
&\|p_D-p_{D,h}\|_{L^2(\Omega_D)}+\|\bm{u}_S-\bm{u}_{S,h}\|_{L^2(\Omega_S)}\\
&\quad\leq Ch^{\min\{k+1,\alpha\}}\Big(\|p_S\|_{H^{\min\{k+1,\alpha\}}(\Omega_S)}
+\|\bm{u}_S\|_{H^{\min\{k+2,\alpha+1\}}(\Omega_S)}\\
&\qquad+\|\bm{u}_D\|_{W^{\min\{k+1,\alpha\},4}(\Omega_D)}
+\|p_D\|_{H^{\min\{k+1,\alpha\}}(\Omega_D)}\Big).
\end{align*}

\end{theorem}

\begin{proof}
%Next, an application of triangle inequality, the inverse estimates and Sobolev imbedding theorem yields for which $\bm{u}\in L^6(\Omega)^2$
%\begin{align*}
%\|\bm{u}_h\|_{0,6}&\leq \|\bm{u}_h-J_h\bm{u}\|_{0,6}
%+\|J_h\bm{u}\|_{0,6}\\
%&\leq C ( h^{-2/3}\|\bm{u}_h-J_h\bm{u}\|_{0,2}+\|\bm{u}\|_{0,6})\\
%&\leq C ( h^{-2/3}\|\bm{u}_h-J_h\bm{u}\|_{0,2}+\|\bm{u}\|_{H^1}).
%\end{align*}
%By (\ref{continuity}), we can obtain
%\begin{align*}
%(A(\bm{u})-A(\bm{u}_h), \bm{\varphi}_h)&\leq \frac{\mu}{\varrho}\|K^{-1}\|_{L^\infty(\Omega)}
%\|\bm{u}-\bm{u}_h\|_0\|\bm{\varphi}_h\|_0\\
%&+\frac{\beta}{\varrho}\|\bm{u}-\bm{u}_h\|_{0,2}
%(\|\bm{u}_h\|_{0,6}+\|\bm{u}\|_{0,6})\|\bm{\varphi}_h\|_{0,3}.
%\end{align*}
%Recall that we have
%\begin{align*}
%\|I_h\bm{u}_{S}-\bm{u}_{S,h}\|_h+\|p_D-p_{D,h}\|_Z
%\leq C (\|\bm{\sigma}_S-\bm{\sigma}_{S,h}\|_0+\|\bm{u}_S-\bm{u}_{S,h}\|_0).
%\end{align*}
Lemma~\ref{lemma:poincarev} reveals that
\begin{align*}
\|I_h p_D-p_{D,h}\|_0\leq C\|I_hp_D-p_{D,h}\|_{Z_D}.
\end{align*}

By the discrete Poincar\'{e}-Friedrichs inequality given in \cite{Brenner}, Lemma~\ref{lemma:poincarev}, \eqref{eq:IhuuS} and \eqref{eq:IhppD}, we have
\begin{align*}
&\|I_h\bm{u}_{S}-\bm{u}_{S,h}\|_{L^2(\Omega_S)}+\|I_hp_D-p_{D,h}\|_{L^2(\Omega_D)}\\
&\leq C\Big( \|I_h\bm{u}_{S}-\bm{u}_{S,h}\|_h+\|I_hp_D-p_{D,h}\|_{Z_D}\Big)\\
&\leq C\Big(\|\bm{\sigma}_S-\bm{\sigma}_{S,h}\|_{L^2(\Omega_S)}
+\|\bm{u}_D-\bm{u}_{D,h}\|_{L^2(\Omega_D)}\\
&\;+\|\bm{u}_D-\bm{u}_{D,h}\|_{L^2(\Omega_D)}
(\|\bm{u}_{D,h}\|_{L^6(\Omega_D)}+\|\bm{u}_D\|_{L^6(\Omega_D)})\Big).
\end{align*}
The assertion follows by employing Theorem~\ref{thm:con1}.

\end{proof}

\section{Numerical experiments}\label{sec:numerical}
The section presents several numerical tests to confirm the proposed theoretical results. Fist, we present three examples with known exact solution to test the convergence of the proposed method, where the permeability for the third example is highly oscillatory. Then, the fourth example shows the application of the proposed method to problems with high contrast permeability, and the numerical results indicate that our method is a good candidate for practical applications. We note that all the experiments are tested for piecewise linear polynomial, i.e., $k=1$.

The numerical examples presented below can violate the interface conditions \eqref{interface-normal} and \eqref{interface-BJS}, that is, \eqref{interface-normal} and \eqref{interface-BJS} are replaced by
\begin{equation*}
\begin{split}
p_S-\nu\bm{n}_{S}\frac{\partial\bm{u}_S}{\partial \bm{n}_{S}}&=p_D+g_1,\hspace{1cm} \mbox{on}\; \Gamma,\\
-\nu \bm{t}\frac{\partial \bm{u}_{S}}{\partial \bm{n}_{S}}&=G\bm{u}_{S}\cdot\bm{t}+g_2,\quad \mbox{on}\; \Gamma,
\end{split}
\end{equation*}
to deal with this case, the variational formulation  has only a small change: the first equation of \eqref{eq:discrete1} now includes the two terms $-\langle g_1, \bm{v}_S\cdot\bm{n}_{S}\rangle_\Gamma-\langle g_2,\bm{v}_S\cdot\bm{t}\rangle_\Gamma$ on the right side.

\begin{example}\label{ex1}

\end{example}
In this example, we consider $\Omega_S=(0,1)\times(0,1)$ and $\Omega_D=(0,1)\times (1,2)$. We set $K$ to be the identity tensor in $\mathbb{R}^{2\times 2}$, $\mu=\rho=\beta=\nu=1$ and the exact solution is given by
\begin{align*}\bm{u}_S=
\left(
  \begin{array}{c}
    -\cos^2(\frac{\pi y}{2})\sin(\frac{\pi x}{2}) \\
    \frac{1}{4}\cos(\frac{\pi x}{2})(\sin(\pi y)+\pi y) \\
  \end{array}
\right),\quad p_S=-\frac{\pi}{4} \cos(\frac{\pi x}{2})(y-2\cos(\frac{\pi y}{2})^2)
\end{align*}
and
\begin{align*}\bm{u}_D=
\left(
  \begin{array}{c}
     -\frac{1}{8}y \pi^2\sin(\frac{\pi x}{2}) \\
    \frac{1}{4}\pi \cos(\frac{\pi x}{2}) \\
  \end{array}
\right),\quad p_D=-\frac{\pi}{4} \cos(\frac{\pi x}{2}) y.
\end{align*}
The corresponding $\bm{f}_S$, $f_D$ and $\bm{g}_D$ can be calculated. Note that this example satisfy the interface condition \eqref{interface-condition}. The numerical approximations on triangular meshes are displayed in Figure~\ref{ex1:solution}, where $\bm{u}_{S,h}=(u_{S,h}^{1},u_{S,h}^2), \bm{u}_{D,h}=(u_{D,h}^1,u_{D,h}^2)$. Here, the triangular meshes are matching across the interface. The convergence history against the number of degrees of freedom for all the variables in $L^2$-norm are reported in Figure~\ref{ex1-con}, from which we can see that the optimal convergence in $L^2$-norm for all the variables can be obtained as indicated by our theory.

\begin{figure}[H]
    \centering
    \begin{minipage}[b]{0.4\textwidth}
      \includegraphics[width=1\textwidth]{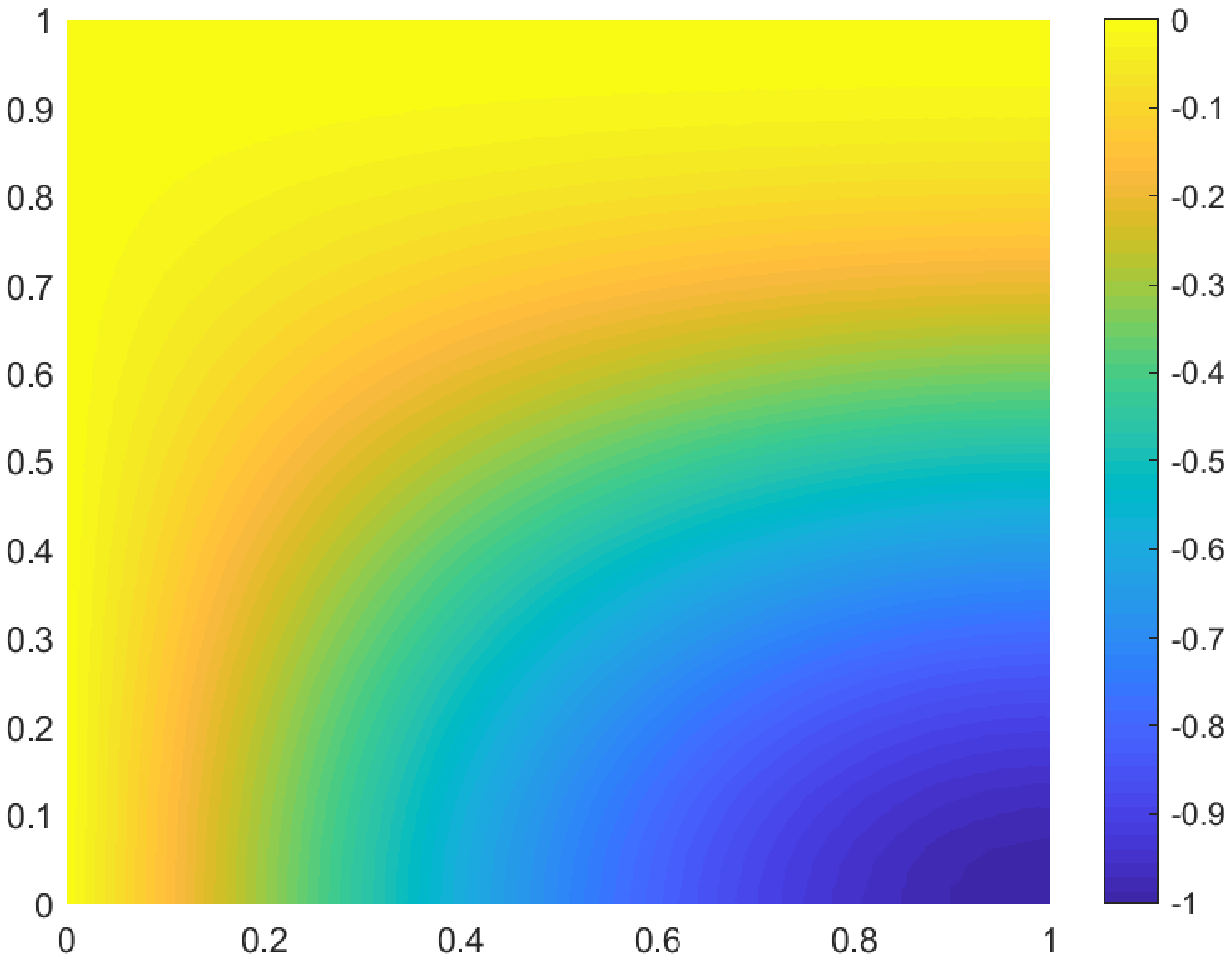}
    \end{minipage}%
    \begin{minipage}[b]{0.4\textwidth}
      \includegraphics[width=1\textwidth]{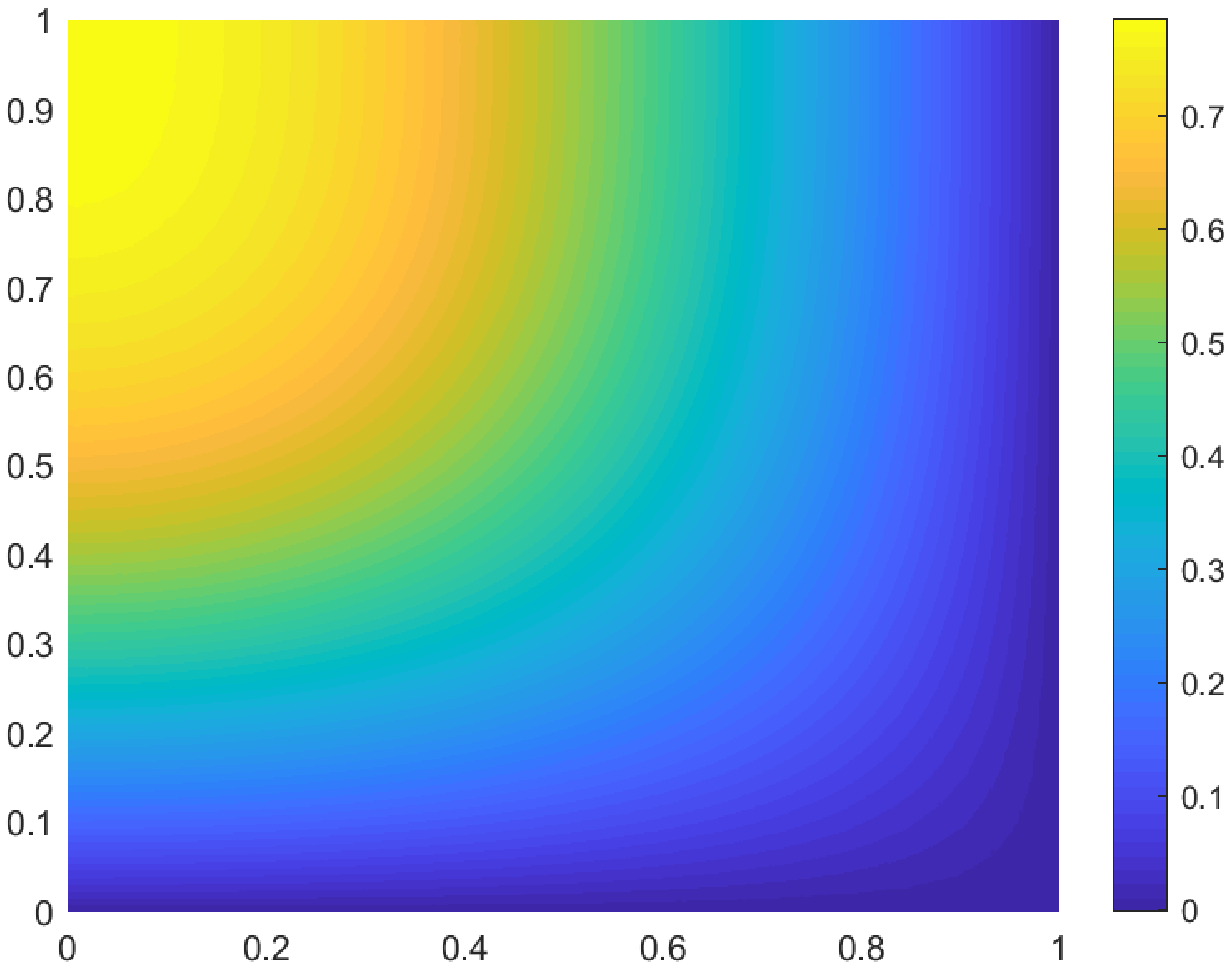}
    \end{minipage}
    \begin{minipage}[b]{0.4\textwidth}
      \includegraphics[width=1\textwidth]{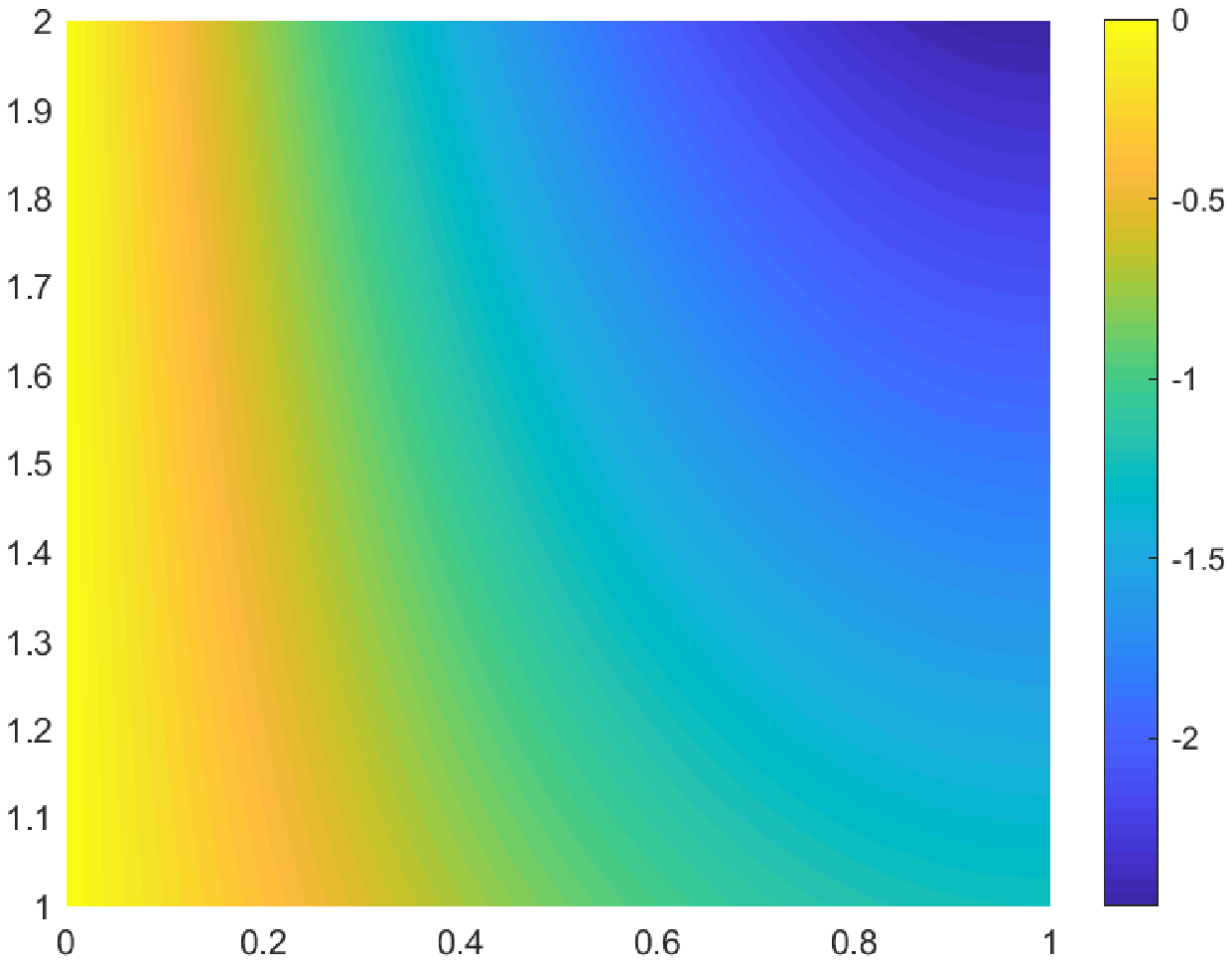}
    \end{minipage}
     \begin{minipage}[b]{0.4\textwidth}
      \includegraphics[width=1\textwidth]{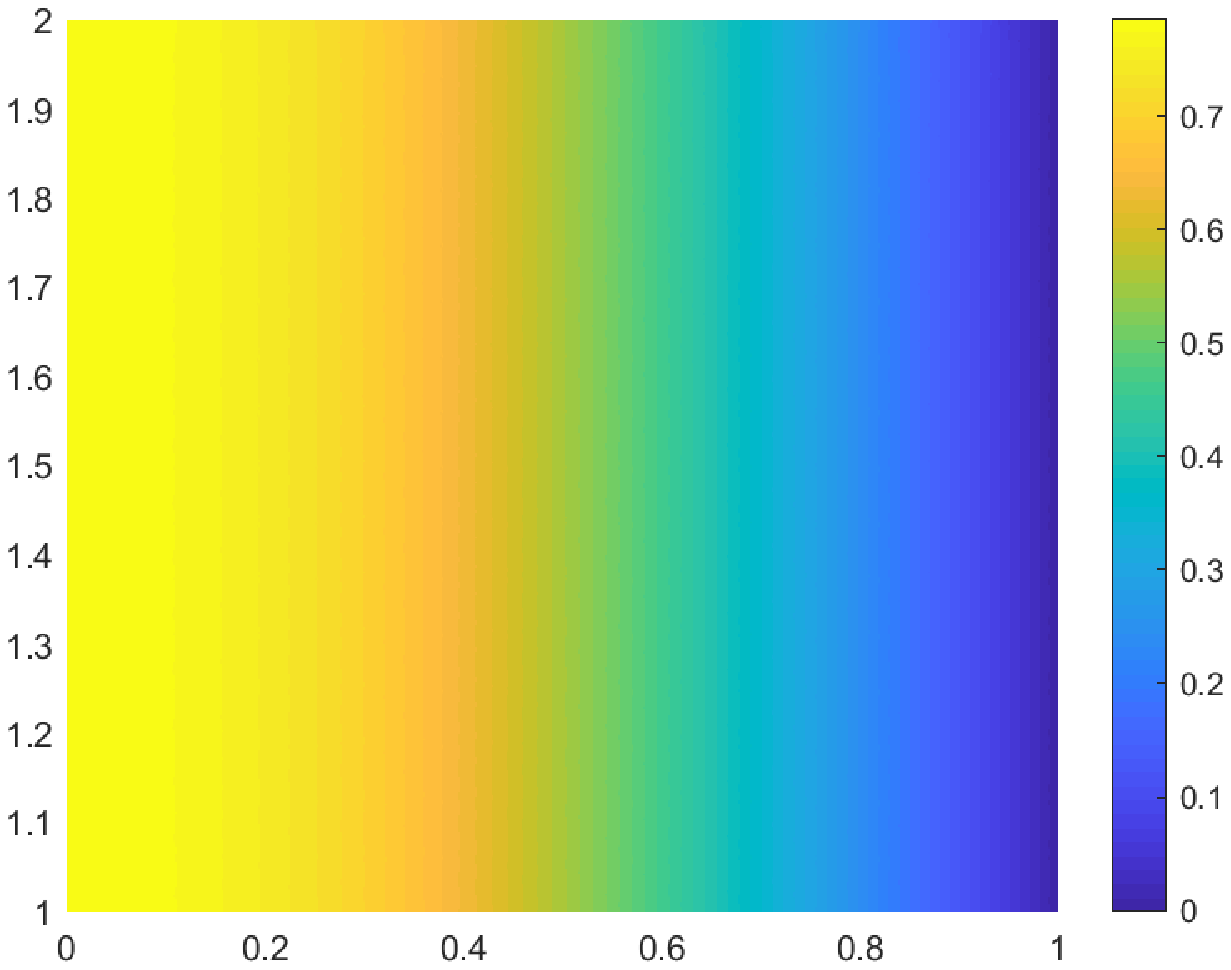}
    \end{minipage}
     \begin{minipage}[b]{0.4\textwidth}
      \includegraphics[width=1\textwidth]{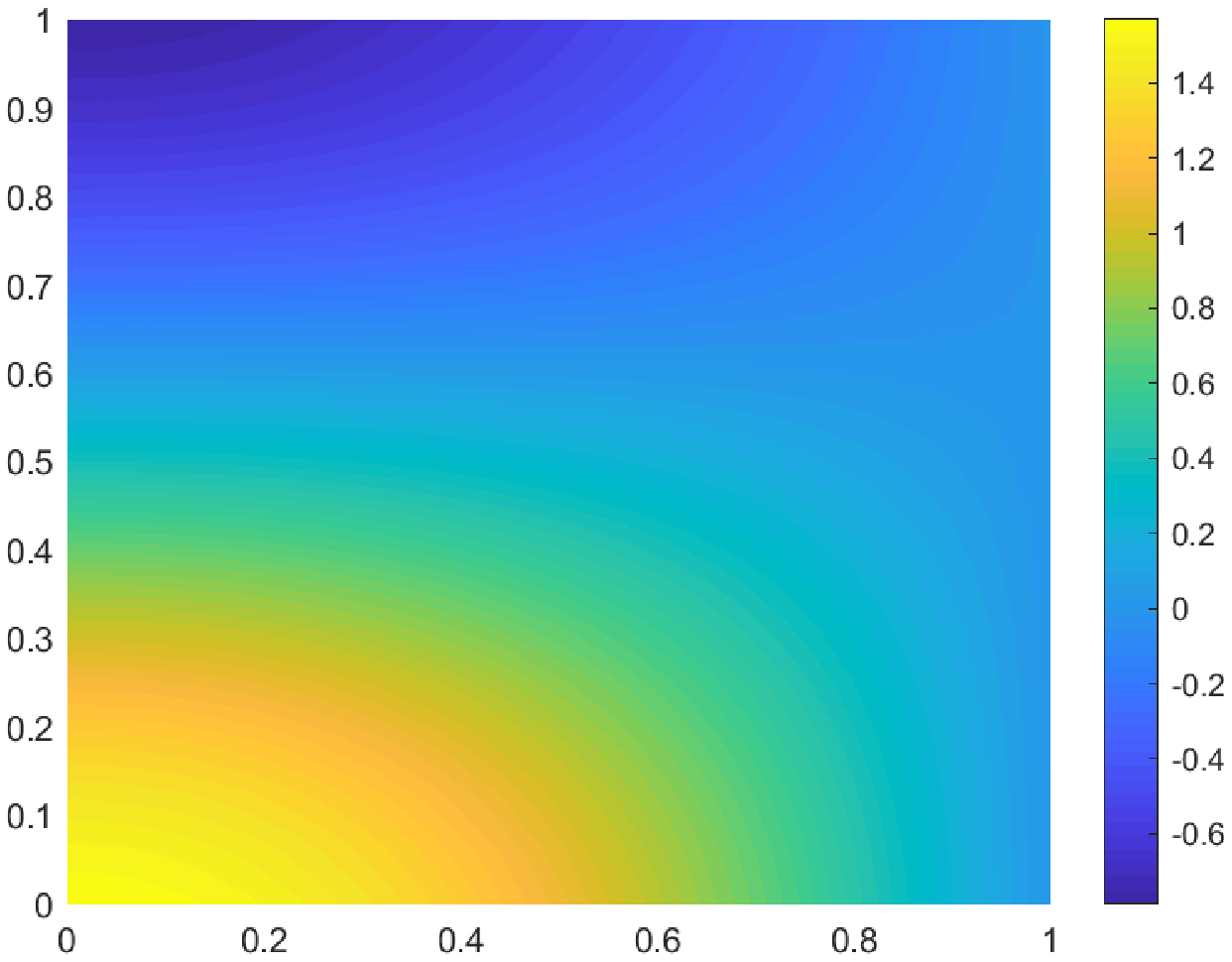}
    \end{minipage}
    \begin{minipage}[b]{0.4\textwidth}
      \includegraphics[width=1\textwidth]{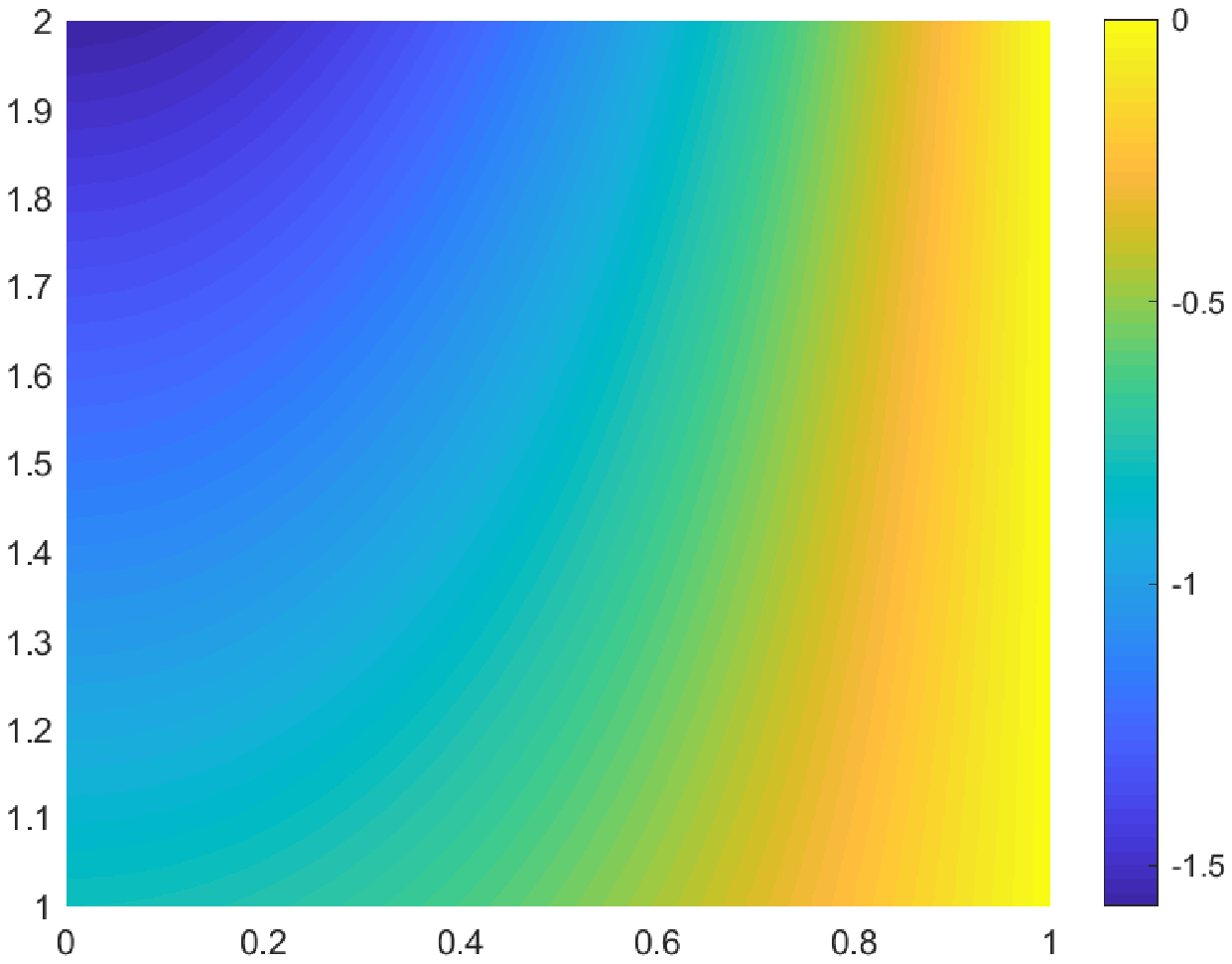}
    \end{minipage}%
  \caption{Numerical solution: $u_{S,h}^1$,$u_{S,h}^2$, $u_{D,h}^{1}$,$u_{D,h}^2$, $p_{S,h}$ and $p_{D,h}$ (left to right, top to bottom).}
  \label{ex1:solution}
\end{figure}

\begin{figure}[H]
\centering
\scalebox{0.3}{
\includegraphics[width=20cm]{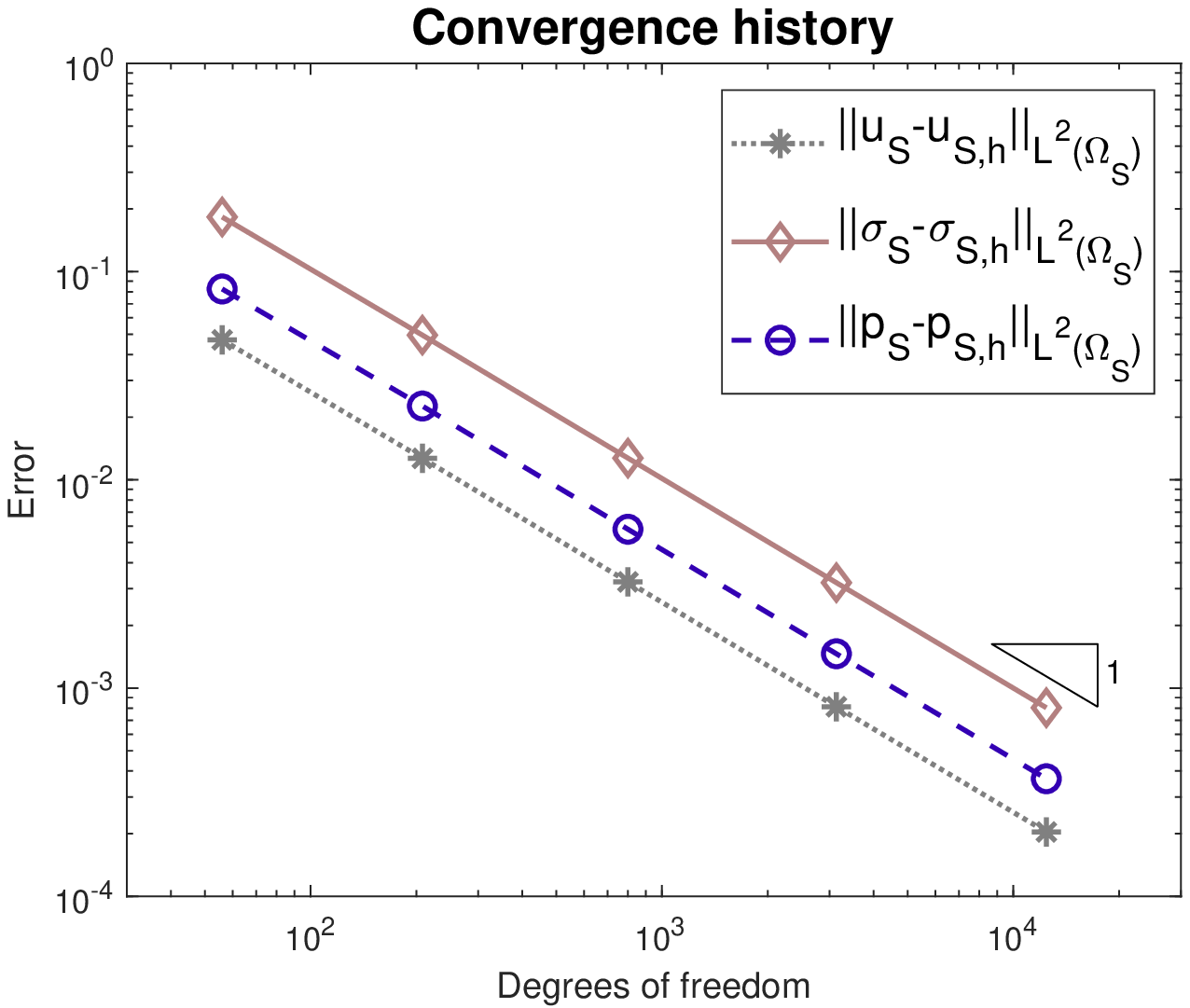}
}
\scalebox{0.3}{
\includegraphics[width=20cm]{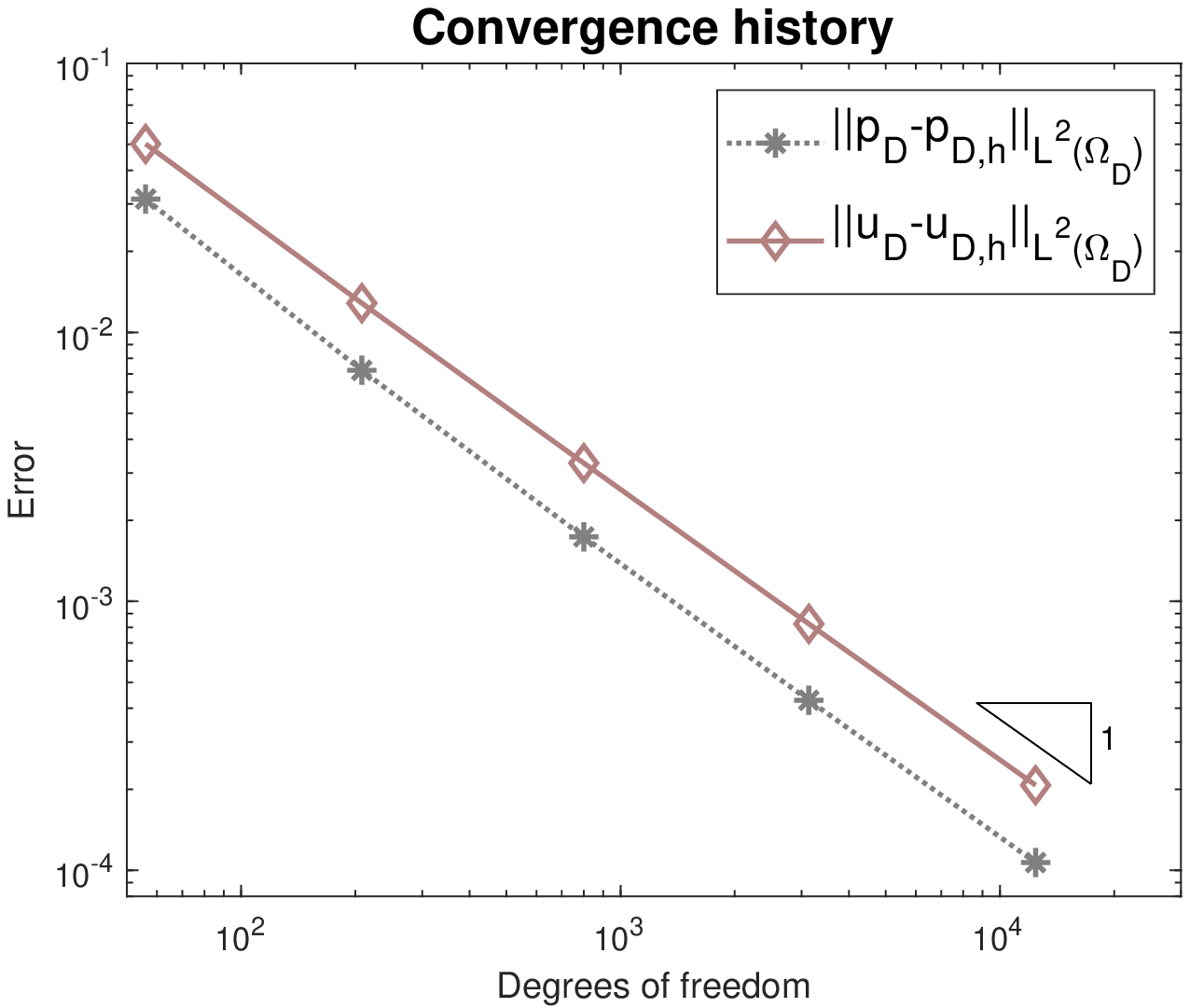}
}
\caption{Convergence history for Example~\ref{ex1}.}
\label{ex1-con}
\end{figure}

%\begin{figure}
%\centering
%\scalebox{0.3}{
%\includegraphics[width=20cm]{figs/ex1_uS1.eps}
%}
%\scalebox{0.3}{
%\includegraphics[width=20cm]{figs/ex1_uS2.eps}
%}
%\caption{Convergence history.}
%\end{figure}
%
%\begin{figure}
%\centering
%\scalebox{0.3}{
%\includegraphics[width=20cm]{figs/ex1_uD1.eps}
%}
%\scalebox{0.3}{
%\includegraphics[width=20cm]{figs/ex1_uD2.eps}
%}
%\caption{Convergence history.}
%\end{figure}
%
%
%\begin{figure}
%\centering
%\scalebox{0.3}{
%\includegraphics[width=20cm]{figs/ex1_pD.eps}
%}
%\scalebox{0.3}{
%\includegraphics[width=20cm]{figs/ex1_pS.eps}
%}
%\caption{Convergence history.}
%\end{figure}

Next, we test the convergence on the rectangular grids and highly distorted grids, see Figure~\ref{ex1:grid}. The convergence history for both cases are reported in Figure~\ref{ex1:consd} and Figure~\ref{ex1:condistort}, from which we can see that the optimal convergence rates can be achieved for both cases. In addition, the accuracy of the proposed method is slightly influenced by the distortion of the grids. Thus, we can conclude that the proposed method is robust in the sense that it can be flexibly applied to rough grids without losing the accuracy.
\begin{figure}[H]
\centering
\scalebox{0.3}{
\includegraphics[width=20cm]{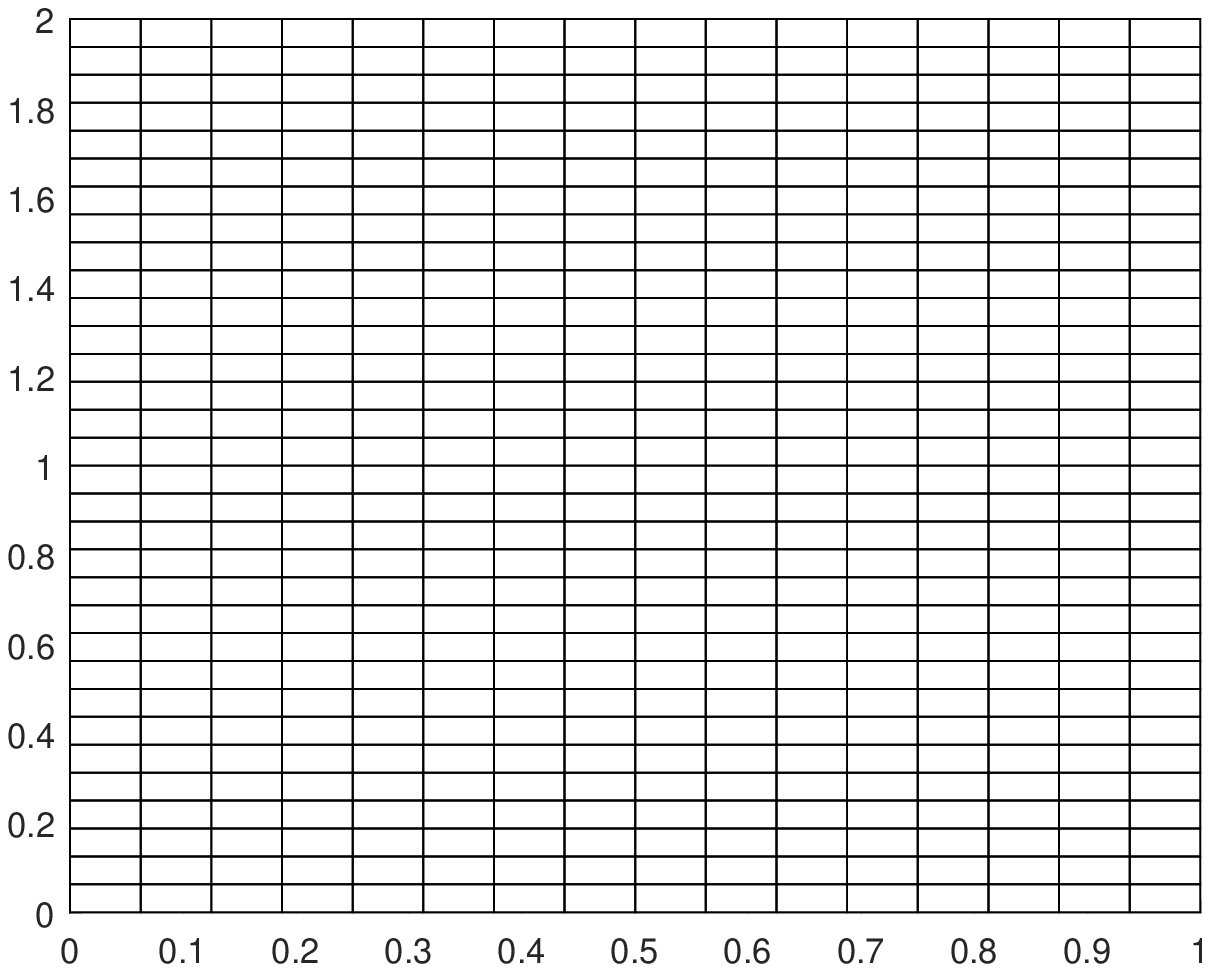}
}
\scalebox{0.3}{
\includegraphics[width=20cm]{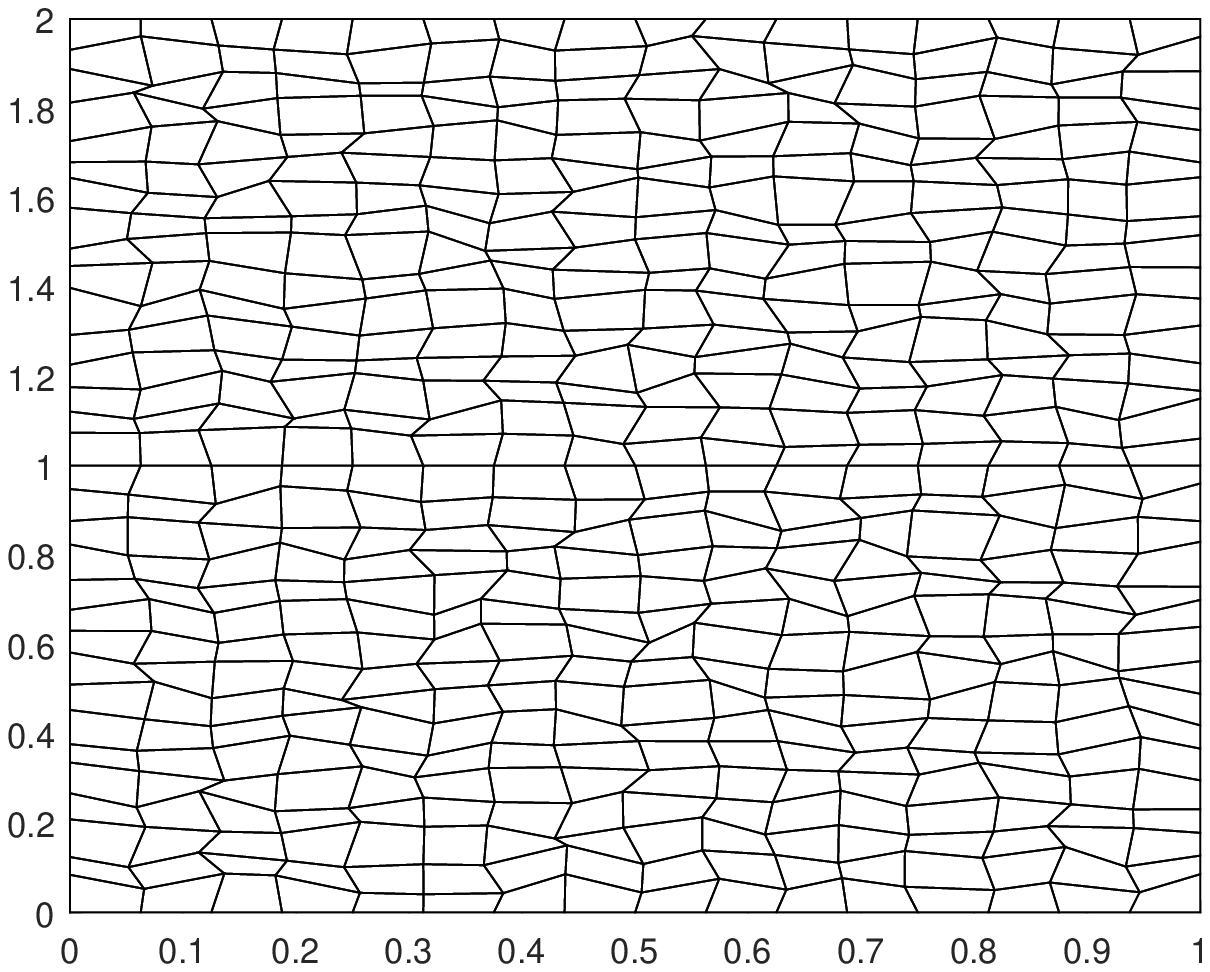}
}
\caption{Rectangular mesh (left) and highly distorted rectangular mesh (right).}
\label{ex1:grid}
\end{figure}

\begin{figure}[H]
\centering
\scalebox{0.3}{
\includegraphics[width=20cm]{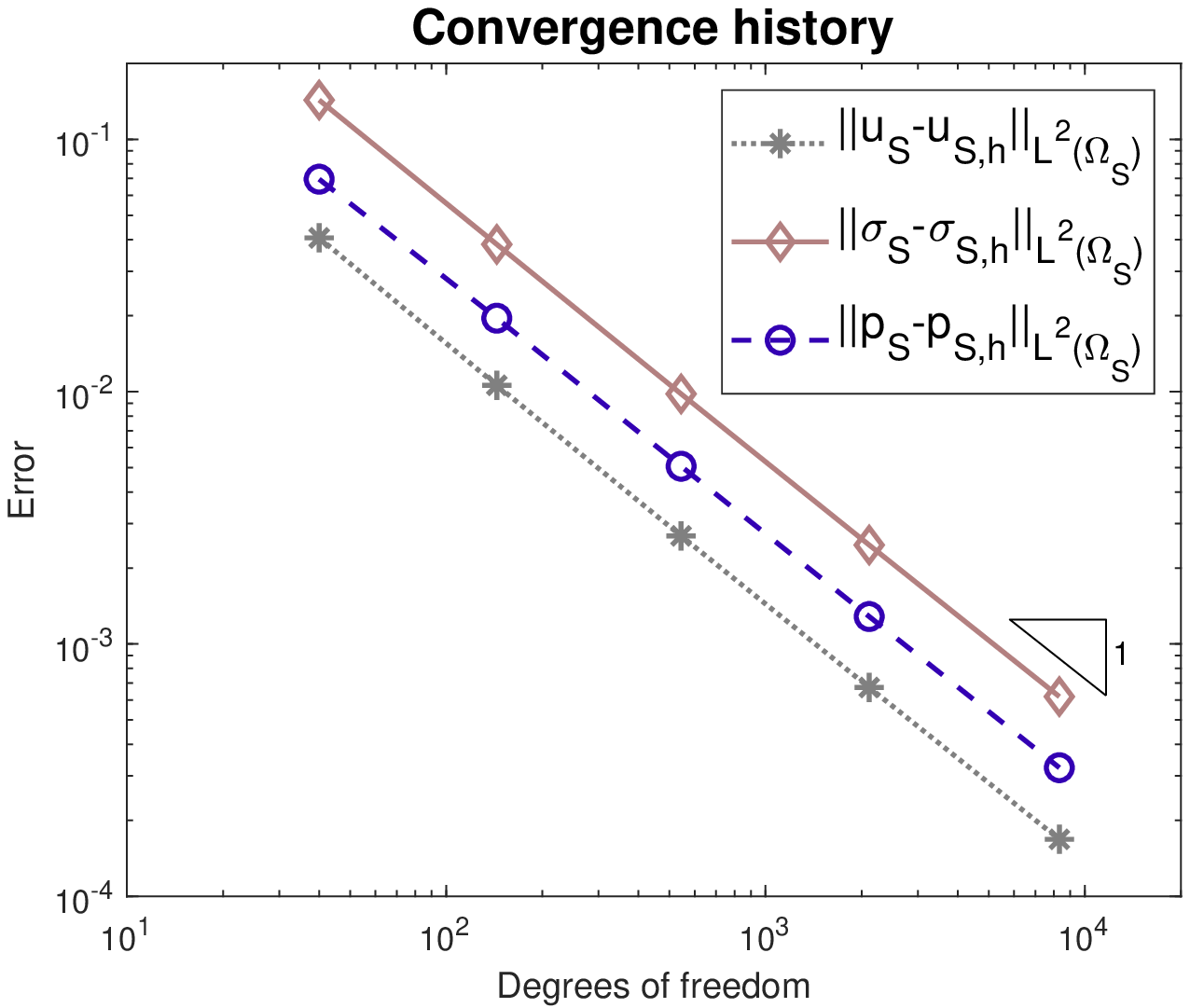}
}
\scalebox{0.3}{
\includegraphics[width=20cm]{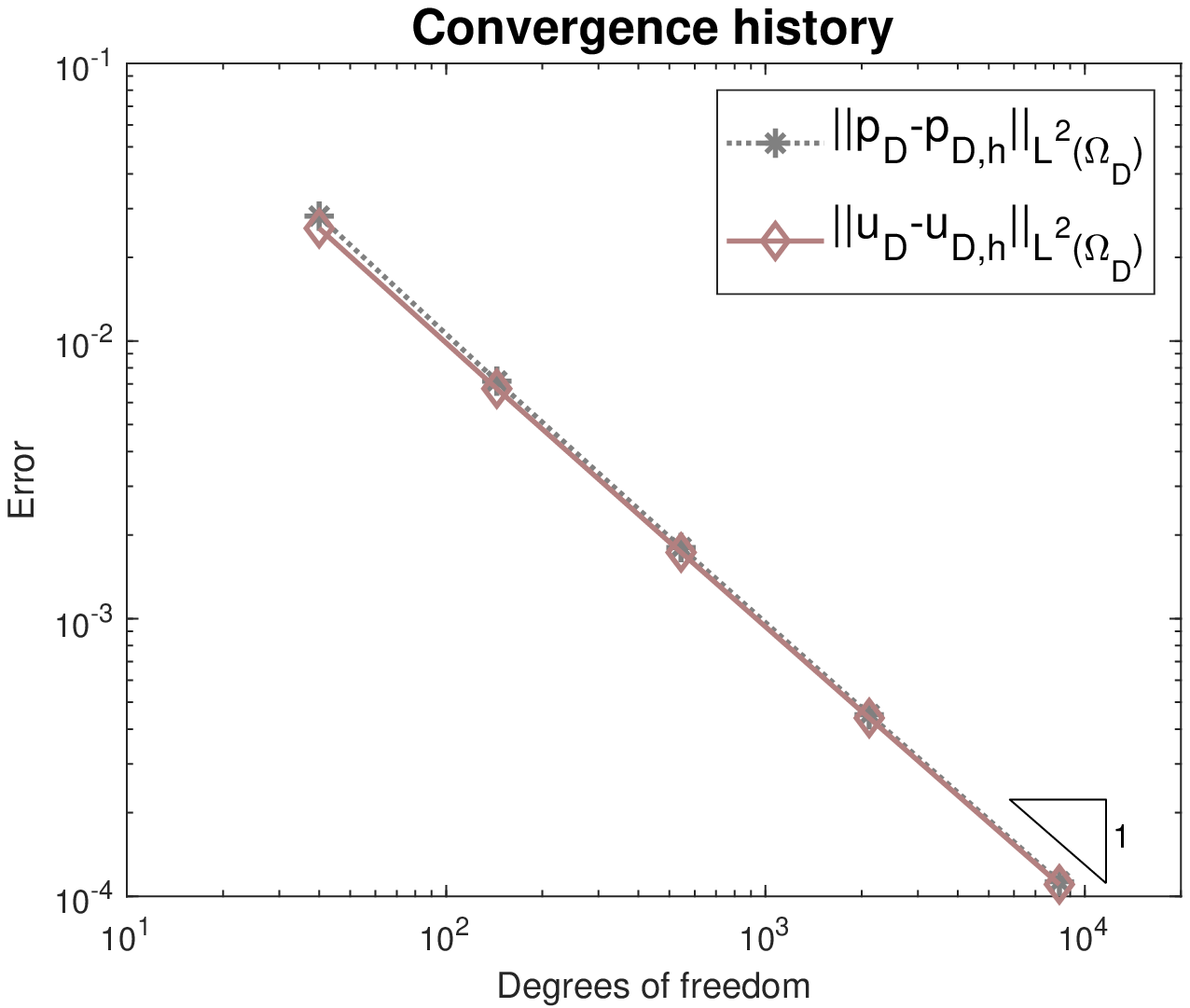}
}
\caption{Convergence history for rectangular mesh.}
\label{ex1:consd}
\end{figure}

\begin{figure}[H]
\centering
\scalebox{0.3}{
\includegraphics[width=20cm]{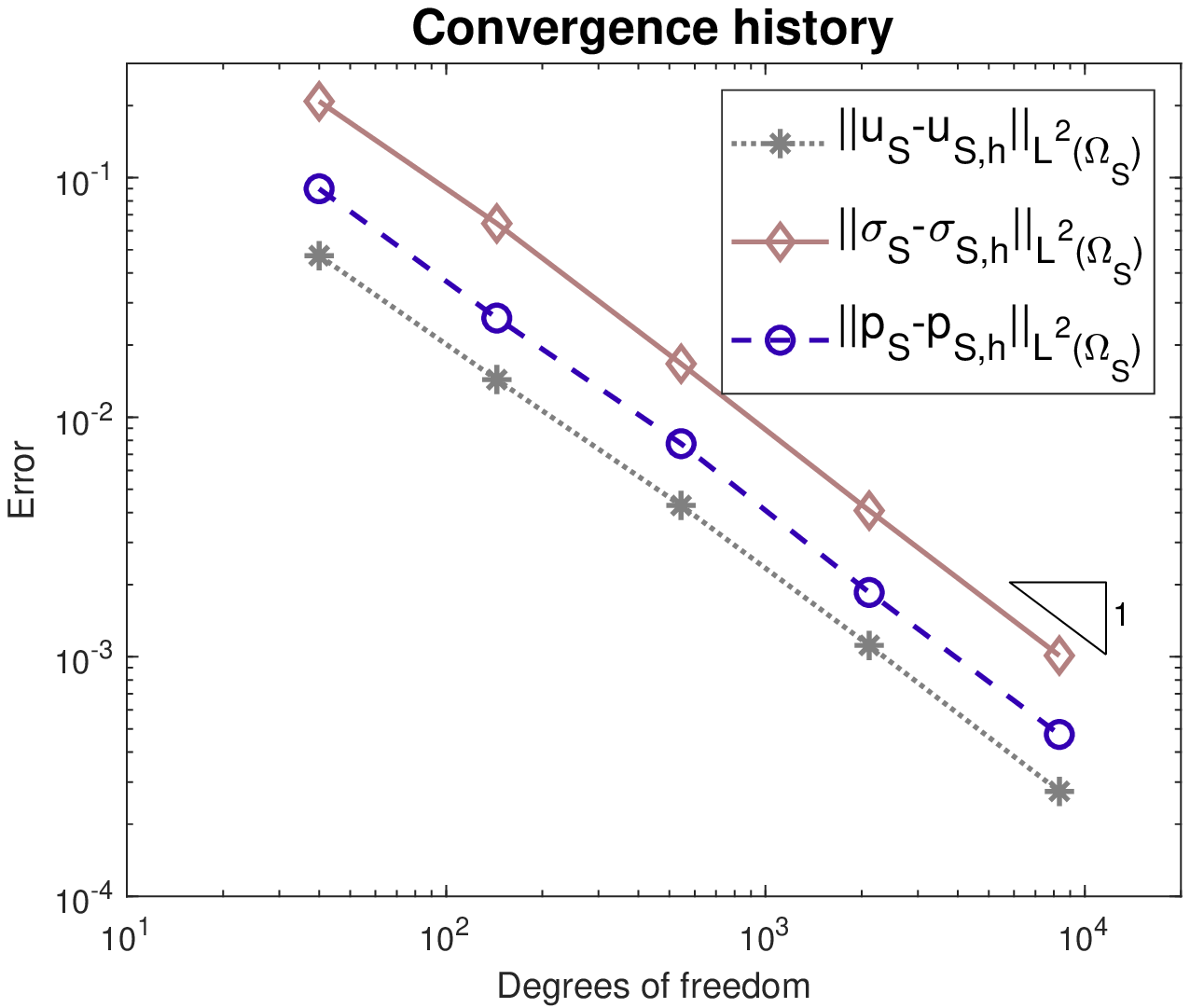}
}
\scalebox{0.3}{
\includegraphics[width=20cm]{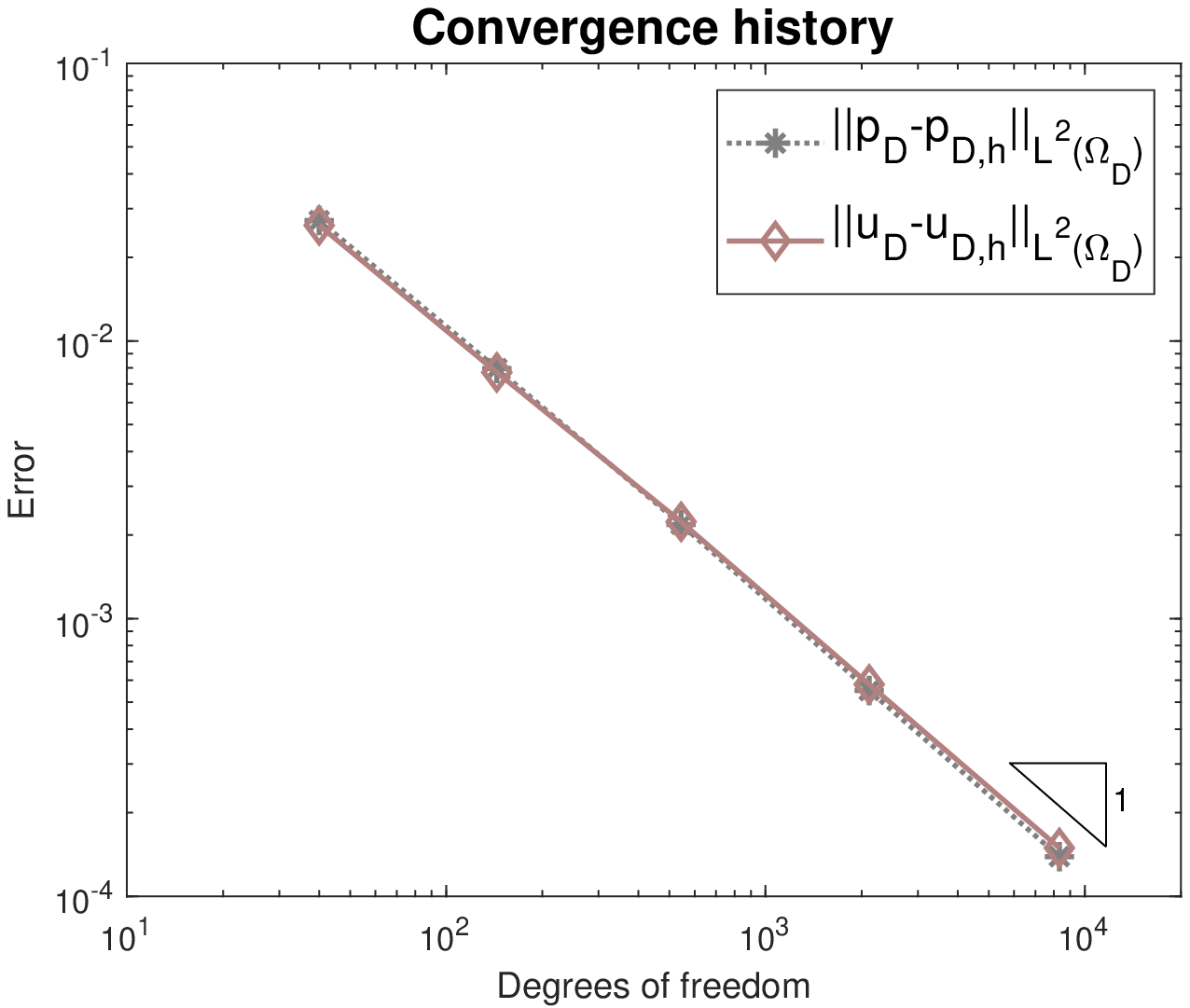}
}
\caption{Convergence history for highly distorted rectangular mesh.}
\label{ex1:condistort}
\end{figure}

To simplify the presentation, we only display the numerical results on triangular meshes for the following examples.

\begin{example}\label{ex2}
\end{example}

In this example, we consider $\Omega_S=(0,1)\times(0,1)$ and $\Omega_D=(0,1)\times(1,2)$. We set $K$ to be the identity tensor in $\mathbb{R}^{2\times 2}$, $\mu=\rho=\beta=\nu=1$ and the exact solution is given by
\begin{align*}\bm{u}_S=
\left(
  \begin{array}{c}
      x^2\pi\sin(2y\pi)(x - 1)^2 \\
    -2x\sin(y\pi)^2(2x - 1)(x - 1) \\
  \end{array}
\right),\quad p_S= (\cos(1)-1)\sin(1)+\cos(y)\sin(x)
\end{align*}
and
\begin{align*}\bm{u}_D=
\left(
  \begin{array}{c}
      \sin(\pi x)\sin(\pi y); \\
     -2x\sin(y\pi)^2(2x - 1)(x - 1) \\
  \end{array}
\right),\quad p_D=\sin(\pi x)\cos(\pi y).
\end{align*}

The numerical approximations are reported in Figure~\ref{ex3:solution} and the convergence history for all the variables measured in $L^2$ errors can be found in Figure~\ref{ex3:con}. Again, the optimal convergence can be obtained.
\begin{figure}[H]
    \centering
    \begin{minipage}[b]{0.4\textwidth}
      \includegraphics[width=1\textwidth]{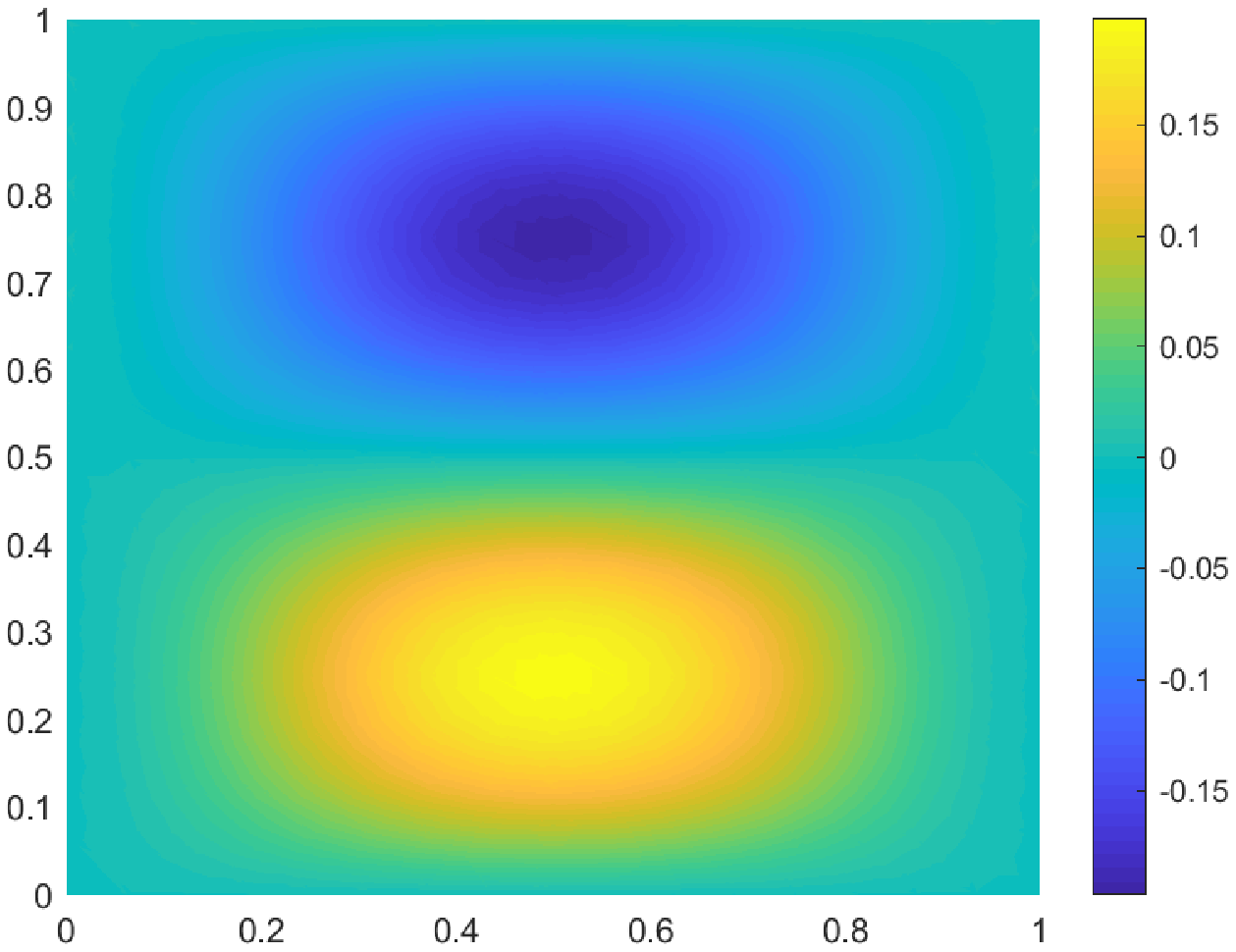}
    \end{minipage}%
    \begin{minipage}[b]{0.4\textwidth}
      \includegraphics[width=1\textwidth]{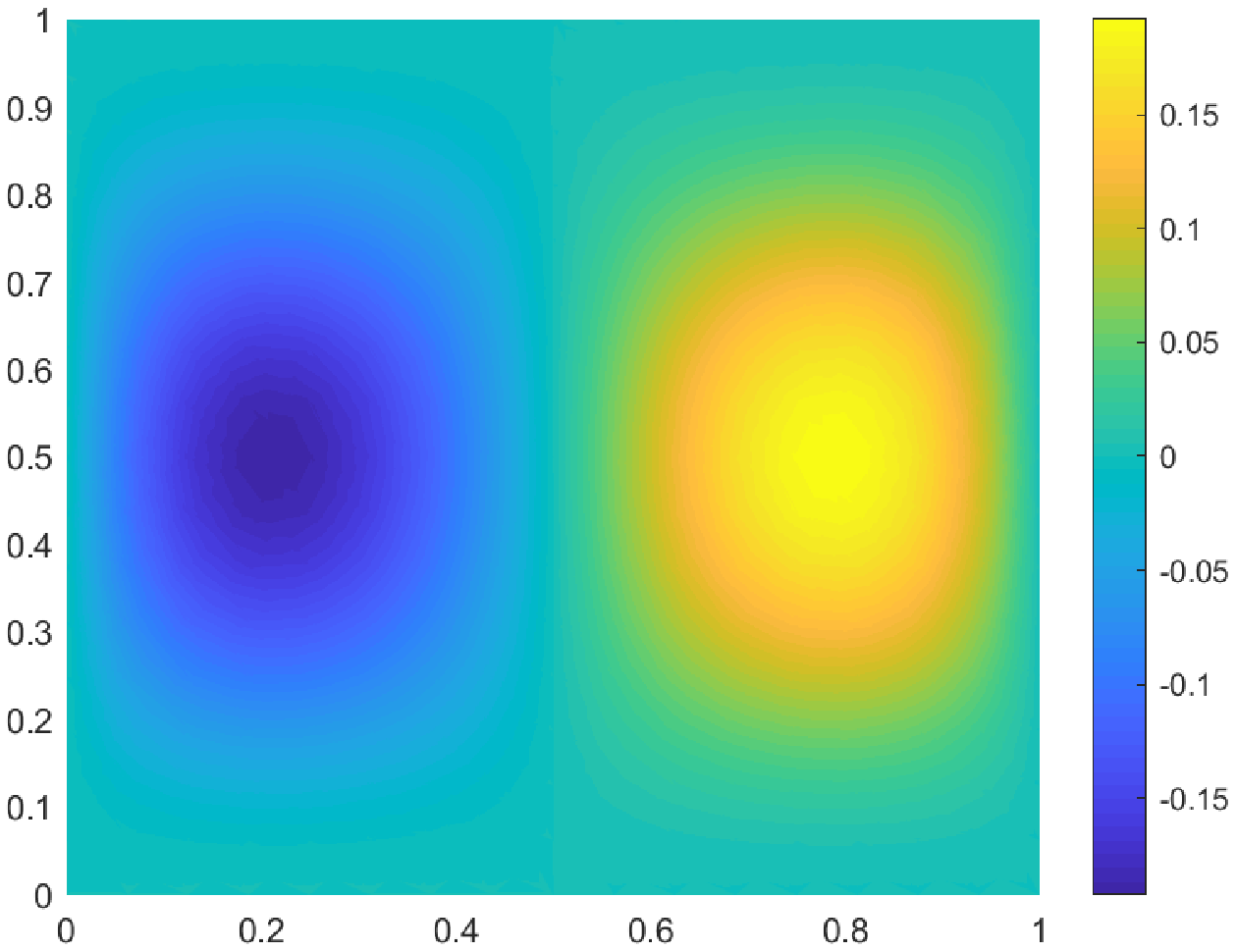}
    \end{minipage}
    \begin{minipage}[b]{0.4\textwidth}
      \includegraphics[width=1\textwidth]{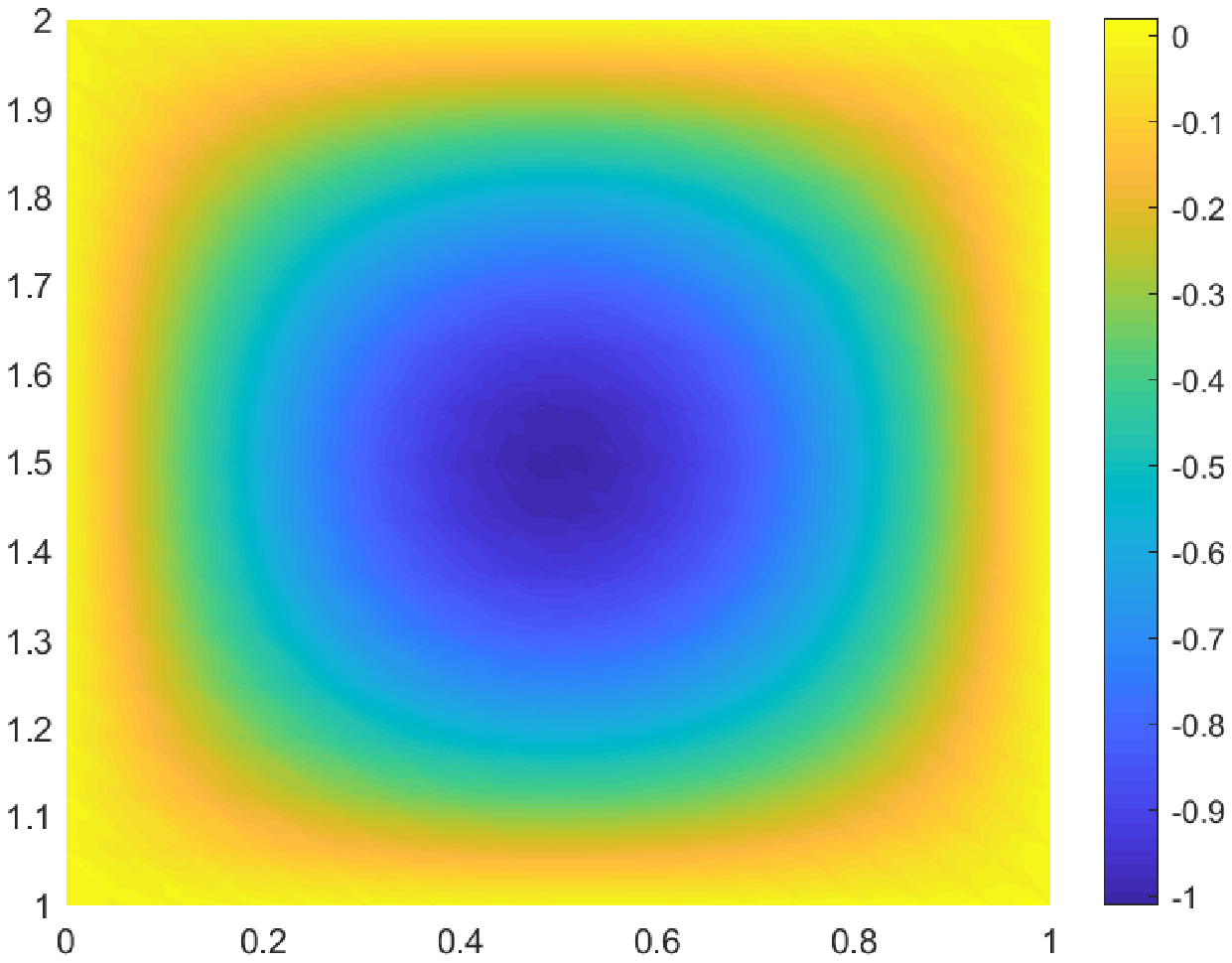}
    \end{minipage}
     \begin{minipage}[b]{0.4\textwidth}
      \includegraphics[width=1\textwidth]{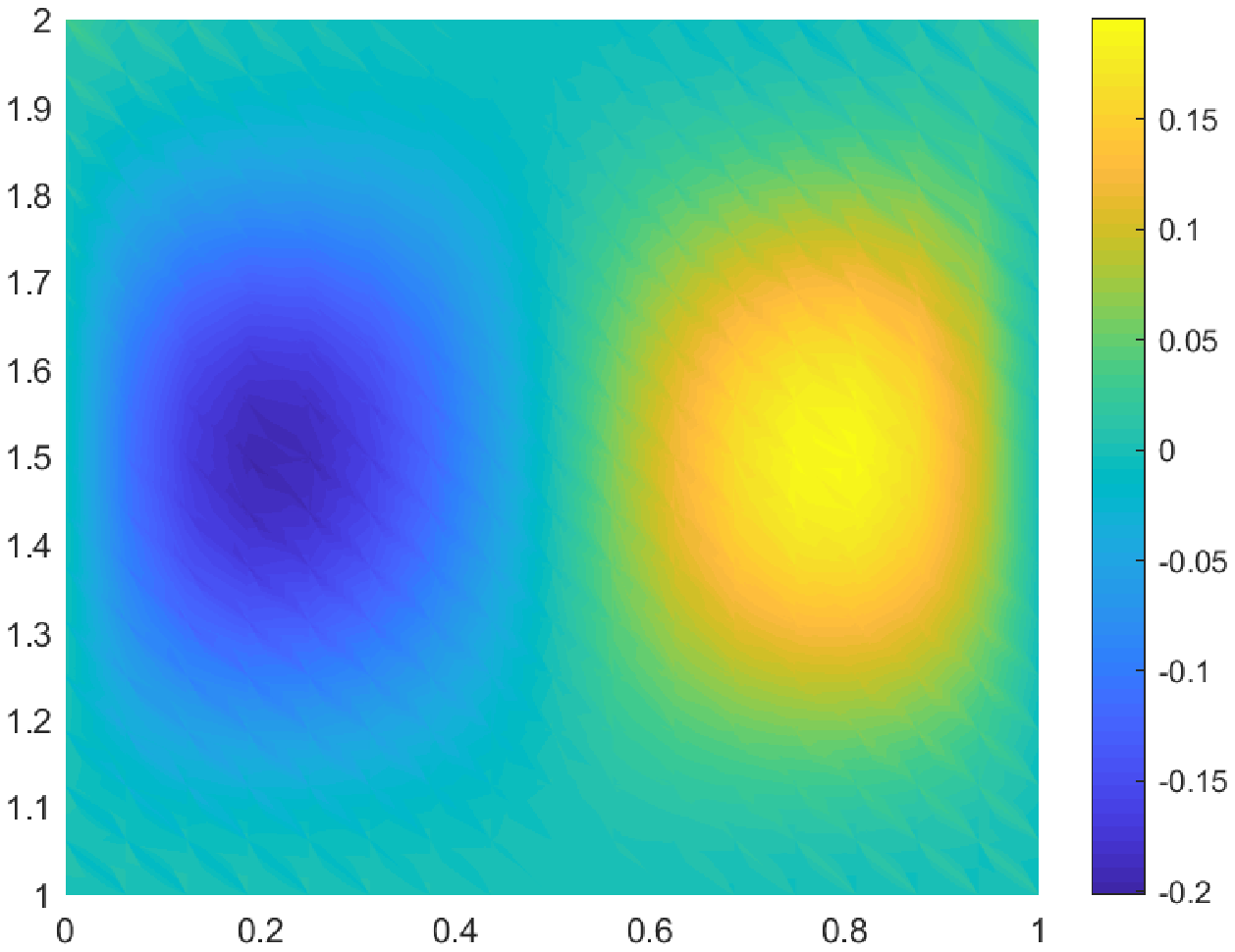}
    \end{minipage}
     \begin{minipage}[b]{0.4\textwidth}
      \includegraphics[width=1\textwidth]{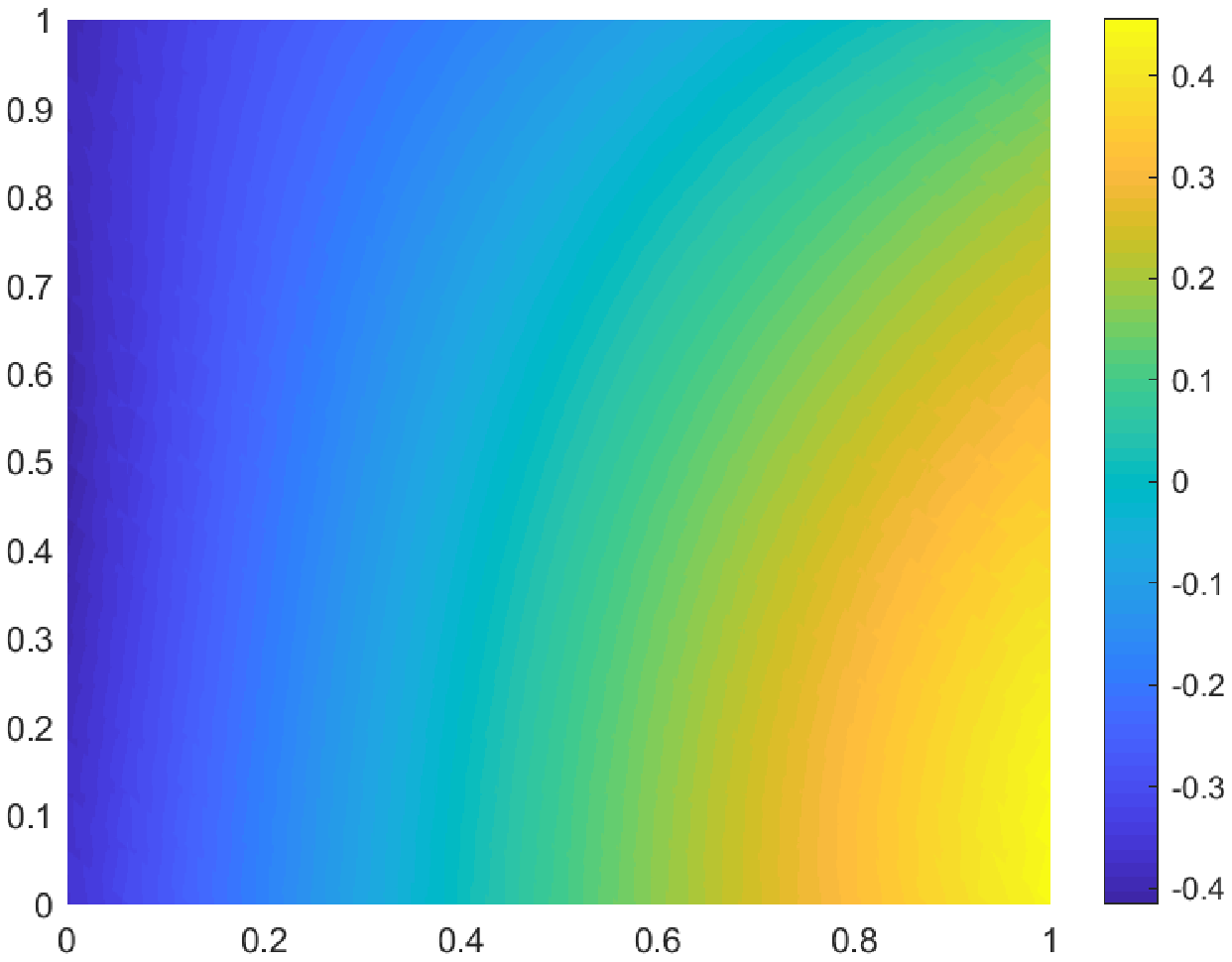}
    \end{minipage}
    \begin{minipage}[b]{0.4\textwidth}
      \includegraphics[width=1\textwidth]{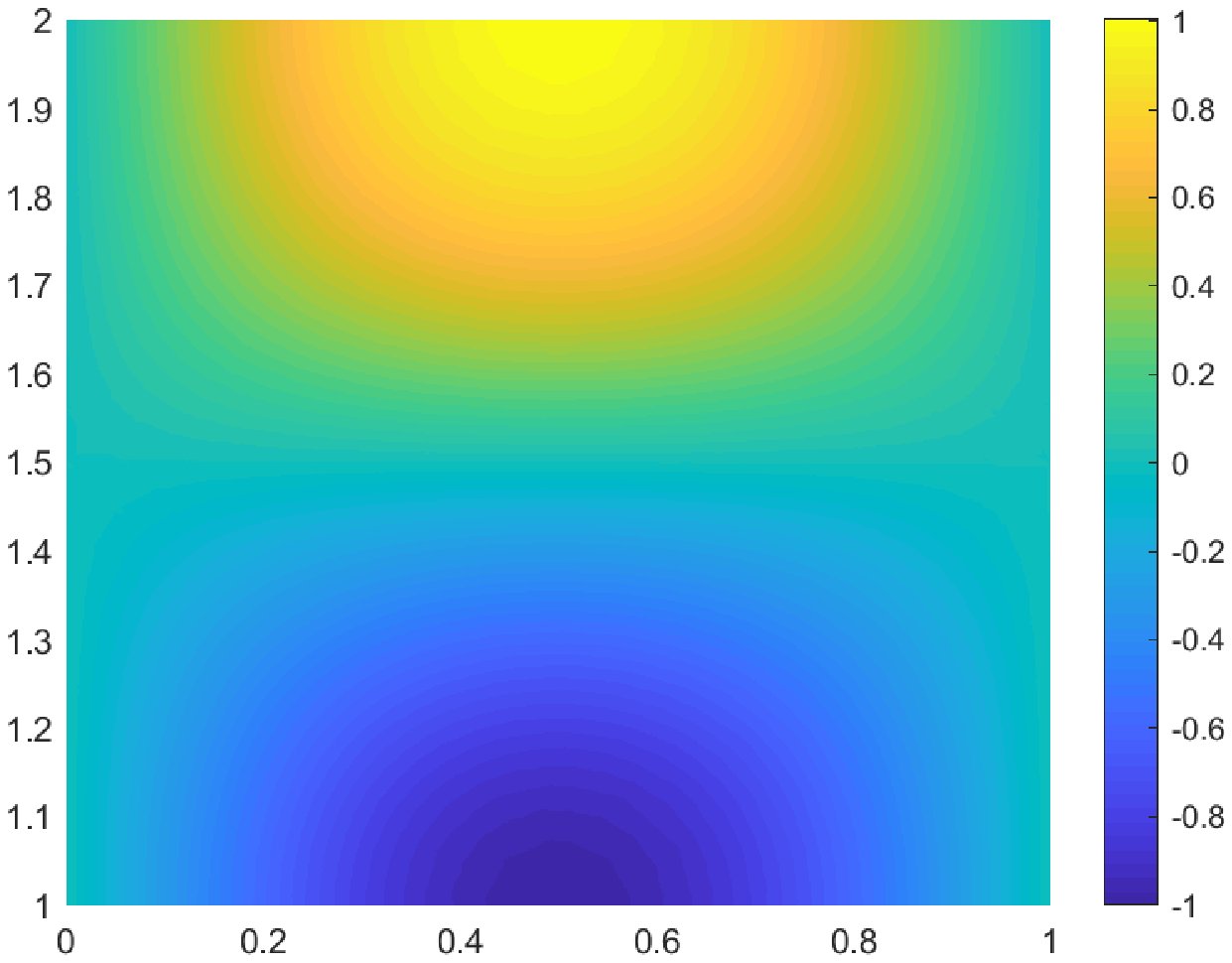}
    \end{minipage}%
  \caption{Numerical solution: $u_{S,h}^1$,$u_{S,h}^2$, $u_{D,h}^{1}$,$u_{D,h}^2$, $p_{S,h}$ and $p_{D,h}$ (left to right, top to bottom).}
\end{figure}
\label{ex3:solution}

\begin{figure}[H]
\centering
\scalebox{0.3}{
\includegraphics[width=20cm]{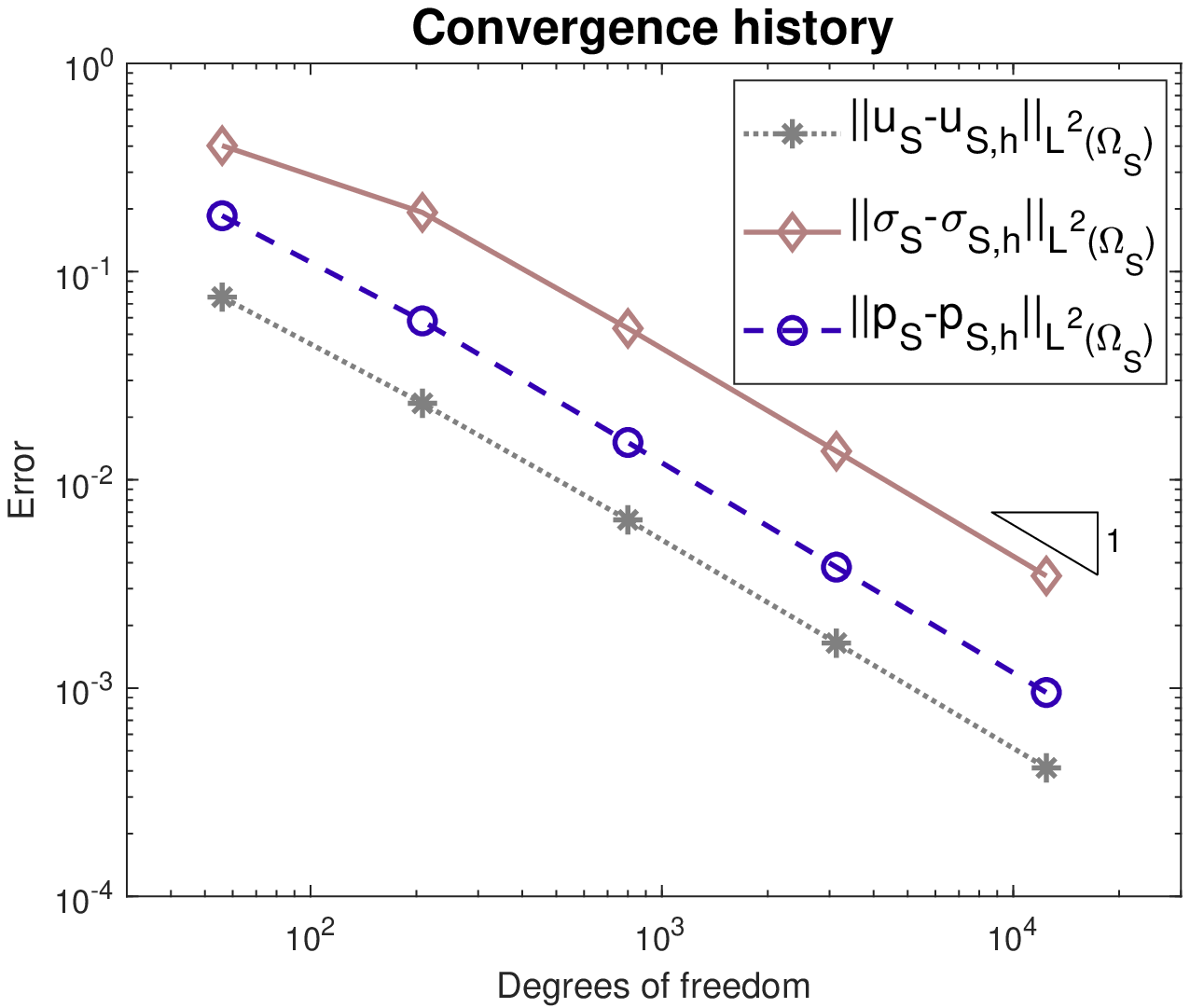}
}
\scalebox{0.3}{
\includegraphics[width=20cm]{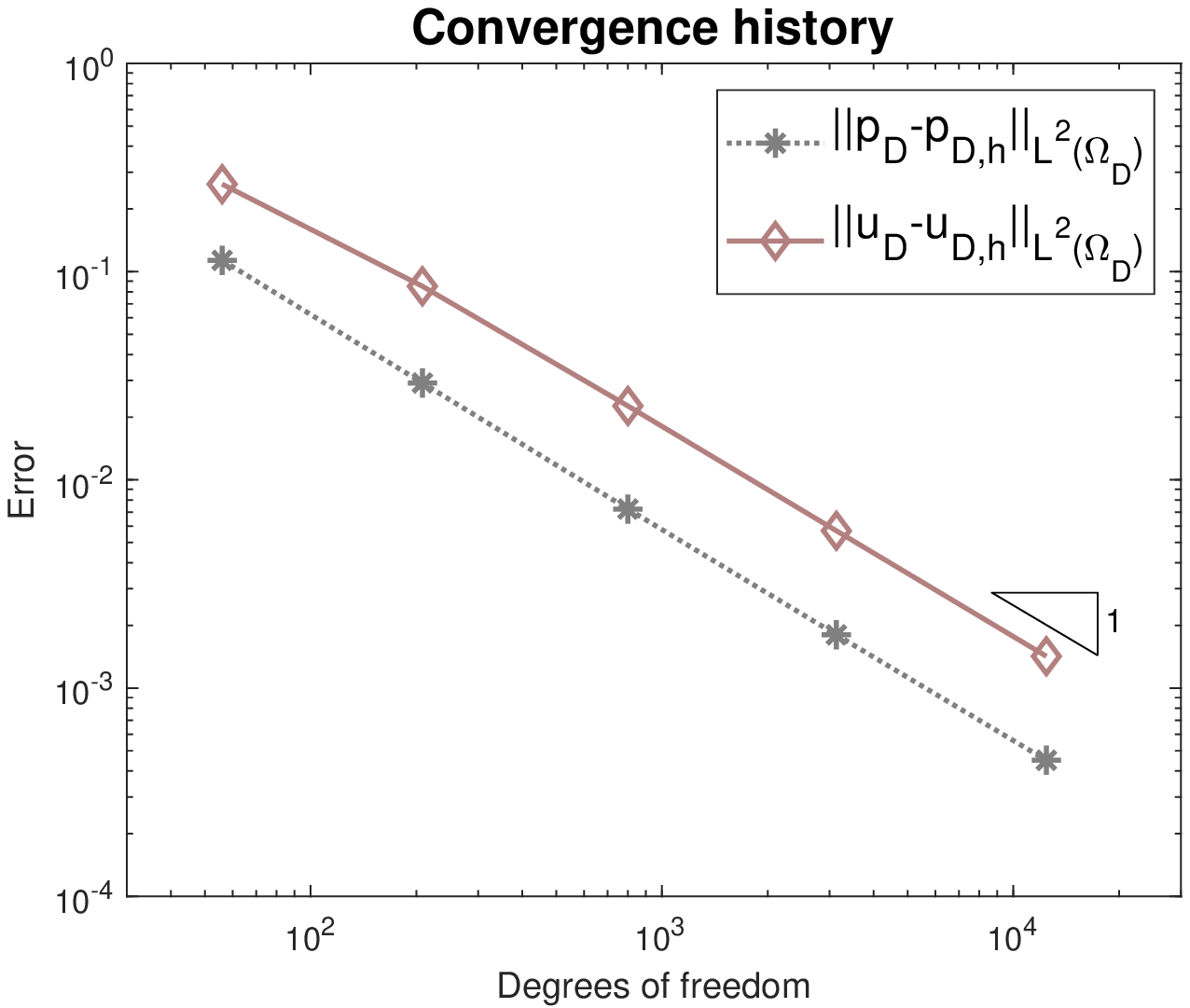}
}
\caption{Convergence history for Example~\ref{ex2}.}
\label{ex3:con}
\end{figure}

Next, we test more challenging examples by using highly oscillatory permeability in the next example.
\begin{example}\label{ex3}

\end{example}
We test our method with $\Omega_S=(0,1)\times(0,1/2)$ and $\Omega_D=(0,1)\times (1/2,1)$. We set $\mu=\rho=\beta=\nu=1$, and the highly oscillatory permeability $K^{-1}=\varrho I$ and $\varrho$ is defined by
\begin{align*}
\varrho= \frac{2+1.8\sin(2\pi x/\epsilon)}{2+1.8\sin(2\pi y/\epsilon)}+\frac{2+1.8\sin(2 \pi y/\epsilon)}{2+1.8\cos(2\pi x/\epsilon)},
\end{align*}
where $\epsilon=1/16$. The profile of $\varrho$ is shown in Figure~\ref{ex4-a}, see also \cite{HouWu97}.
\begin{figure}[H]
\centering
\includegraphics[width=6cm]{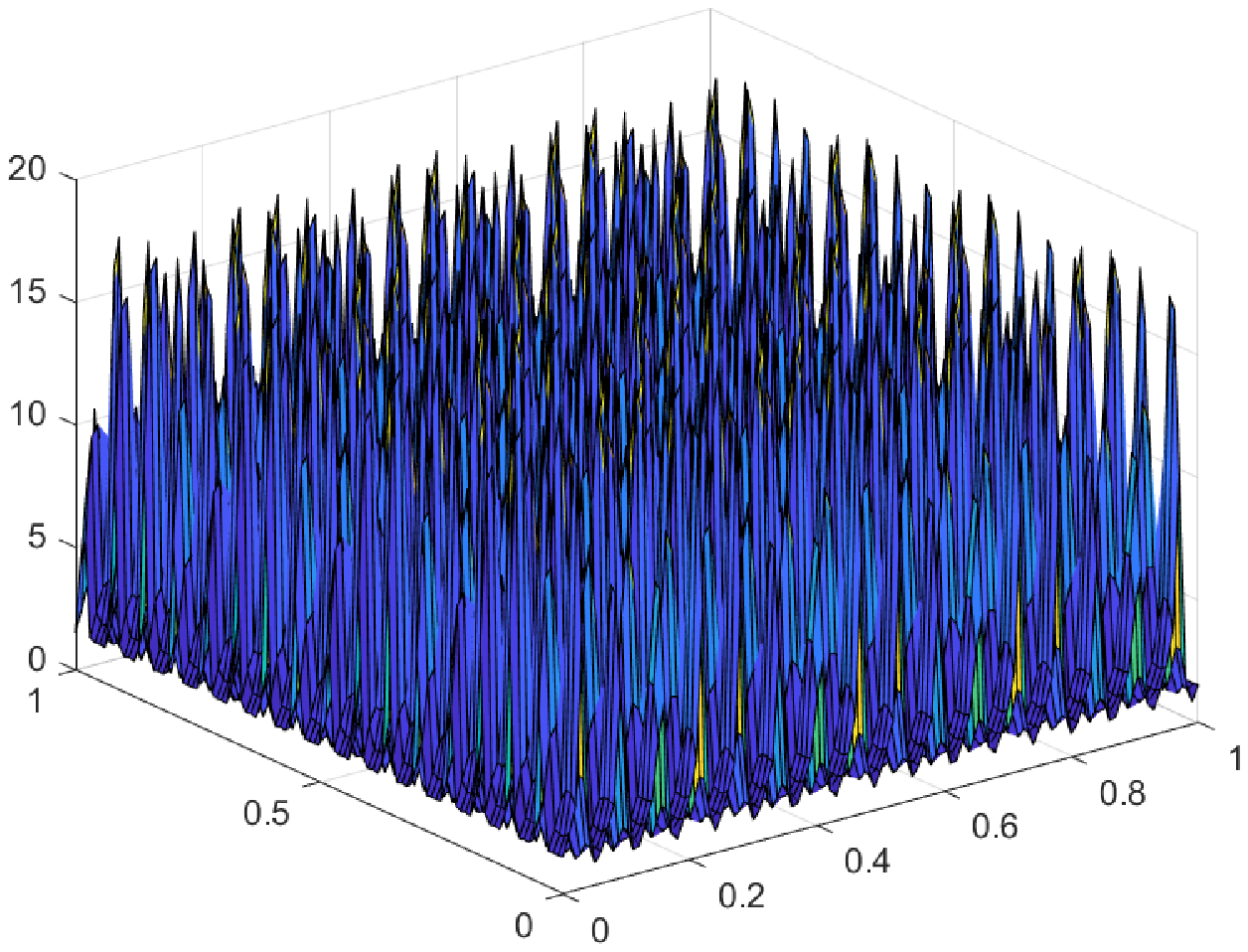}
\caption{Profile of $\varrho$.}
\label{ex4-a}
\end{figure}
The exact solution is given by
\begin{align*}\bm{u}_S=
\left(
  \begin{array}{c}
    16y\cos(\pi x)^2(y^2-0.25) \\
     8\pi\cos(\pi x)\sin(\pi x)(y^2-0.25)^2 \\
  \end{array}
\right),\quad p_S=x^2.
\end{align*}
and
\begin{align*}
\bm{u}_D=
\left(
  \begin{array}{c}
    \sin(2\pi x)\cos(2\pi y) \\
    -\cos(2\pi x)\sin(2\pi y) \\
  \end{array}
\right),\quad p_D=\cos(2 \pi x)\cos(2\pi y).
\end{align*}

The numerical approximations are shown in Figure~\ref{ex4:solution} and the convergence history against the number of degrees of freedom are displayed in Figure~\ref{ex4:con}. It can be observed that the optimal convergence can be achieved. This example once again highlights that our method can achieve the optimal convergence, in addition, it is robust with respect to the permeability.
\begin{figure}[H]
    \centering
    \begin{minipage}[b]{0.4\textwidth}
      \includegraphics[width=1\textwidth]{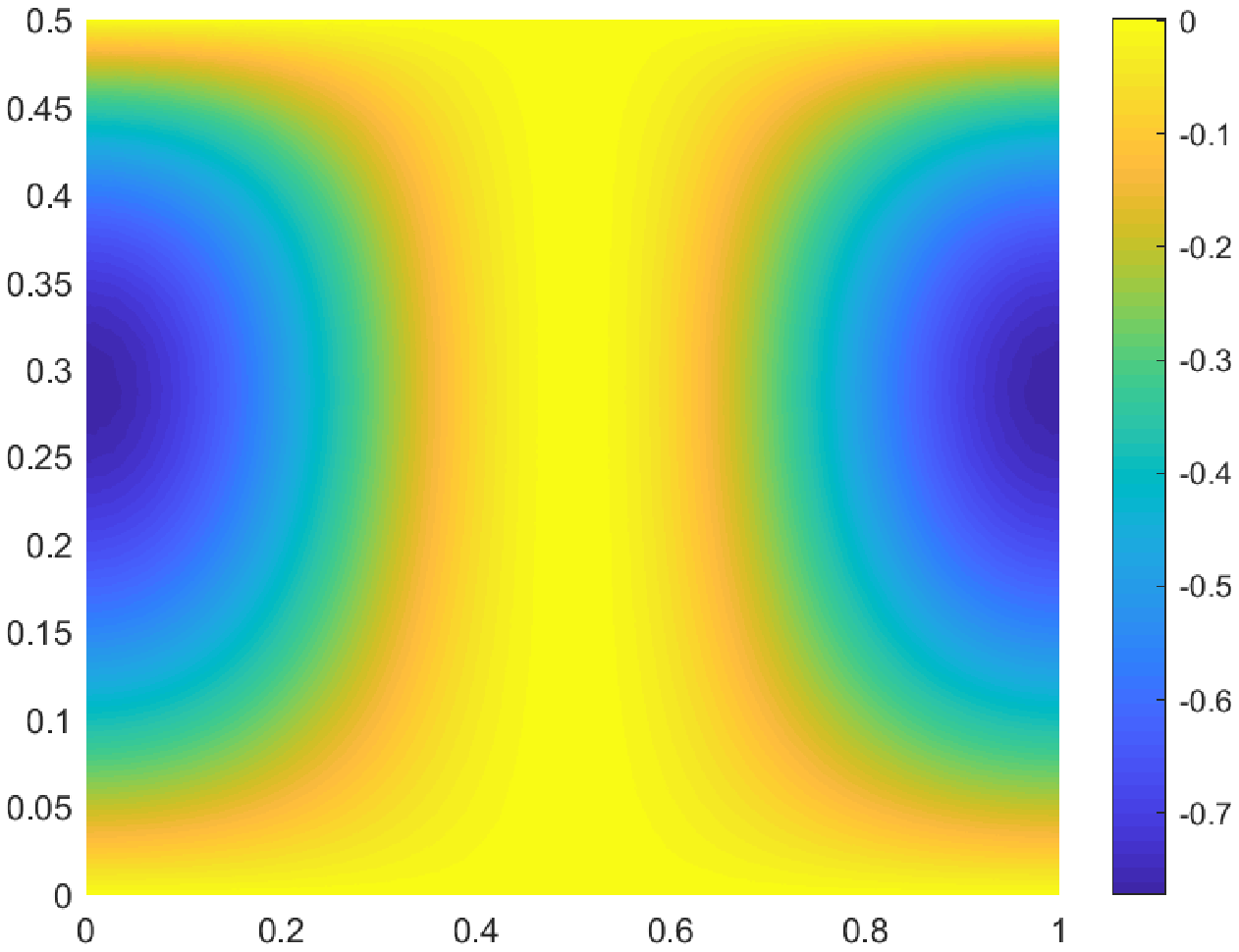}
    \end{minipage}%
    \begin{minipage}[b]{0.4\textwidth}
      \includegraphics[width=1\textwidth]{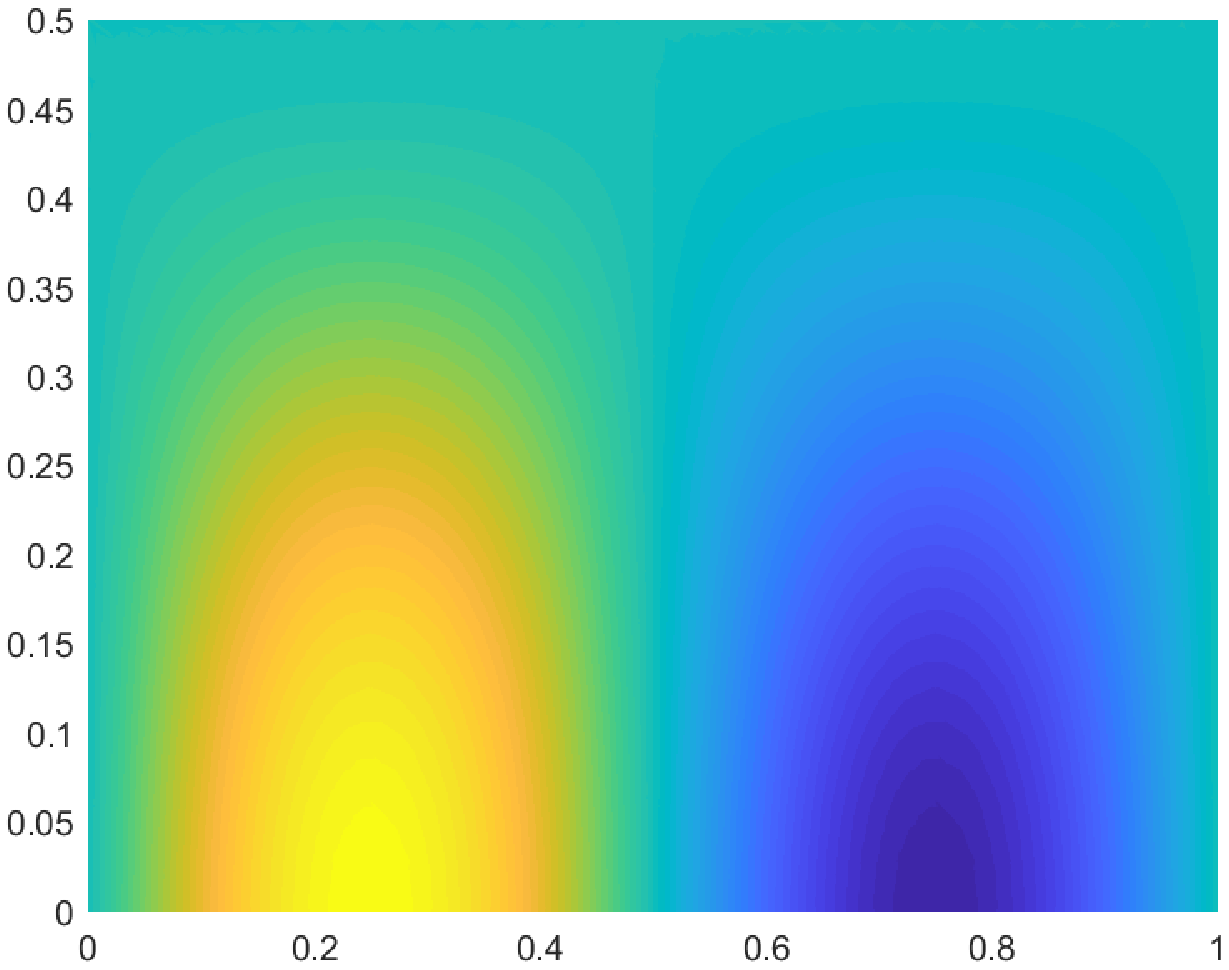}
    \end{minipage}
    \begin{minipage}[b]{0.4\textwidth}
      \includegraphics[width=1\textwidth]{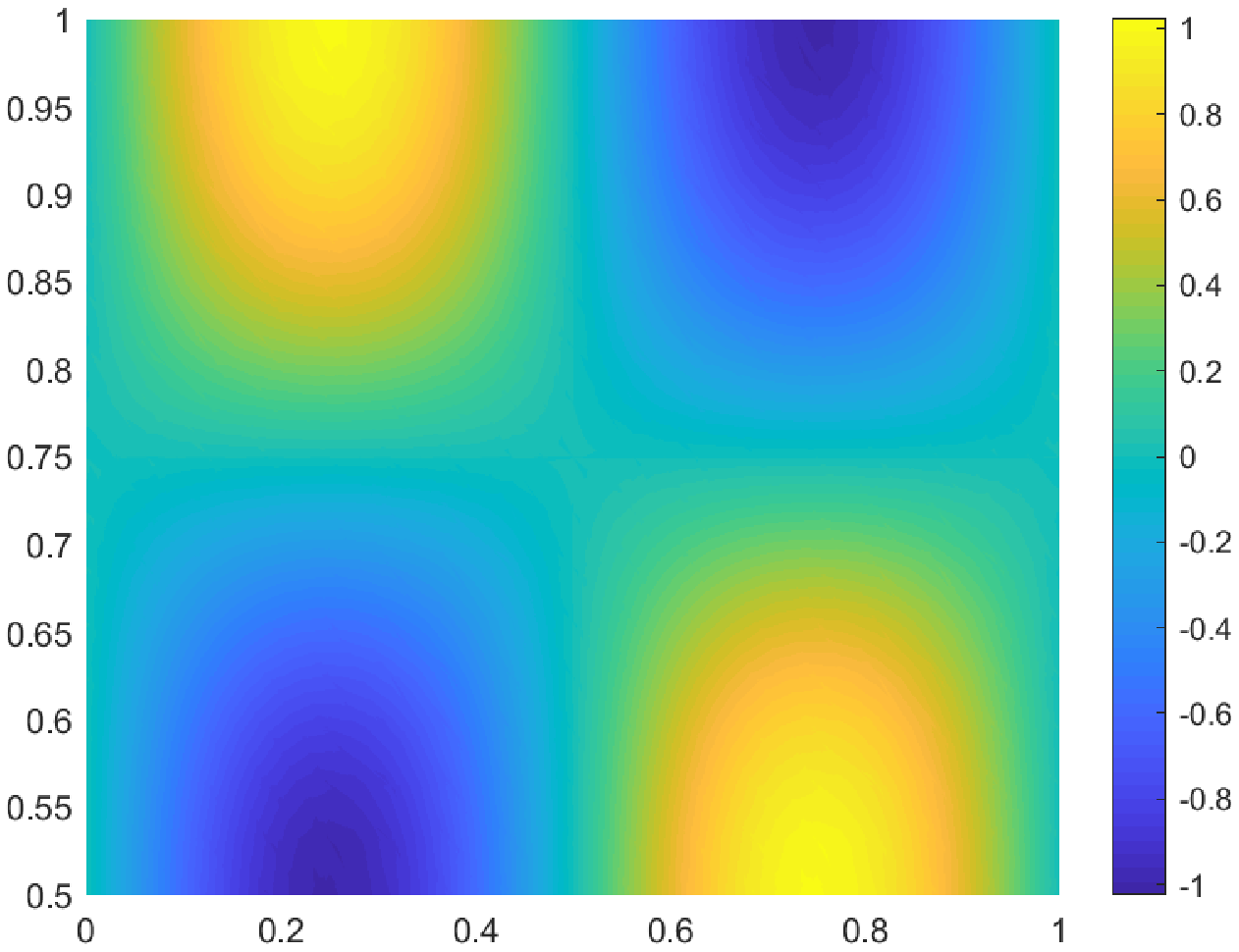}
    \end{minipage}
     \begin{minipage}[b]{0.4\textwidth}
      \includegraphics[width=1\textwidth]{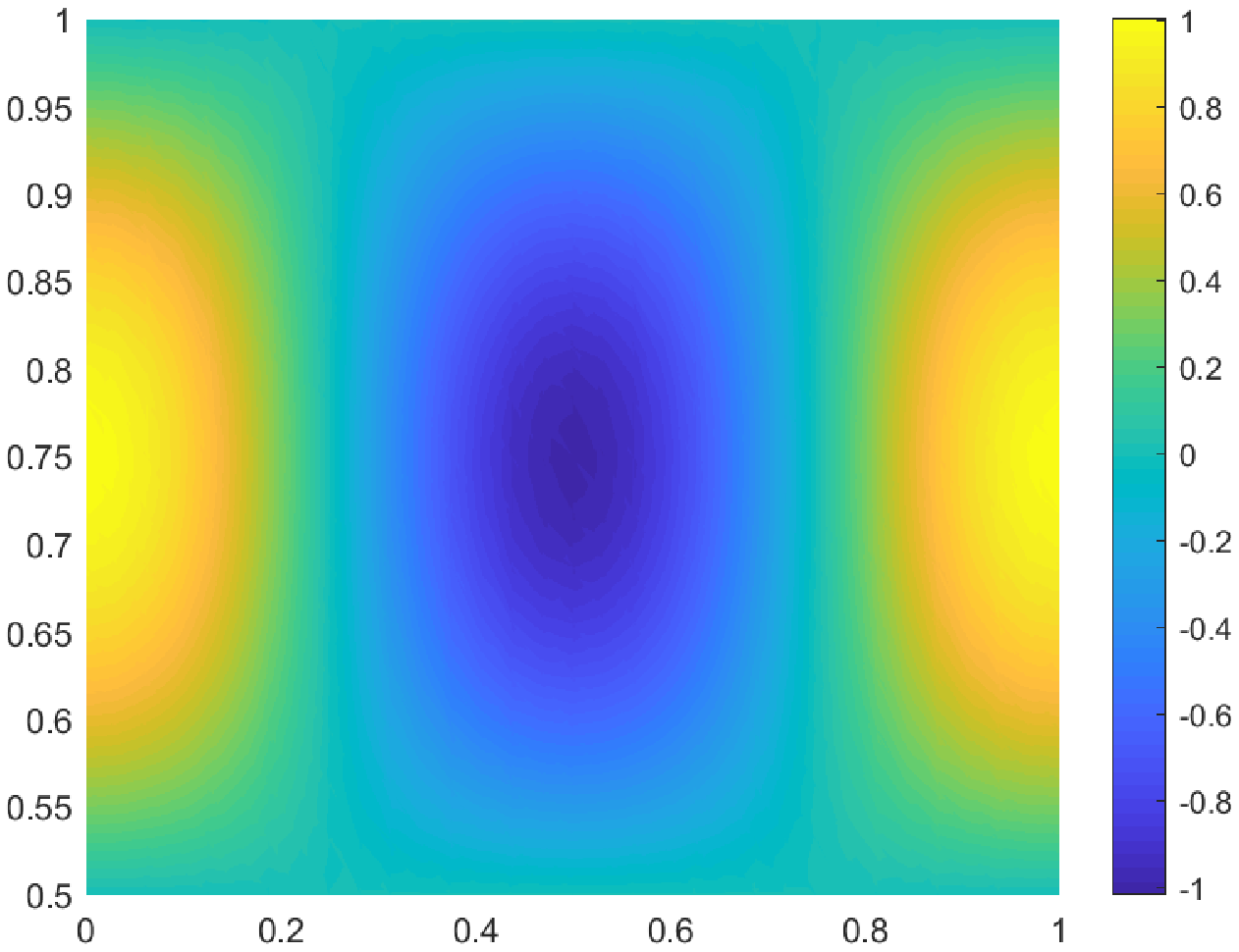}
    \end{minipage}
     \begin{minipage}[b]{0.4\textwidth}
      \includegraphics[width=1\textwidth]{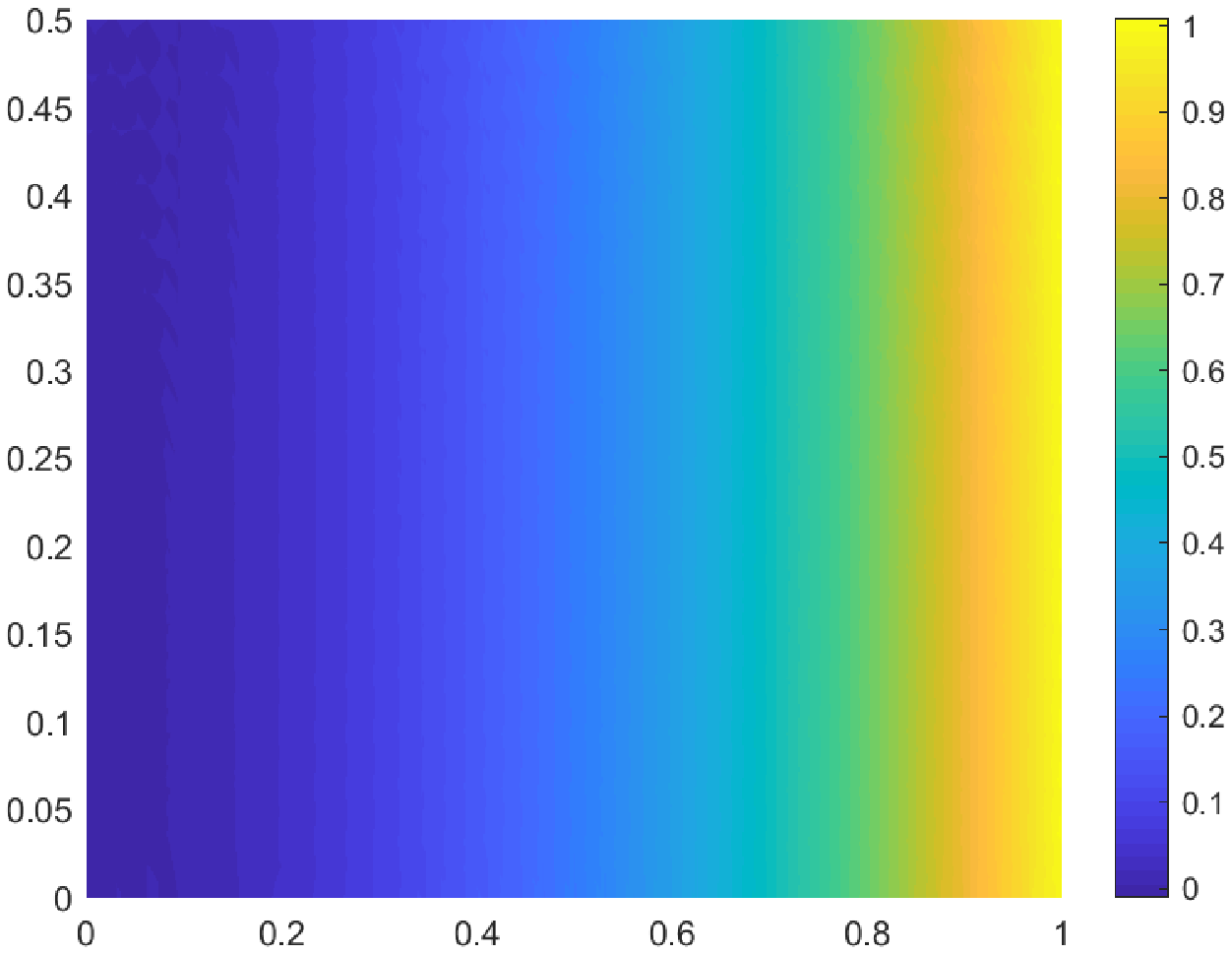}
    \end{minipage}
    \begin{minipage}[b]{0.4\textwidth}
      \includegraphics[width=1\textwidth]{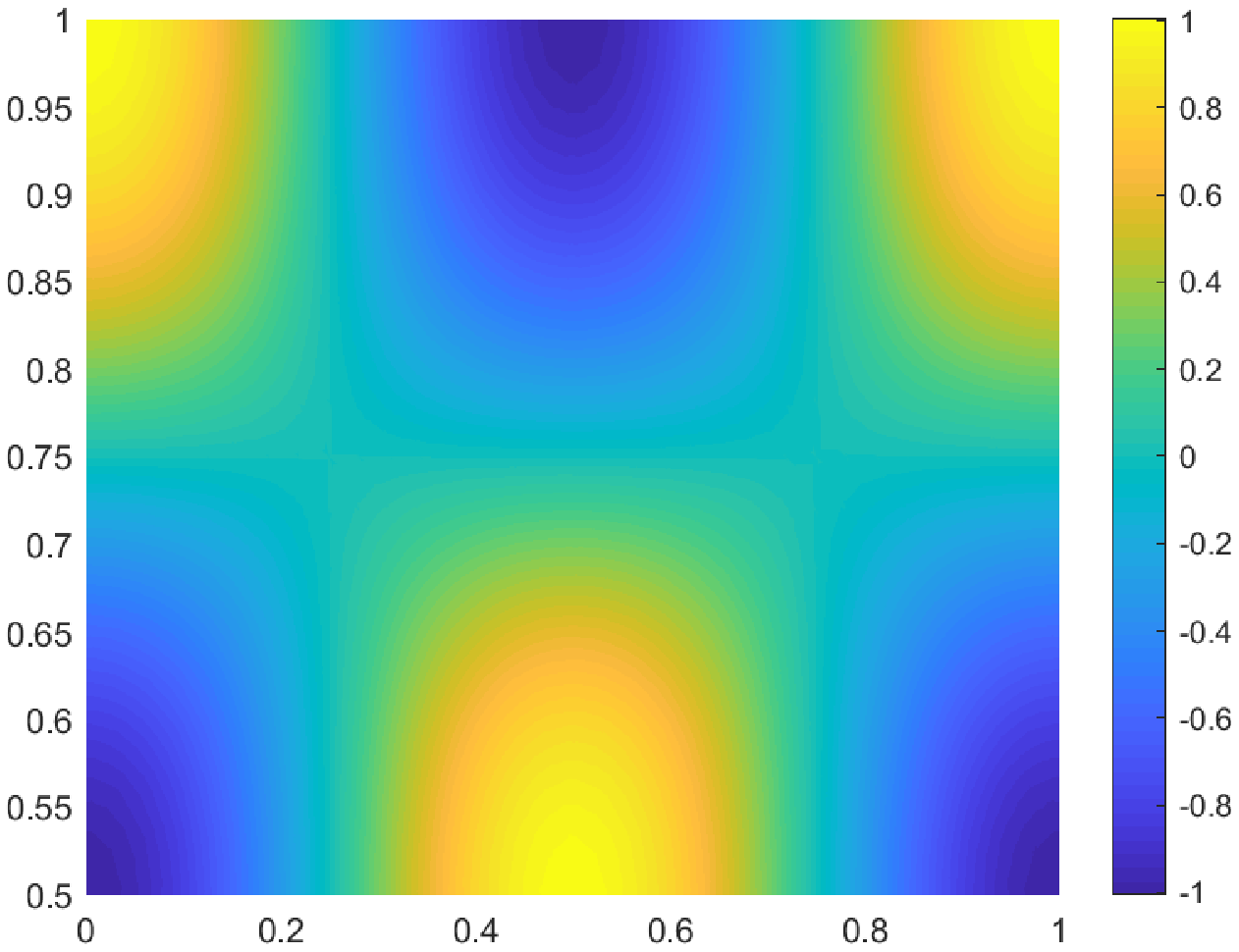}
    \end{minipage}%
  \caption{Numerical solution: $u_{S,h}^1$,$u_{S,h}^2$, $u_{D,h}^{1}$,$u_{D,h}^2$, $p_{S,h}$ and $p_{D,h}$ (left to right, top to bottom).}
  \label{ex4:solution}
\end{figure}

\begin{figure}[H]
\centering
\scalebox{0.3}{
\includegraphics[width=20cm]{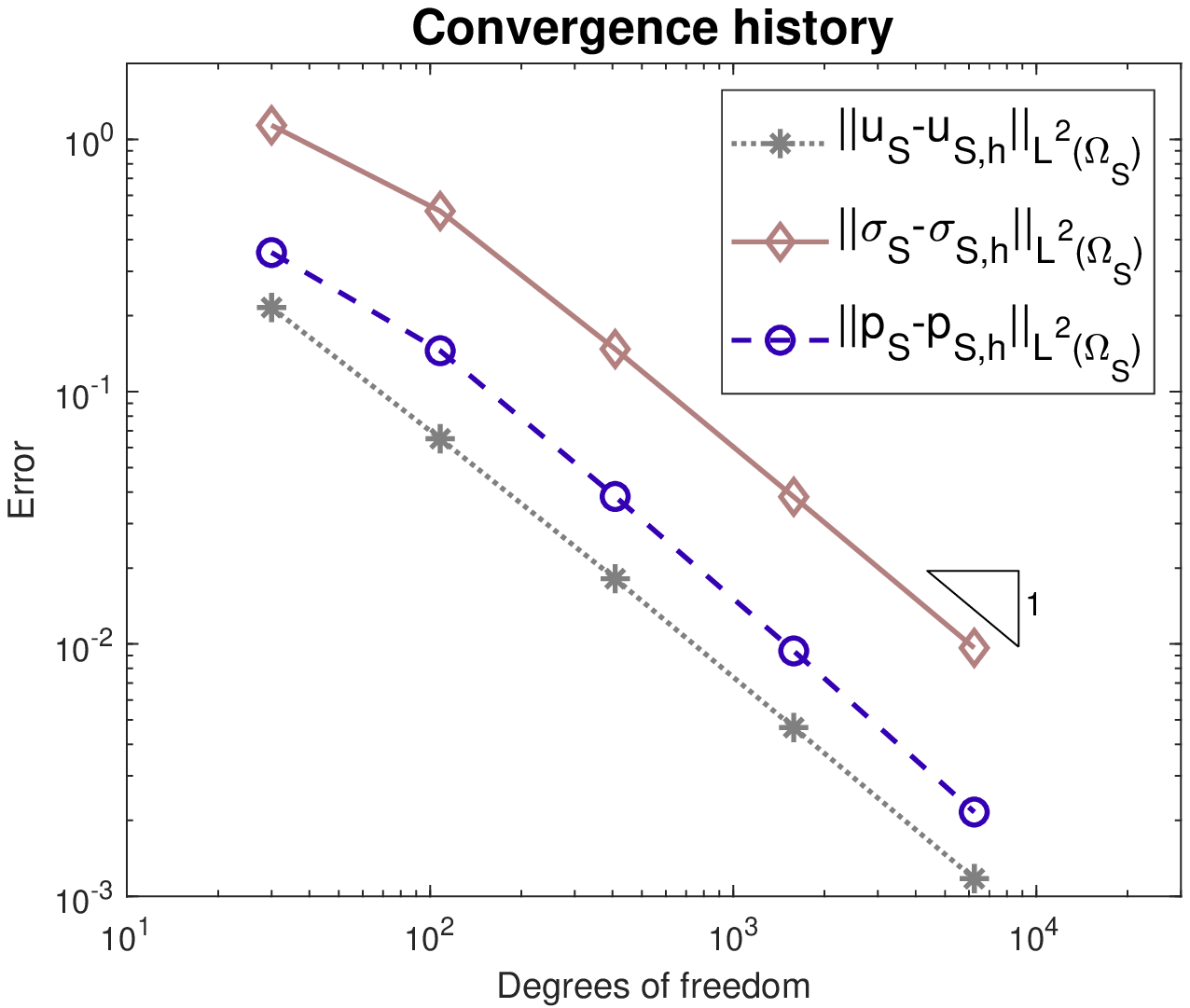}
}
\scalebox{0.3}{
\includegraphics[width=20cm]{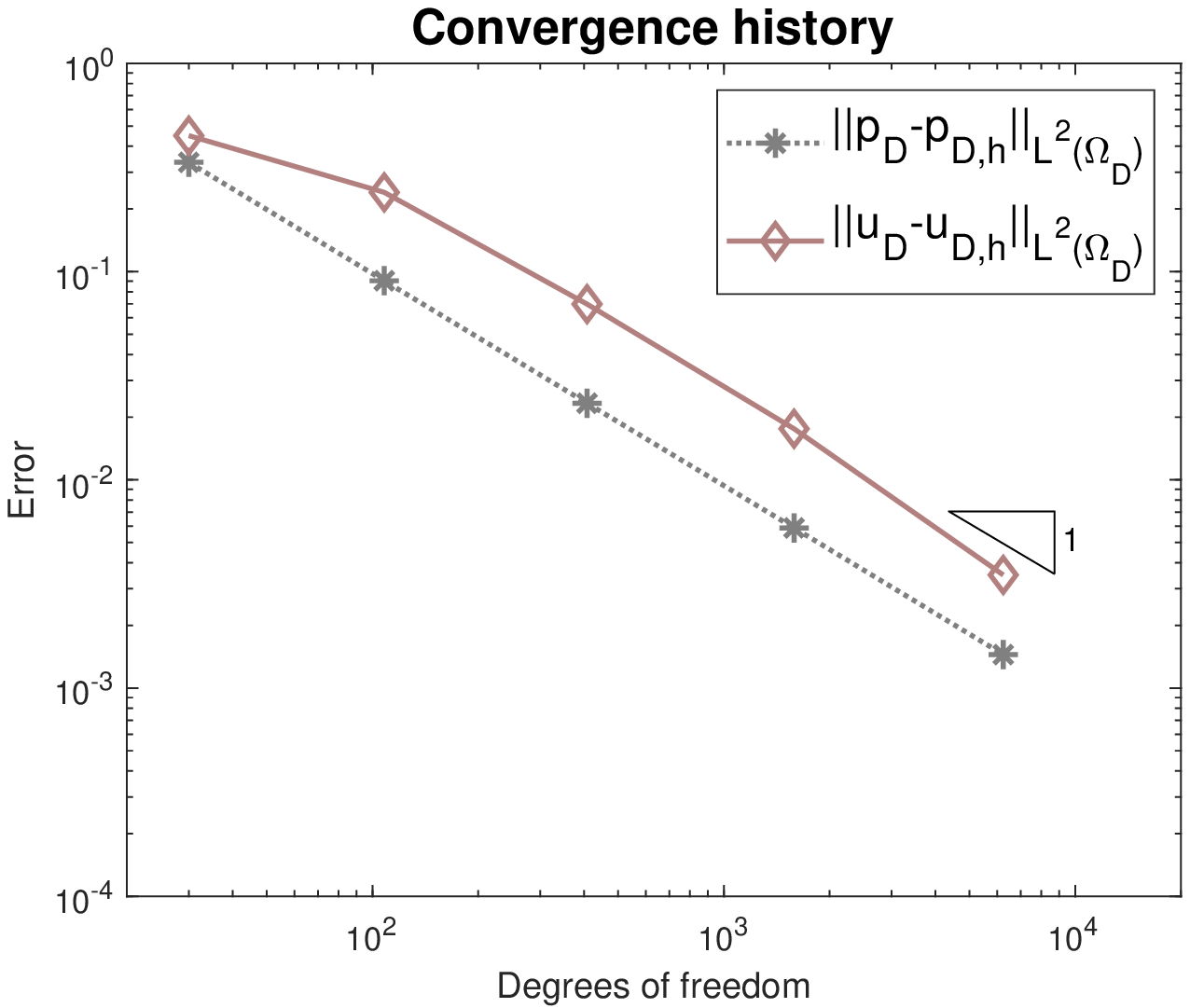}
}
\caption{Convergence history for Example~\ref{ex3}.}
\label{ex4:con}
\end{figure}

Finally, we exploit one example with high contrast permeability to illustrate that our method could be a viable option for the practical applications.
\begin{example}

\end{example}

Here, we consider the two-dimension domain with $\Omega_S=(-1/2,3/2)\times (0,2)$ and $\Omega_D=(-1/2,3/2)\times (-2,0)$. In the Stokes region, the Dirichlet boundary condition is given by Kovasznay flow (cf. \cite{Kovasznay48}).
\begin{align*}
\bm{u}_S=
\left(
  \begin{array}{c}
    1-e^{\lambda x}\cos(2\pi y) \\
    \frac{\lambda}{2\pi}e^{\lambda x}\sin(2\pi y) \\
  \end{array}
\right),
\end{align*}
%\begin{align*}
%p_S=-\frac{1}{2}e^{2\lambda x}.
%\end{align*}
where $\lambda =\frac{-8\pi^2}{\epsilon^{-1}+\sqrt{\epsilon^{-2}+64\pi^2}}$. In our test, we take $\epsilon=1$.

We assume that the homogeneous interface condition without source is satisfied
\begin{align*}
\bm{g}_D=0,\; f_D=0,\; \bm{f}_S=0.
\end{align*}
In addition
$\bm{p}_D$ satisfies homogeneous Dirichlet boundary condition along $y=-2$, otherwise it has homogeneous Neumann boundary condition. We remark that in this example, the exact solution is unknown.

The permeability is taken to be $K=\varrho I$ and the profile of $\varrho$ is shown in Figure~\ref{ex5-a}, where the yellow region represents the permeability with value $10^4$.
\begin{figure}[H]
\centering
\includegraphics[width=6cm]{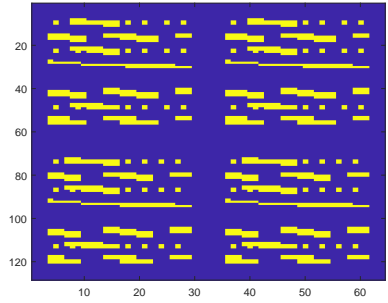}
\caption{Profile of $\varrho$.}
\label{ex5-a}
\end{figure}

The numerical approximations for meshsize $h=1/16$ and $h=1/32$ are reported in Figure~\ref{ex5:solution1} and Figure~\ref{ex5:solution2}, respectively. We can observe that the numerical approximations tend to approach a certain solution.
\begin{figure}[H]
    \centering
    \begin{minipage}[b]{0.4\textwidth}
      \includegraphics[width=1\textwidth]{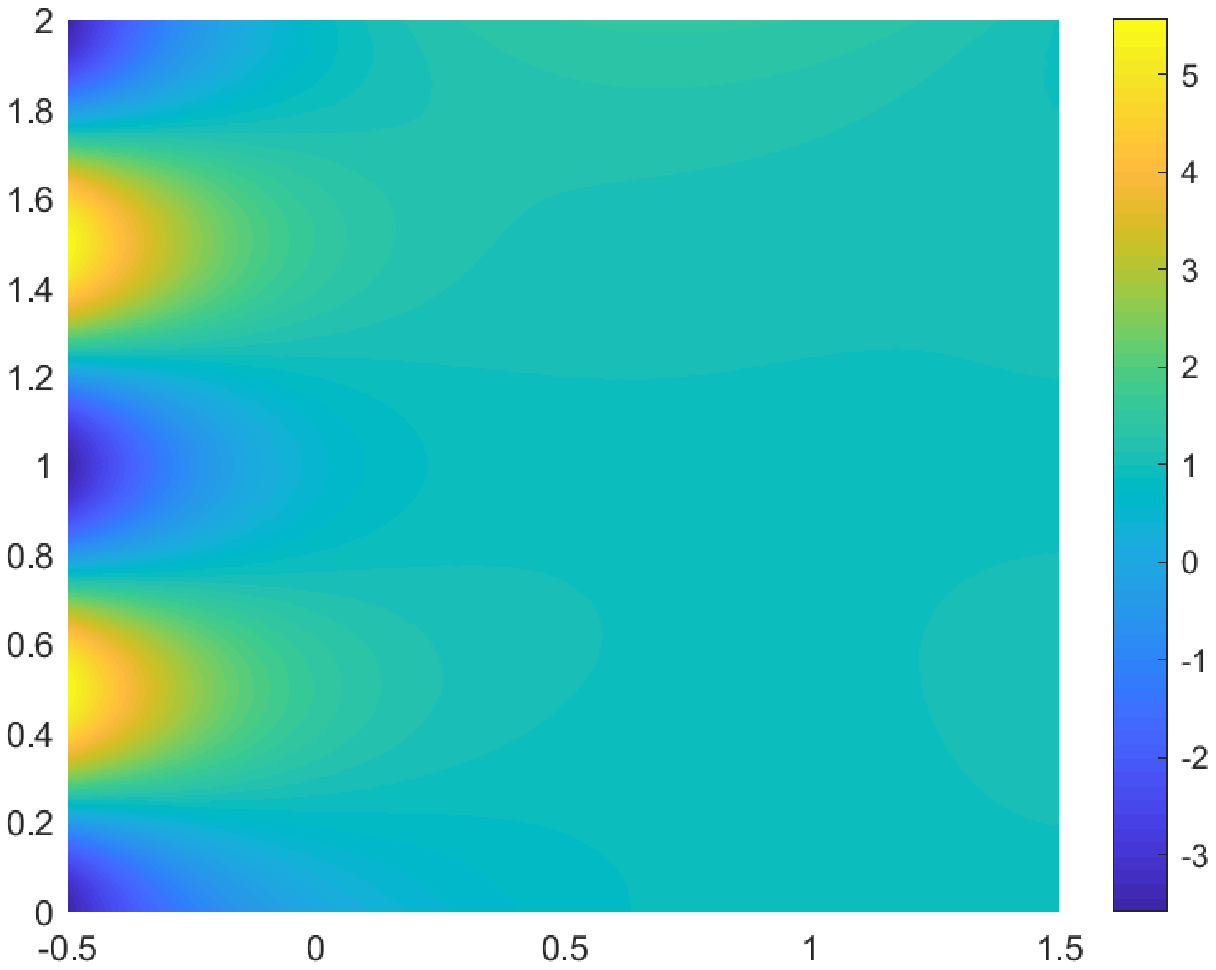}
    \end{minipage}%
    \begin{minipage}[b]{0.4\textwidth}
      \includegraphics[width=1\textwidth]{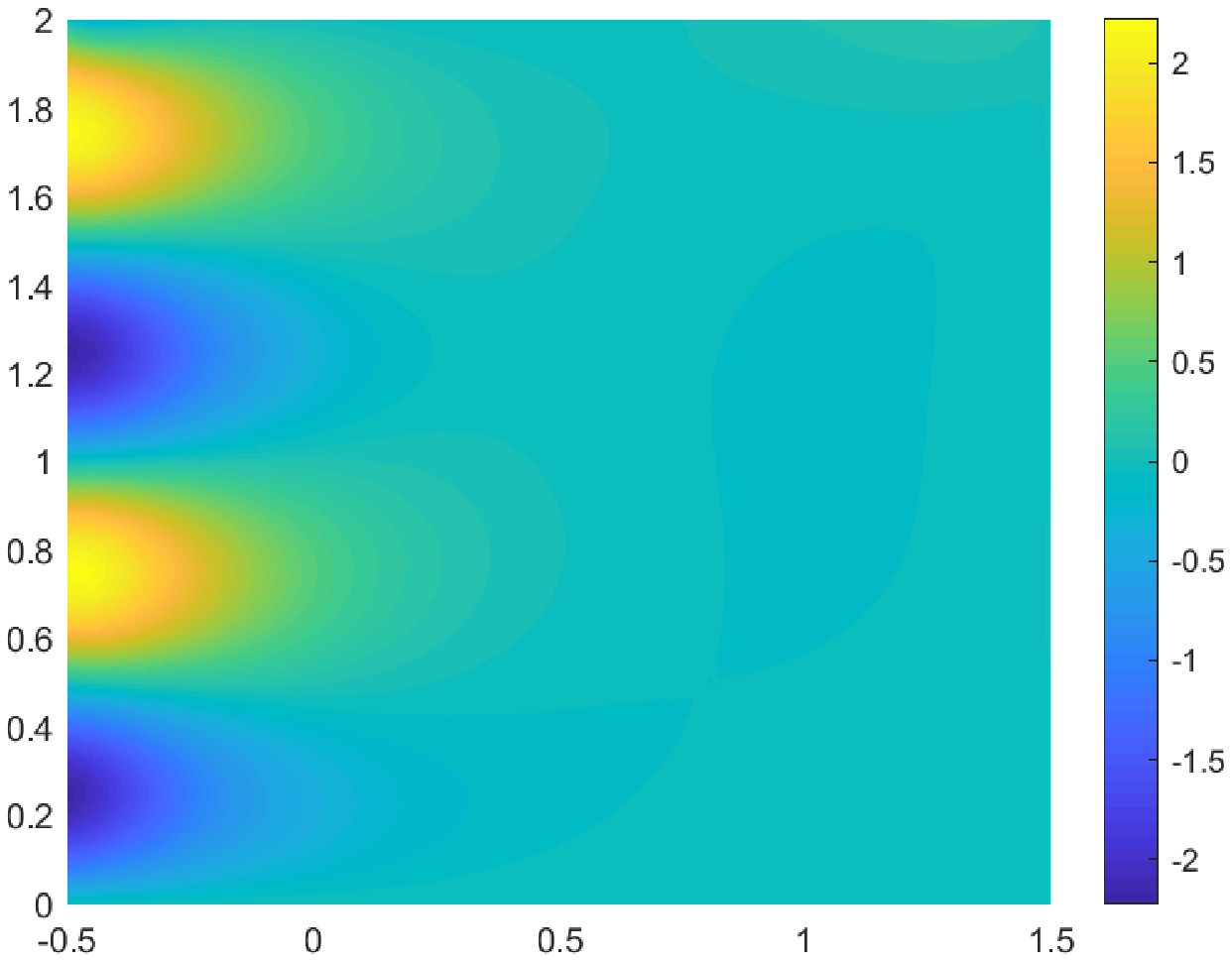}
    \end{minipage}
     \begin{minipage}[b]{0.4\textwidth}
      \includegraphics[width=1\textwidth]{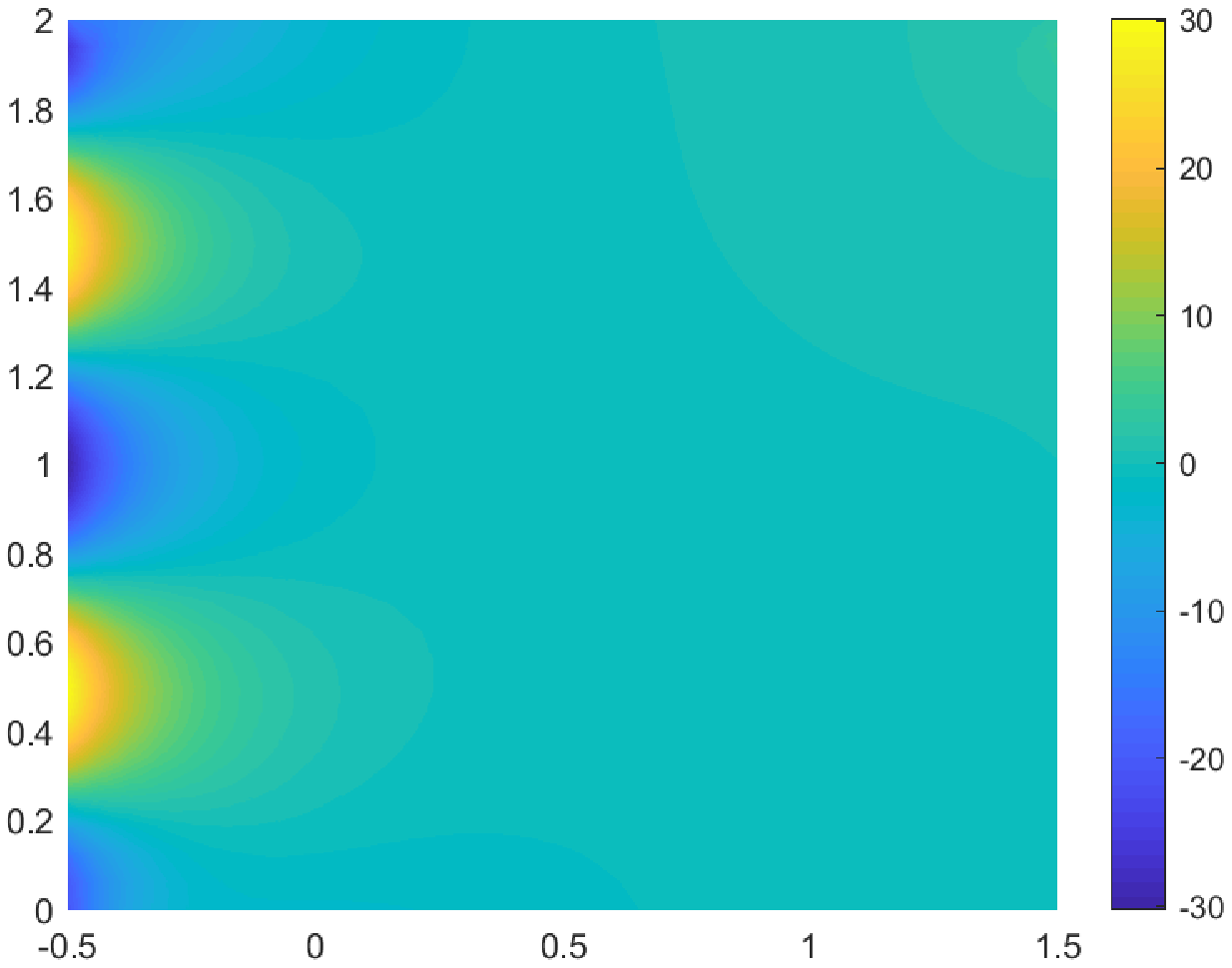}
    \end{minipage}
    \begin{minipage}[b]{0.4\textwidth}
      \includegraphics[width=1\textwidth]{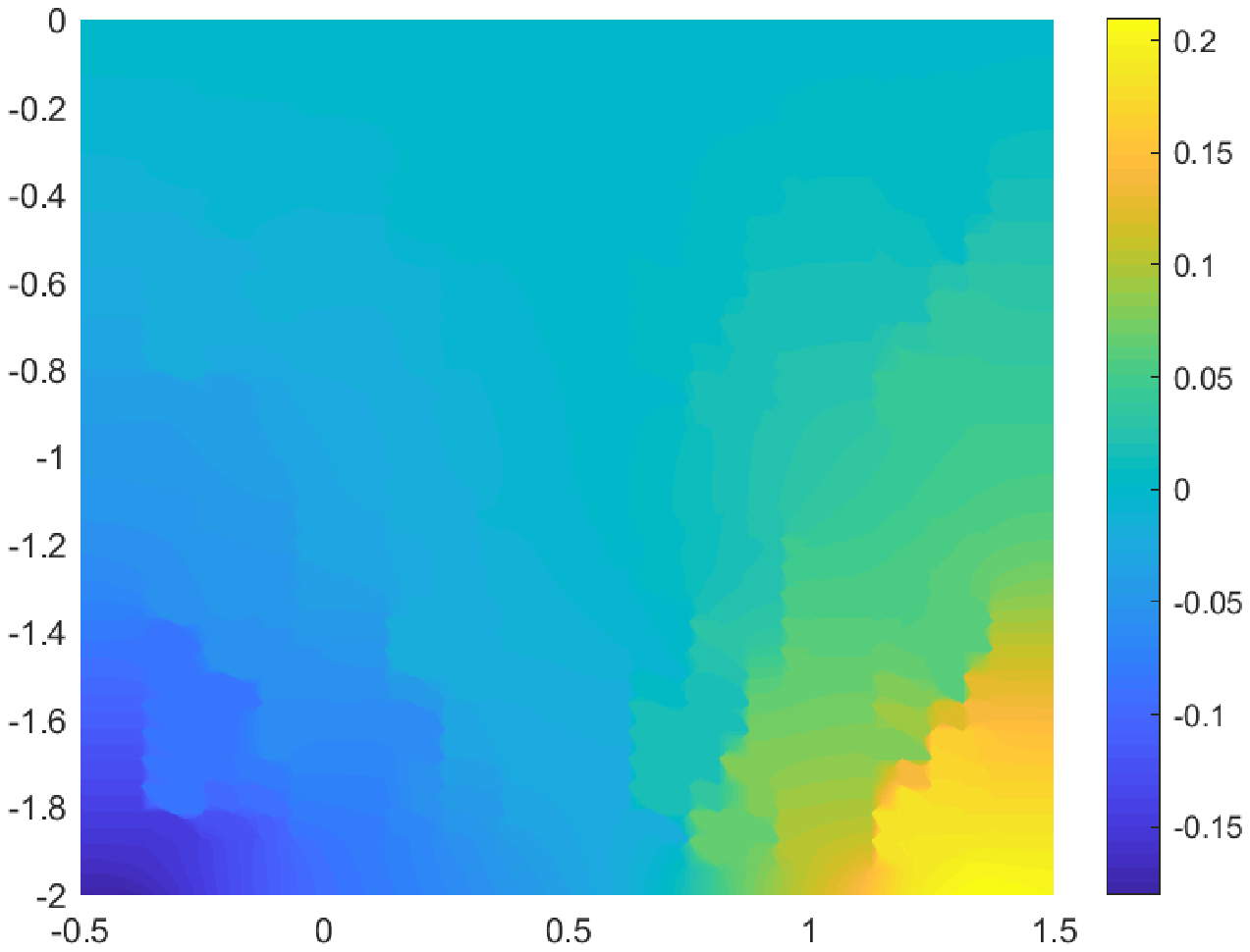}
    \end{minipage}%
  \caption{Numerical solution: $u_{S,h}^1$,$u_{S,h}^2$, $p_{S,h}$ and $p_{D,h}$ (left to right, top to bottom).}
  \label{ex5:solution1}
\end{figure}

\begin{figure}[H]
    \centering
    \begin{minipage}[b]{0.4\textwidth}
      \includegraphics[width=1\textwidth]{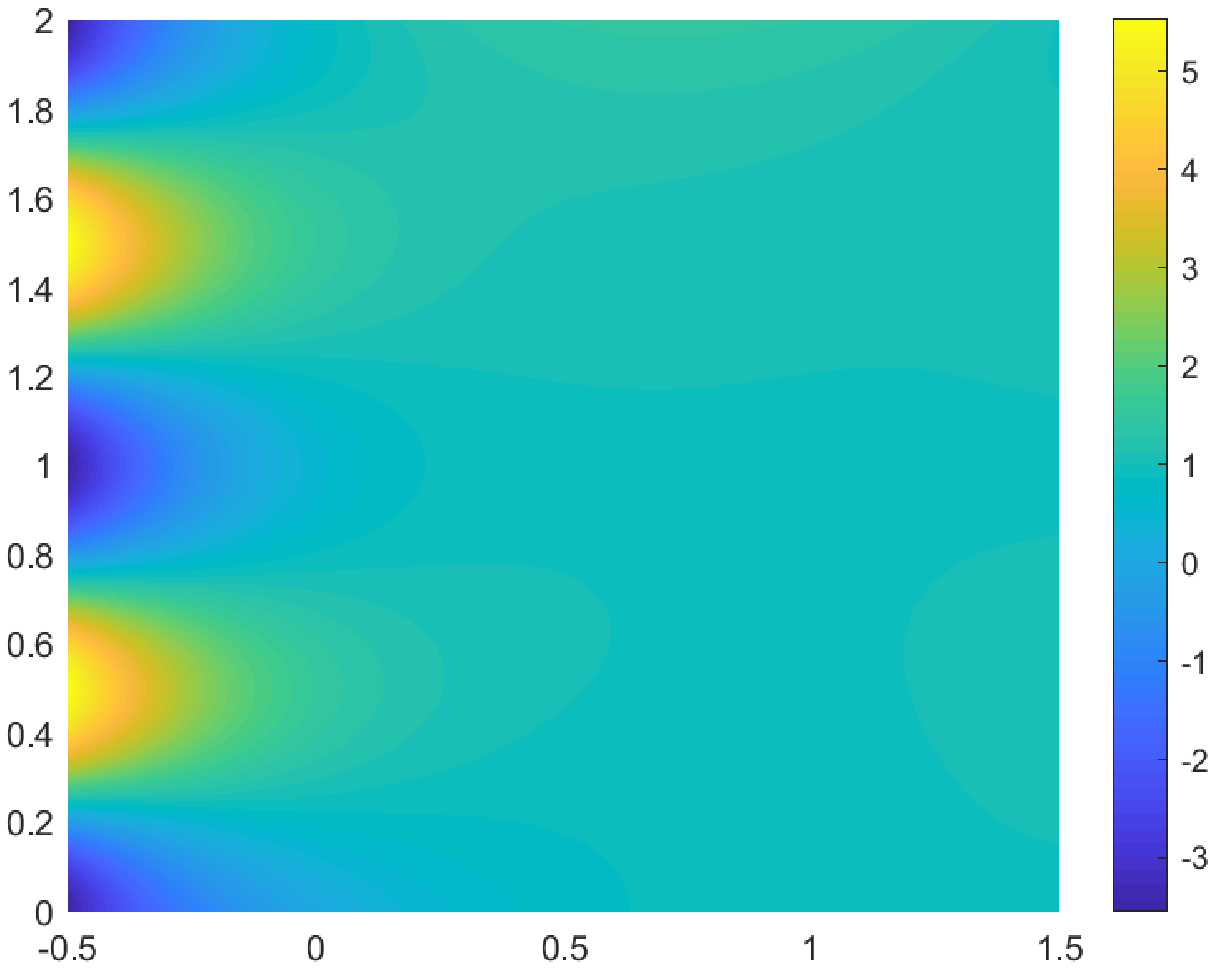}
    \end{minipage}%
    \begin{minipage}[b]{0.4\textwidth}
      \includegraphics[width=1\textwidth]{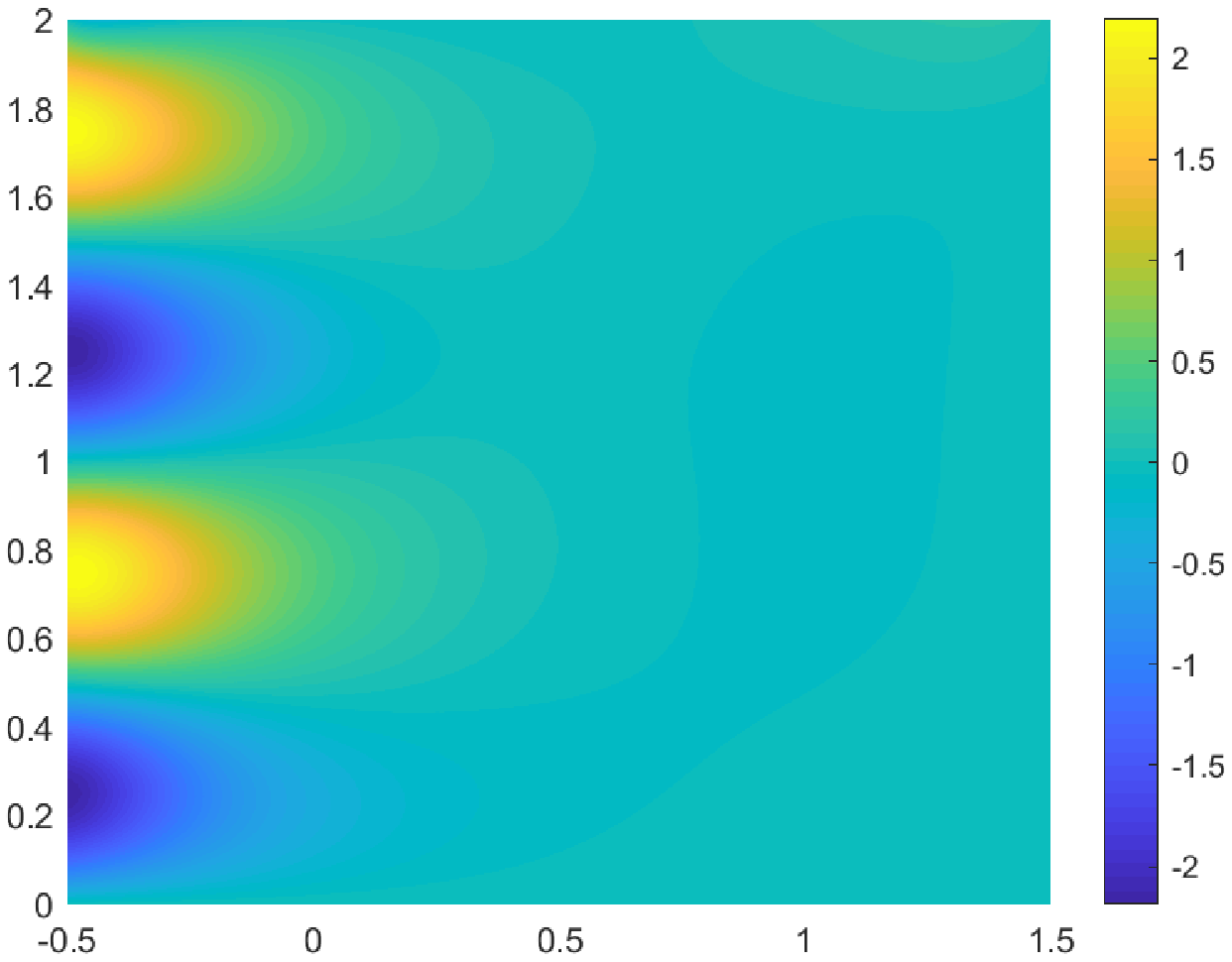}
    \end{minipage}
     \begin{minipage}[b]{0.4\textwidth}
      \includegraphics[width=1\textwidth]{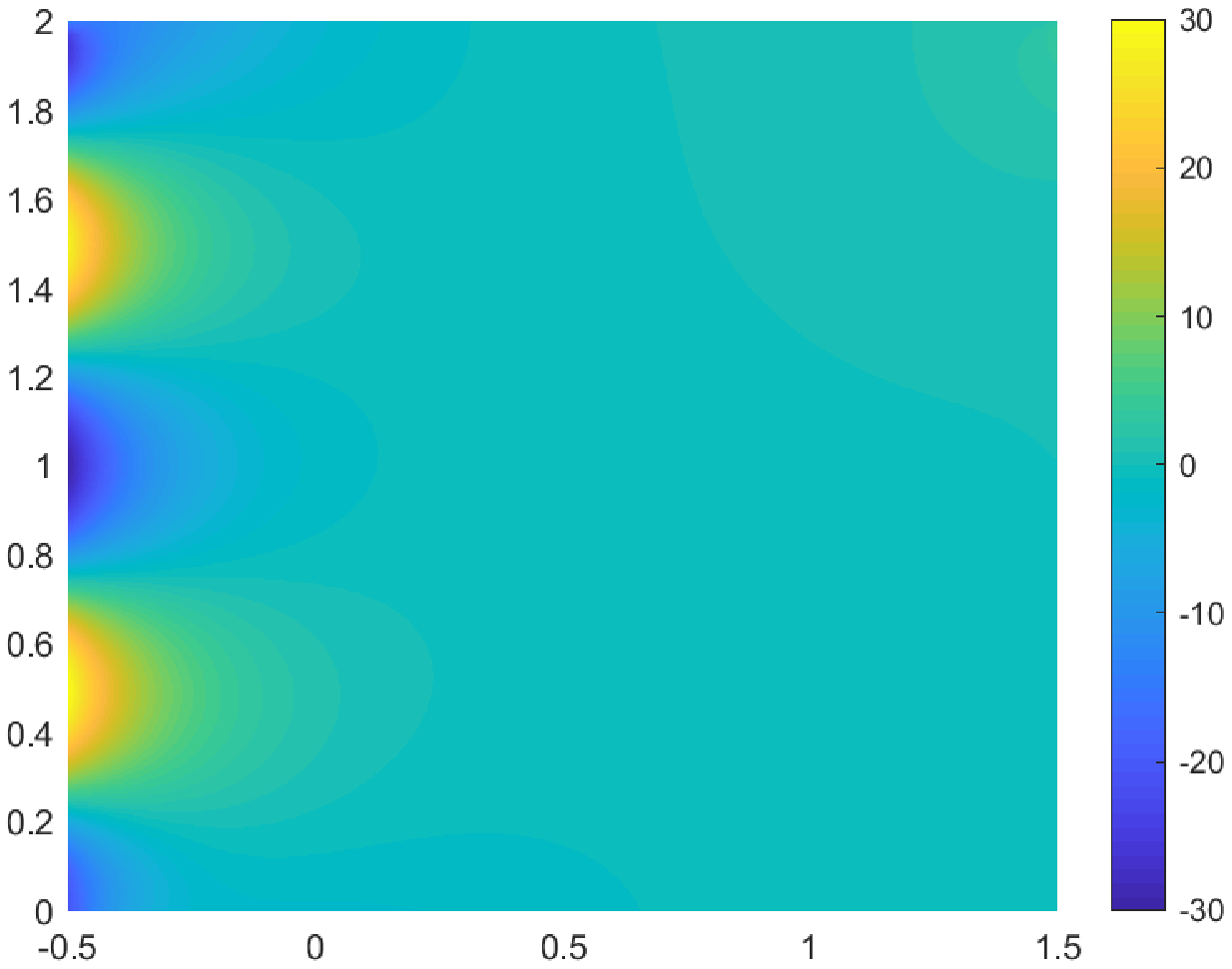}
    \end{minipage}
    \begin{minipage}[b]{0.4\textwidth}
      \includegraphics[width=1\textwidth]{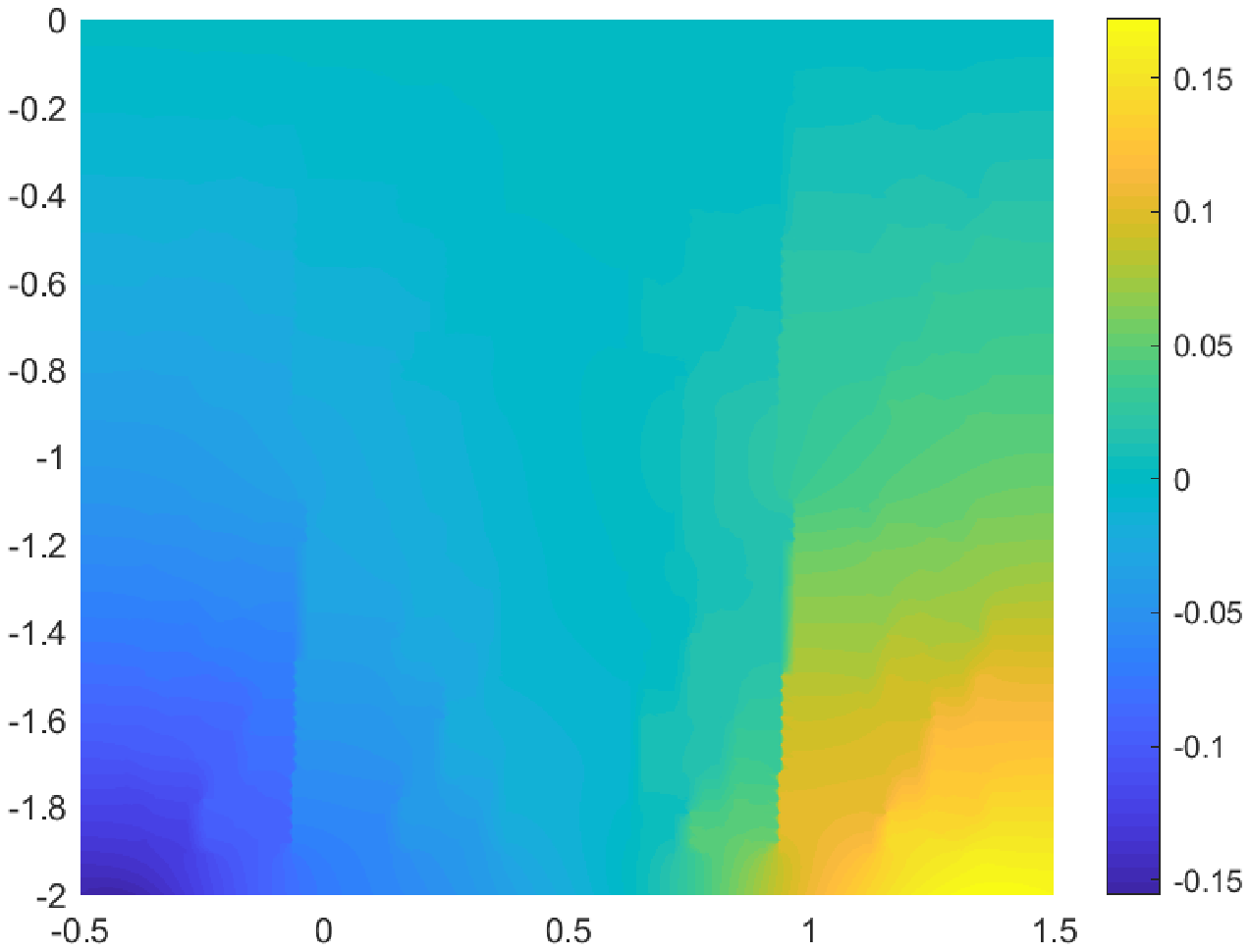}
    \end{minipage}%
  \caption{Numerical solution: $u_{S,h}^1$,$u_{S,h}^2$, $p_{S,h}$ and $p_{D,h}$ (left to right, top to bottom).}
  \label{ex5:solution2}
\end{figure}

\section{Conclusion}\label{sec:conclusion}
In this paper, we develop the first staggered DG method for coupling of the Stokes and Darcy-Forchheimer problems on general quadrilateral and polygonal meshes. Unlike the fully mixed method presented in \cite{Almonacid19}, we impose the interface condition by switching the variable met on the interface. By doing so, no additional variables need to be introduced. In addition, thanks to the special inclusion of the interface conditions, the proposed method allows nonmatching grids across the interface, which is highly appreciated from a practical point of view. By employing the discrete trace inequality and the generalized Poincar\'{e}-Friedrichs inequality, the optimal convergence estimates can also be proved. Finally, several numerical experiments are demonstrated, which indicate that our method is efficient and accurate, in addition, it is highly flexible by allowing highly distorted grids without losing the accuracy and convergence order. Our method is the first one for the coupled Stokes and Darcy-Forchheimer problem that can be flexibly applied to rough grids, and the flexibility of the method makes it a viable option for the practical applications.

\section*{Acknowledgements}

The research of Eric Chung is partially supported by the Hong Kong RGC General Research Fund (Project numbers 14304217 and 14302018), CUHK Faculty of Science Direct Grant 2018-19 and NSFC/RGC Joint Research Scheme (Project number HKUST620/15). \Red{The research of Guanyu Zhou is supported by JSPS KAKENHI 18K13460 and JSPS A3 Foresight Program.}

\end{document}